\documentclass[a4paper]{amsart}

\usepackage{amsfonts,amsmath,amssymb,amsthm}

\newtheorem{thm}{Theorem}[section]
\newtheorem{cor}[thm]{Corollary}
\newtheorem{lem}[thm]{Lemma}

\newtheorem{prop}[thm]{Proposition}
\newtheorem{obs}[thm]{Remark}

\theoremstyle{remark}

\theoremstyle{definition}
\newtheorem{defn}[thm]{Definition}
\newtheorem{nota}[thm]{Notation}

\newcommand{\bR}{{\mathbb{R}}}
\newcommand{\bP}{{\mathbb{P}}}
\newcommand{\bZ}{{\mathbb{Z}}}
\newcommand{\bA}{{\mathbb{A}}}

\newcommand{\bK}{{\mathbb{K}}}
\newcommand{\bC}{{\mathbb{C}}}
\newcommand{\bN}{{\mathbb{N}}}
\newcommand{\bG}{{\mathbb{G}}}

\newcommand{\cM}{{\mathcal{M}}}
\newcommand{\cO}{{\mathcal{O}}}
\newcommand{\cL}{{\mathcal{L}}}

\newcommand{\cT}{{\mathcal{T}}}
\newcommand{\cB}{{\mathcal{B}}}
\newcommand{\cA}{{\mathcal{A}}}
\newcommand{\cP}{{\mathcal{P}}}
\newcommand{\cI}{{\mathcal{I}}}

\newcommand{\cU}{{\mathcal{U}}}

\newcommand{\ta}{\tilde{a}}
\newcommand{\tb}{\tilde{b}}
\newcommand{\me}{\varepsilon}
\newcommand{\mo}{\vartheta}
\newcommand{\id}{1}

\usepackage{mathrsfs}

\DeclareMathOperator{\red}{r}

\DeclareMathOperator{\Div}{Div}
\DeclareMathOperator{\PGL}{PGL}
\DeclareMathOperator{\Prin}{Prin}
\DeclareMathOperator{\Star}{Star}
\DeclareMathOperator{\Supp}{Supp}
\DeclareMathOperator{\rank}{rank}
\DeclareMathOperator{\Proj}{Proj}
\DeclareMathOperator{\HC}{HC}
\DeclareMathOperator{\Ker}{Ker}

\usepackage[all]{xy}

\newcommand{\ds}[1]{\displaystyle{#1}}

\newcommand{\mc}[1]{\mathcal{#1}}
\newcommand{\mb}[1]{\mathbb{#1}}
\newcommand{\lra}{\longrightarrow}
\newcommand{\al}{\alpha}

\usepackage{mathtools}
\newcommand{\mint}{\mathop{\mathrlap{\pushMI}}\!\int}
\newcommand{\pushMI}{\mathchoice
  {\mkern2.5mu \times}
  {\scriptstyle \times}
  {\scriptscriptstyle \times}
  {\scriptscriptstyle \times}
}

\usepackage{hyperref}

\author[Iago Gin\'{e} V\'{a}zquez]{Iago Gin\'{e} V\'{a}zquez}
\address{Departament de Matem\`atiques\\Universitat Aut\`onoma de
Barcelona\\08193 Bellaterra, Barcelona, Spain}
\email{iagogv@mat.uab.cat}

\author[Xavier Xarles]{Xavier Xarles}
\address{Departament de Matem\`atiques\\Universitat Aut\`onoma de
Barcelona\\08193 Bellaterra, Barcelona, Spain}
\email{xarles@mat.uab.cat}

\title{The Abel Jacobi map for Mumford Curves via Integration}

\begin{document}

\begin{abstract} We construct the Abel-Jacobi map for Mumford curves over any complete non-archimedean field,
using multiplicative integrals and in the setting of Berkovich
analytic geometry. Along the way, we proof some results concerning
graphs and mesures related to tropical geometry.
\end{abstract}


\thanks{Both authors partially supported by Generalitat de Catalunya grant 2014-SGR-206 and Spanish MINECO grant MTM2013-40680-P.
First author partially supported by Spanish MECD grant FPU AP2010-5558}

\maketitle
\tableofcontents

\section*{Introduction}

Let $K$ be a complete field with respect to a non-trivial non-archimedean absolute value $|\cdot|:=|\cdot|_K$.
Mumford built in 1972 some algebraic curves associated to certain subgroups of the linear
group $\PGL_2(K)$ analogous to a construction of Schottky over the complex numbers.
He restricted to the case of discrete absolute value and used the geometry given by formal schemes.

This was generalized to every non-archimedean absolute
value by Gerritzen and van der Put in \cite{GvdP80} in 1980. They
named such curves Mumford curves. Shortly after Mumford's paper,
Drinfeld and Manin in \cite{MD73} showed that the Jacobian of a Mumford
curve is isomorphic to an analytic torus and that it can be built with some theta
functions, in the case $K$ is a finite extension of the $p$-adic
numbers. This construction was done also in the general case by
Gerritzen and van der Put in \cite{GvdP80}. Both took advantage of rigid analytic geometry, introduced by Tate some years ago.

More recently, Dasgupta showed in his thesis (\cite{Das04}) an
equivalent construction of the Jacobian to the ones cited above, but restricted to
the local case, by means of multiplicative integrals, defined
previously by Darmon in \cite{Dar01} and generalized by Longhi in
\cite{Lon02}.

Before that, in 1990 Berkovich introduced an alternative analytic
theory to the one of Tate in his seminal book \cite{Ber90}. The
biggest difference over a variety consists in introducing more
points instead of removing Zariski open sets. This does not stop
from getting equivalent categories of ``good'' enough analytic
varieties which can be seen as generic fibres of formal schemes,
thanks to works of Raynaud, Bosch and L\"{u}tkebohmert.

Concurrently, tropical geometry was developed and found in big
relation with Berkovich analytic geometry.

In this paper, we give a new construction of the Jacobian of a
Mumford curve over any complete non-archimedean field, departing
from Berkovich geometry, which we believe we get clarify the
constructions previously done.

It should be also recognized a great parallelism of this work with
part of the paper by van der Put \cite{vdP92}. Some of the results
are directly related to results by Baker and Rabinoff appeared in
\cite{BR15} in slightly different language.

In order to get the asserted goal, we make the basic  constructions
given by Berkovich theory in sections~\ref{trees}~and~\ref{ret},
from which later in section~\ref{schottky} we build our Mumford
curve. They are the Berkovich projective line together
$({\bP^1_K}^*)^{an}$ with its skeleton $\cT_K$, which coincides with
the Bruhat-Tits building of $\PGL_2(K)$, the locally finite subtree
$\cT_K(\cL)$ associated to a compact set $\cL$ and the retraction
map
$$
\red_\cL:({\bP^1_K}^*)^{an}\longrightarrow\overline{\cT_K(\cL)}.
$$
In sections~\ref{MI}~and~\ref{poisson} we develope the theory of multiplicative integrals and analytic functions that we need -completed later through sections~\ref{APS}~and~\ref{automorphic}. Essentially, we define these integrals, we build the ones in which we are interested and we relate them to analytic functions through the Poisson formula and the map
$$
\tilde{\mu}:\cO(\Omega_\cL)^*:\longrightarrow\cM(\cL^*,\bZ)_0
$$
Later we study the automorphic forms for a Schottky group $\Gamma\subset\PGL_2(K)$. Meanwhile, we show some essential results on metric graphs, which can be thought as tropical curves, along sections~\ref{graphs}~and~\ref{tropical}. Mainly, we prove the isomorphism between harmonic measures and harmonic cochains and the isomorphism between the $\Gamma$-invariant harmonic measures and the abelianized of $\Gamma$. Finally, the last part of this work gather all previous topics to build the desired Abel-Jacobi map.

\begin{nota} The ring of integers will be denoted $\cO_K=\{x\in K \ | \ |x|\le 1\}$ and its maximal ideal $\mathfrak{m}_K=\{x\in K \ | \ |x|< 1\}$. We will denote also $k:=\cO_K/\mathfrak{m}_K$ its residue field.

We denote by $\log$ the natural logarithm. If the absolute value $|\cdot|$ is discretely valued, we will assume that $-\log|x|\in \bZ$ for any $x\in K^*$ and that is the discrete valuation $\mathit{v}_K$ associated to $|\cdot|$. Otherwise, we define the valuation of $x$ by $\mathit{v}_K(x):=-\log|x|$.

By a complete extension $L|K$ we will refer to a field $L$ containing $K$ complete with respect to an absolute value $|\ |_L$ which extends $|\cdot |$.

The dual projective line ${\bP^1_K}^*$ over $K$ is equal to the projective spectrum of the polynomial ring $K[X_0,X_1]$, after defining ${\bP^1_K}^*=\bP(V^*)=\Proj(S^\bullet V)$ and identifying $V:=KX_0\oplus KX_1\cong K^2$.
The $K$-rational points ${\bP^1}^*(K)$ correspond to $(K^2)^*\setminus\{(0,0)\}$ modulo homothety. We denote the class of $(x_0,x_1)$ by $[x_0:x_1]$.

The infinite point in the dual projective line will be $\infty=[0:1]$ and we embed $K$ in ${\bP^1_K}^*(K)$ by means of $i^*(z)=[1:-z]$.
Therefore, an $f\in K[X_0,X_1]$ defines a function $K\longrightarrow K$ that by abuse of notation we also denote $f$, by $f(z):=f(1,-z)$.

On the other hand we inject $K$ in ${\bP^1_K}(K)$ by $i(z)=[z:1]$, taking as infinity of the projective line $[1:0]$.

We will consider the usual left action of the projective linear group $\PGL_2(K)$ on ${\bP^1_K}^*(K)$ by the contragredient representation, that is $\gamma\cdot\omega:=\omega\gamma^{-1}$ for all $\gamma\in \PGL_2(K),\ \omega\in{\bP^1_K}^*(K)$ (and also the usual left action on ${\bP^1_K}(K)$).

Given a point $p=[a:b]\in{\bP^1_K}^*(K)$, we will denote its corresponding point $p^*=[-b:a]\in{\bP^1_K}(K)$ (or if $p\in{\bP^1_K}(K)$, then $p^*\in{\bP^1_K}^*(K)$). Note that this implies $i^*(z)=i(z)^*$ for all $z\in K$ and $(\gamma\cdot p)^*=\gamma\cdot p^*$ for all $p\in{\bP^1_K}(K)$ (or $p\in{\bP^1_K}^*(K)$), $\gamma\in \PGL_2(K)$.

For any $x\in K$, $r\in\bR_{\geq0}$, we consider the ball in the completion $\bC_K:=\widehat{\overline{K}}$ of the algebraic closure of $K$, $B(x,r):=\{ y\in \bC_K \ | \ |y-x|\le r\}$.

\end{nota}




\section{Trees and Skeletons}\label{trees}

The main objective of  this section is the construction of a metric
tree associated to an arbitrary compact set
$\cL\subset{\bP^1}^*(K)$, study its structure and define the open
sets associated to its edges. This subtree generalizes to a
non-discrete setting the one defined by Mumford in \cite{Mum72} and
gives an alternative and more complete definition to the one given
in \cite[Ch.~1]{GvdP80}. In order to do it we recall some well known
notions coming from Berkovich analytic geometry and Bruhat-Tits
theory. This first part is mainly extracted from \cite{Bak08}, but
it is also greatly indebted to \cite{Wer04}, where some ideas we
recall here and along the second section are shown.

Consider the Berkovich analytic projective line $({\bP^1_K}^*)^{an}$
defined over $K$, which is the set of all the
multiplicative seminorms of the polynomial ring $K[X_0,X_1]$
extending $|\ |$ on $K$ modulo an equivalence relation which is specified below; that is,
the maps
$$\alpha:K[X_0,X_1] \to \bR_{\ge 0}$$
such that
\begin{enumerate}
\item $\alpha_{|K}=|\ |$.
\item $\alpha(X_0K+X_1K)\ne \{0\}$.
\item $\alpha(f\cdot g)=\alpha(f)\cdot \alpha(g)$
\item $\alpha(f+ g)\le \max\{\alpha(f), \alpha(g)\}$
\end{enumerate}
with $\alpha\sim \beta$ if there exists a constant $C\in\bR_{>0}$
such that $\alpha(f)=C^d\beta(f)$ for all $f\in K[X_0,X_1]$
homogeneous of degree $d$ and for all $d\ge 0$.

We associate to an $x\in {\bP^1_K}^*(K)$, $x\ne \infty=[0:1]$ and an
$r\in \bR_{\ge 0}$ an element $\alpha(x,r) \in ({\bP^1_K}^*)^{an}$ by
defining
$$
\alpha(x,r)(f)=\sup\{|f(y)|\ : \ y \in B(x,r) \}\text{ for }f\in K[X_0,X_1]
$$
and $\alpha(\infty,0)(f):=|f(0,1)|$.

We will call these seminorms the ones associated to the balls (or to
$K$-rational points if $r=0$).

\begin{obs}\label{frem}
Take $q\in K$ and assume that $f=qX_0+1X_1$ which by abuse of
notation we denote by $q$. Then $\alpha(x,r)(q)=\max\{|q-x|,r\}$. In
order to show this we compute
$$
\alpha(x,r)(q)=\sup\{|q\cdot1+1\cdot(-y)|:\ y\in B(x,r)\}=\sup\{|q-y|:\ |x-y|\leq r\}
$$
If $|q-x|<r$, then $B(x,r)=B(q,r)$ and there exists a sequence $(y_n)_n$ inside $B(x,r)$
such that $\displaystyle{\lim_n|q-y_n|=r}$, so $\alpha(x,r)(q)=r$.\\
If $|q-x|>r$, then $|q-y|=\max\{|q-x|,|x-y|\}=|q-x|$, since $|x-y|\leq r$.\\
If $|q-x|=r$, then $\alpha(x,r)(q)=\sup\{|q-y|:\ |x-y|\leq r\}=r$, since $|q-y|\leq r$.
\end{obs}

\begin{defn}
We call maximal skeleton of $({\bP^1_K}^*)^{an}$ and denote  $\cT_K$
the set of points associated to balls with $r>0$, and the
compactified skeleton $ \overline{\cT_K}$ is the skeleton together
with the (points associated to) rational points ${\bP^1}^*(K)$. It
is well known that this set is a topological space, and together
with a natural metric, which we will recall in the following, forms
a metric tree (\cite{BPR13}).
\end{defn}

\begin{obs}\label{types}
If $K$ is algebraically closed, then it is well-known (look at
\cite{Ber90}) that the points in $({\bP^1_K}^*)^{an}$ can be divided
in four types, the type I being associated to $K$-rational points,
types II and III associated to (closed) balls with center some $x\in
{\bP^1}^*(K)$, and with radius $r \in |K^*|$ or $r \in
\bR_{>0}\setminus |K^*|$ respectively, and a fourth type associated
to sequences of nesting balls with empty intersection. Then the
topological space $({\bP^1_K}^*)^{an}$ has the structure of a metric
tree. The maximal skeleton $\cT_K$ of $({\bP^1_K}^*)^{an}$ is the
set of points of type II and III, which is a metric subtree, and $
\overline{\cT_K}$ is the set of points of type I, II and III.

Recall that in \cite{BPR13} is defined a skeleton in
$({\bP^1_K}^*)^{an}$ and corollary 5.56. asserts that $\cT_K$ is the
inductive limit of all their skeleta. Note also that
$({\bP^1_K}^*)^{an}$ is homeomorphic to the inverse limit of the set
of all skeleta with respect to the natural retraction maps
(\cite[Thm.~5.57.]{BPR13}).
\end{obs}

Given any two distinct points $x_0$ and $x_1
\in{\bP^1}^*(K)\setminus\{\infty\}$, if $R=|x_0-x_1|$, we will
denote by ${x_0\vee x_1:=\alpha(x_0,R)=\alpha(x_1,R)}$. For any two
points $\alpha_0=\alpha(x_0,r_0)$ and $\alpha_1=\alpha(x_1,r_1)\in
\cT_K$, either the corresponding balls $B(x_0,r_0)\cap B(x_1,r_1)\ne
\emptyset$, in which case $\alpha(x_i,r_i)=\alpha(y,r_i)$ for all
$y\in B(x_0,r_0)\cap B(x_1,r_1)$ and $i=0,\ 1$, and we denote
$\alpha_0\vee\alpha_1:=\alpha(y,\max(r_0,r_1))$, or $B(x_0,r_0)\cap
B(x_1,r_1)= \emptyset$ and we denote $\alpha_0\vee\alpha_1:=x_0\vee
x_1$.

Let us consider two points $\alpha=\alpha(x,r),\
\alpha'=\alpha(x,r')$ of the tree $\overline{\cT_K}$, with $0\leq
r\leq r'$ and $x\neq\infty$. We denote the (oriented) path from
$\alpha$ to $\alpha'$ as $P(\alpha,\alpha')$, being as a set of
points ${\{\alpha(x,s)| r\leq s\leq
r'\}\cong[r,r']\subset\bR_{\geq0}}$. The (oriented) path
$P(\alpha',\alpha)$ from $\alpha'$ to $\alpha$ is the same set of
points oriented with the opposite direction. Finally, the (oriented)
path $P(\alpha(x,r),\alpha(\infty,0))$ from $\alpha(x,r)$ to
$\alpha(\infty,0)$ is the set of points ${\{\alpha(x,s)|s\geq
r\}\bigcup\{\alpha(\infty,0)\}\cong[r,\infty]\subset\bR_{\geq0}\bigcup\{\infty\}}$
with the orientation given by the isomorphism (as above), and we
define similarly the opposite path $P(\alpha(\infty,0),\alpha(x,r))$
reversing the orientation. Given two arbitrary points $\alpha,\
\alpha'\in\overline{\cT_K}\setminus\{\alpha(\infty,0)\}$, the
(oriented) path $P(\alpha,\alpha')$ from $\alpha$ to $\alpha'$ is
the path $P(\alpha,\alpha\vee\alpha')$ followed by the path
$P(\alpha\vee\alpha',\alpha')$.

Recall that given any two distinct points $x_0$ and $x_1 \in
{\bP^1}^*(K)$, there is a unique line in $\cT_K$ going from $x_0$ to
$x_1$, being the open path $\mathring{P}(\alpha(x_0,0),\alpha(x_1,0))$
 -the interior of the path $P(\alpha(x_0,0),\alpha(x_1,0))$. 
This line is homeomorphic as a metric tree to $\bR$, and we denote
it by $\bA_{\{x_0,x_1\}}$: it is called an apartment of the skeleton
$\cT_K$. Its closure is, by definition,
$\ds{\overline{\bA_{\{x_0,x_1\}}}=\bA_{\{x_0,x_1\}}\cup\{x_0,x_1\}}$.

Given two points $\alpha_0=\alpha(x_0,r_0)$ and
$\alpha_1=\alpha(x_0,r_1)\in\bA_{\{x_0,\infty\}}$, we define
$\displaystyle{d(\alpha_0,\alpha_1)=\left|\log \frac{r_1}{r_0}\right|}$; and in general we define
$$d(\alpha_0,\alpha_1):=d(\alpha_0,\alpha_0\vee
\alpha_1)+d(\alpha_0\vee\alpha_1,\alpha_1).$$ Then $d$ determines a
well defined metric on $\cT_K$.

A seminorm on $V$ is $\alpha:V=X_0K+X_1K\lra\bR_{\ge 0}$ satisfying
(2) and (4) as above and $\alpha(\lambda v)=|\lambda|\alpha(v)$ for
$\lambda\in K,\ v\in V$. We say that a seminorm $\alpha$ on $V$ is
diagonalizable if there exists a basis $v_0,v_1$ of $V$ such that
$\alpha(v)=\max\{|\omega_0(v)|\alpha(v_0),|\omega_1(v)|\alpha(v_1)\}$
for all $v\in V$, where $\omega_0,\omega_1$ is the dual basis of
$v_0,v_1$. We denote that seminorm as
$\alpha_{(v_0,v_1),(\rho_0,\rho_1)}$ with $\rho_0:=\alpha(v_0)$ and
$\rho_1:=\alpha(v_1)$.

\begin{obs}[The action of $\PGL_2(K)$ on $ ({\bP^1_K}^*)^{an}$]\label{ActS}

The left action of $\PGL_2(K)$ on $V$ induces a left action on
${K[X_0,X_1]\cong S^\bullet V}$.  Then, it also induces a left
action on $({\bP^1_K}^*)^{an}$ by defining
${(\gamma\cdot\alpha)(f):=\alpha(\gamma^{-1}\cdot f)}$.

For any $\gamma\in \PGL_2(K)$ we get
$\gamma\cdot\alpha(x,0)=\alpha(\gamma\cdot x,0)$, making the
injection ${\bP^1}^*(K)\lra ({\bP^1_K}^*)^{an}$ defined by
$x\mapsto\alpha(x,0)$ equivariant. We also have
$$\gamma\cdot\alpha_{(v_0,v_1),(\rho_0,\rho_1)}=\alpha_{(\gamma\cdot v_0,\gamma\cdot v_1),(\rho_0,\rho_1)}.$$
\end{obs}

Next we are going to identify $\cT_K$ with the Bruhat-Tits tree of $\PGL_2(K)$.

\begin{prop}
The seminorm $\alpha(x,r)$ restricted to $V$ is the seminorm $\alpha:=\alpha_{(v_0,v_1),(\rho_0,\rho_1)}$, diagonalizable with respect to the basis $v_0=(1,0),\ v_1=(x,1)$ and such that $\rho_0=1$ and $\rho_1=r$ when $x\neq\infty$, and $\alpha(\infty,0)=\alpha_{((1,0),(0,1)),(0,1)}$.
\end{prop}
\begin{proof}
The identification works restricting any seminorm in
$\overline{\cT_K}$ to $KX_0+KX_1$, by means of its
identification with $K^2$.
When the seminorm is $\alpha(x,r)$ for
$x\in K\subset{\bP^1}^*(K)$ and $r\ge 0$, and we apply it to a vector $v=(a,b)=(a-bx)v_0+bv_1$, we have
$$
\alpha(x,r)(v)=\sup\{|a+b(-y)|:y\in B(x,r)\}=\sup\{|a-bx+b(x-y)|: y\in B(x,r)\}
$$
$$
=\sup_{y\in B(x,r)}\{|a-bx|,|b(x-y)|\}=\max\{|a-bx|,|b|r\}=\alpha(v)
$$
Observe that $\omega_0(a,b)=a-bx$ and also that the seminorm on $K^2$ associated to a rational point $x$ has $x^*$ as its kernel, that is to say, the set of vectors $w\in K^2$ with $|\omega_0(w)|=0$ is the subspace generated by $(x,1)$.

In the case of $\alpha(\infty,0)$ we have
$$
\alpha(\infty,0)(v)=|b|=\max\{|a|0,|b|1\}=\alpha_{((1,0),(0,1)),(0,1)}(v)
$$
\end{proof}

In the following result we will specify how the correspondence between classes of seminorms with form $\alpha(x,r)$ and diagonalizable seminorms on $V$ works.

\begin{lem}
Let $v_0,v_1$ be a basis of $V$, $\omega_0,\omega_1\in V^*$ be its dual basis, $y_0=[\omega_0],y_1=[\omega_1]\in{\bP^1}^*(K)$ and $\rho_0,\rho_1\in\bR_{\geq0}$. We suppose that $y_0,y_1\neq\infty$ (look at proposition above for the case in which one point is $\infty$), and then we may take $\omega_i=(1,-y_i)$ for $i=1,2$ (by means of $i^*$).\\
With these hypotheses we get:
$$
\mbox{If }\rho_1<\rho_0,\qquad[\alpha_{(v_0,v_1),(\rho_0,\rho_1)}]=[\alpha_{(v_0,v_1),(1,\frac{\rho_1}{\rho_0})}]
$$
and
$$
\alpha_{(v_0,v_1),(1,\frac{\rho_1}{\rho_0})}=\alpha\left(y_0,\frac{\rho_1}{\rho_0}|y_0-y_1|\right)
$$
$$
\mbox{If }\rho_0<\rho_1,\qquad[\alpha_{(v_0,v_1),(\rho_0,\rho_1)}]=[\alpha_{(v_0,v_1),(\frac{\rho_0}{\rho_1},1)}]
$$
and
$$
\alpha_{(v_0,v_1),(\frac{\rho_0}{\rho_1},1)}=\alpha\left(y_1,\frac{\rho_0}{\rho_1}|y_0-y_1|\right)
$$
$$
\mbox{If }\rho_1=\rho_0,\qquad[\alpha_{(v_0,v_1),(\rho_0,\rho_1)}]=[\alpha_{(v_0,v_1),(1,1)}]
$$
and
$$
\alpha_{(v_0,v_1),(1,1)}=\alpha\left(y_0,|y_0-y_1|\right)=\alpha\left(y_1,|y_0-y_1|\right)
$$
Reciprocally, and for $r\leq|y_0-y_1|$
$$
\alpha(y_0,r)=\alpha_{(v_0,v_1),(1,\frac{r}{|y_0-y_1|})}
$$
$$
\alpha(y_1,r)=\alpha_{(v_0,v_1),(\frac{r}{|y_0-y_1|},1)}.
$$
\end{lem}

\begin{proof}
Assume, just for simplicity, that $\rho_0,\rho_1\neq0$, meaning that $\alpha$ is a norm. Define $\alpha:=\alpha_{(v_0,v_1),(\rho_0,\rho_1)}$.

Next, we start at the end. By definition $\alpha\in\bA_{\{y_0,y_1\}}$, so $\alpha\in P(y_0,y_0\vee y_1)$ or $\alpha\in P(y_0\vee y_1,y_1)$; for some $r\leq|y_0-y_1|$, in the first case we would get $\alpha=\alpha(y_0,r)$ and in the  second we would $\alpha=\alpha(y_1,r)$ up to homothety. Without loss of generality we suppose the first case. Let us take an arbitrary vector $v=(a,b)\in V$. We have
$$
\alpha(y_0,r)(v)=\max\{|a-by_0|,|b|r\}
$$
$$
\alpha(a,b)=\max\{|a-by_0|\rho_0,|a-by_1|\rho_1\}\sim\max\{|a-by_0|,|a-by_1|\frac{\rho_1}{\rho_0}\}
$$
We note that if $|a-by_0|<|b|r$, we have $[\alpha](v)=[\alpha(y_0,r)](v)$ if and only if $\displaystyle{|b|r=|a-by_1|\frac{\rho_1}{\rho_0}}$, or also $\displaystyle{\frac{\rho_1}{\rho_0}=\frac{|b|r}{|a-by_1|}}$.

But since we have $|b||y_0-y_1|\geq|b|r>|a-by_0|$, then we get $|a-by_1|=|a-by_0+b(y_0-y_1)|=\max\{|a-by_0|,|b||y_0-y_1|\}=|b||y_0-y_1|$, so
$$
\frac{\rho_1}{\rho_0}=\frac{r}{|y_0-y_1|}
$$
Therefore we obtain
$$
\displaystyle{[\alpha]=\left[\alpha\left(y_0,\frac{\rho_1}{\rho_0}|y_0-y_1|\right)\right]}
$$
after assuming $r\leq|y_0-y_1|$, that is $\rho_1\leq\rho_0$. In the same way, when $\rho_1\geq\rho_0$ we get
$$
\displaystyle{[\alpha]=\left[\alpha\left(y_1,\frac{\rho_0}{\rho_1}|y_0-y_1|\right)\right]}.
$$

Note that we see the extreme cases too, that is, when $\rho_1=0$ then $[\alpha]=[\alpha(y_0,0)]$, and when $\rho_0=0$, $[\alpha]=[\alpha(y_1,0)]$.
\end{proof}

We keep together the last two results in the next:

\begin{cor}\label{IdentificationBT}
The maximal and the compactified skeletons $\cT_K$ and $\overline{\cT_K}$ can be canonically identified with the
set of  classes
modulo homothety of nontrivial diagonalizable norms and seminorms on
$K^2$ respectively. These are the Bruhat-Tits tree of $\PGL_2(K)$ and its compactification.
\end{cor}
\begin{proof}
The classes of seminorms associated to balls correspond to the classes of diagonalizable norms and seminorms on $K^2$ by the two previous results.
\end{proof}

And now we are going to show that $d$ is invariant with respect to the action of $\PGL_2(K)$.

Consider any apartment $\bA_{\{x_0,x_1\}}$ for $x_0,x_1\in{\bP^1}^*(K)$ and choose representatives $\omega_0,\omega_1\in\ V^*$  respectively. Let $v_0,v_1\in V$ be the dual basis of $\omega_0,\omega_1$. For any two elements in this apartment $\alpha:=\alpha_{(v_0,v_1),(\rho_0,\rho_1)},\alpha':=\alpha_{(v_0,v_1),(\rho'_0,\rho'_1)}$ we define a distance in this apartment as:
$$
d_{x_0,x_1}(\alpha,\alpha'):=\left|\log\left(\frac{\rho_1\rho'_0}{\rho_0\rho'_1}\right)\right|=\left|\log\left(\frac{\rho_1}{\rho_0}\right)-\log\left(\frac{\rho'_1}{\rho'_0}\right)\right|
$$
Note that the homeomorphism (up to orientation) $\bA_{\{x_0,x_1\}}\lra\bR$ is given by $\displaystyle{\alpha\mapsto\log\left(\frac{\rho_1}{\rho_0}\right)}$, so $d_{x_0,x_1}$ is the transported distance from the natural one in $\bR$.

\begin{lem}
The two definitions of distance coincide, that is, for any $x_0,x_1\in{\bP^1}^*(K)$ we have
$$
d_{|\bA_{\{x_0,x_1\}}}=d_{x_0,x_1}
$$
\end{lem}
\begin{proof}
For any
$\alpha:=\alpha_{(v_0,v_1),(\rho_0,\rho_1)},\alpha':=\alpha_{(v_0,v_1),(\rho'_0,\rho'_1)}\in\bA_{\{x_0,x_1\}}$
we want to see $d(\alpha,\alpha')=d_{x_0,x_1}(\alpha,\alpha')$.

First, we may assume that there exists an $x\in{\bP^1}^*(K)$ such
that $\alpha,\alpha'\in\bA_{\{x,\infty\}}$. Otherwise
$d(\alpha,\alpha')=d(\alpha,\alpha\vee\alpha')+d(\alpha\vee\alpha',\alpha')$
and by definition $d_{x_0,x_1}$ satisfies the same equality.

Moreover, it is enough to prove that if
$\alpha,\alpha'\in\bA_{\{x_0,x_1\}}\bigcap\bA_{\{y_0,y_1\}}$ then
$d_{x_0,x_1}(\alpha,\alpha')=d_{y_0,y_1}(\alpha,\alpha')$, since for
the particular case $y_0=x,\ y_1=\infty$ we have $d_{x,\infty}=d$.

We may reduce to the case $y_0=x_0$ by applying the result in two
steps. Let us denote $x_2:=y_1\in{\bP^1}^*(K)$ and let it be
represented by $\omega_2=\lambda\omega_0+\mu\omega_1\in V^*,\
\mu\neq0$. Then
$\displaystyle{u_0=v_0-\frac{\lambda}{\mu}v_1,u_2=\frac{v_1}{\mu}\in
V}$ is the dual basis of $\omega_0,\omega_2$. Now we have that
$$
\alpha:=\alpha_{(v_0,v_1),(\rho_0,\rho_1)}=\alpha_{(u_0,u_2),(\eta_0,\eta_2)}\mbox{ with }\eta_0=\max\left\{\rho_0,\left|\frac{\lambda}{\mu}\right|\rho_1\right\},\ \eta_2=\left|\frac{1}{\mu}\right|\rho_1
$$
and
$$
\alpha':=\alpha_{(v_0,v_1),(\rho'_0,\rho'_1)}=\alpha_{(u_0,u_2),(\eta'_0,\eta'_2)}\mbox{ with }\eta'_0=\max\left\{\rho'_0,\left|\frac{\lambda}{\mu}\right|\rho'_1\right\},\ \eta'_2=\left|\frac{1}{\mu}\right|\rho'_1
$$
Note that $\rho_1=|\mu|\eta_2$ implies that $\eta_0=\max\{\rho_0,|\lambda|\eta_2\}$. Furthermore $\rho_0=\alpha(v_0)=\max\{\eta_0,|\lambda|\eta_2\}$, since $v_0=u_0+\lambda u_2$. Then $\eta_0=\rho_0$. Identically we get $\eta'_0=\rho'_0$.

Therefore
$$
d_{x_0,x_2}(\alpha,\alpha')=\left|\log\frac{\eta_2\eta'_0}{\eta_0\eta'_2}\right|=\left|\log\frac{\rho_1\rho'_0}{\rho_0\rho'_1}\right|=d_{x_0,x_1}(\alpha,\alpha')
$$
\end{proof}

\begin{cor}
The distance is $\PGL_2(K)$-invariant, that is to say, $d(\alpha,\alpha')=d(\gamma\cdot\alpha,\gamma\cdot\alpha')$ for any $\gamma\in \PGL_2(K)$.
\end{cor}
\begin{proof}
First we recall that $\gamma\cdot\alpha_{(v_0,v_1),(\rho_0,\rho_1)}=\alpha_{(\gamma\cdot v_0,\gamma\cdot v_1),(\rho_0,\rho_1)}$. Let us to take now any apartment $\bA_{\{x_0,x_1\}}$ which contains $\alpha,\alpha'$ as above. Then
$d(\alpha,\alpha')=d_{x_0,x_1}(\alpha,\alpha')=d_{\gamma\cdot x_0,\gamma\cdot x_1}(\gamma\cdot\alpha,\gamma\cdot\alpha')=d(\gamma\cdot\alpha,\gamma\cdot\alpha')$, where the second equality is due to the remark~\ref{ActS}.
\end{proof}

Let $x_0$, $x_1$ and $x_2$ be three distinct points in ${\bP^1}^*(K)$. Then
there exists a unique point $t(x_0,x_1,x_2)\in \cT_K$ which is
contained in the three lines they form. If $x_2=\infty$, then
$t(x_0,x_1,\infty)=\alpha(x_0,R)=x_0\vee x_1$, where  $|x_1-x_0|=R$.
If none of them is equal to $\infty$, it corresponds to the smallest
ball containing all three points.

Observe that the points $t(x_0,x_1,x_2)$ are always of type II, so
they have the form $\alpha(x_0,r)$ with $r\in |K^*|$.

\begin{defn}\label{tree} Let $\cL$ be a subset of ${\bP^1}^*(K)$ which contain at least two points. Denote by
$$\cT_K(\cL):=\bigcup_{\{x_0,x_1\}\subset \cL} \bA_{\{x_0,x_1\}}$$
the metric tree associated to $\cL$ (which is the subspace of
$\cT_K$ generated by the lines between two points in $\cL$). Note
that $\overline{\cT_K(\cL)}:=\cT_K(\cL)\cup\cL$ with the natural
topology.
\end{defn}

It is clear that for any extension of fields $L|K$ the tree
associated to $\cL$ is always the same: $\cT_{L}(\cL)=\cT_K(\cL),\
\overline{\cT_{L}(\cL)}=\overline{\cT_K(\cL})$.

We will show in the sequel that $\cT_K(\cL)$ is a locally finite
metric tree if $\cL$ is compact.

\begin{lem} The points of the form $t(x_0,x_1,x_2)$ for three
distinct points $x_0, x_1, x_2\in
\cL$ are the points in $\cT_K(\cL)$ with valence greater than $2$.

Suppose that $\infty,x_0\in \cL$ and consider a point
$\alpha:=\alpha(x_0,r)\in \cT_K(\cL)$ of the form
$t(x_0,x_1,\infty)$ for some $x_1\in \cL$. Then
$$\{y \in \cL\setminus\{x_0,\infty\} \ | \
\alpha=t(x_0,y,\infty)\}=\{y \in \cL \ | \ |y-x_0|=r\}.$$ Moreover,
there is a bijection between the set of directions from $\alpha(x_0,r)$ except the ones which connect with $\infty$ and $x_0$, and the image
of the map
$$\psi:\{y \in \cL \ | \ |y-x_0|=r\} \to k^*$$ given by sending
$$\psi(y)=\frac{y-x_0}{x_1-x_0} \pmod{\mathfrak{m}_K}.$$
\end{lem}
\begin{proof}
The unique claim that needs a proof is the bijection. From the equality shown, we see that a direction can be identified with a set of points $E_y\subset\{y \in \cL \ | \ |y-x_0|=r\}$ such that $|y'-y''|<r$ for all $y',y''\in E_y$.
Thus, the only thing we have to prove is that $\psi(y)=\psi(y')$ if and only if $|y-y'|< r$.

To start with this equivalence we note that $\psi(y)=\psi(y')$ means that there exists $z\in\mathfrak{m}_K$, or equivalently $|z|<1$, such that
$$
\frac{y-x_0}{x_1-x_0}=\frac{y'-x_0}{x_1-x_0}+z
$$
We may write this equality as $y-y'=z(x_1-x_0)$ and taking absolute value $|y-y'|=|z|r<r$. Finally, the other option, $|z|=1$, is that for which $\psi(y)\neq\psi(y')$.
\end{proof}

\begin{cor}\label{locft} If $\cL$ is compact, then $\cT_K(\cL)$ is a locally finite
metric tree, that is to say, any vertex has a finite number of directions arriving
to it and any finite lenght path contains only a finite number of vertices of valence greater than $2$.
\end{cor}
\begin{proof} We suppose $\cL$ has at least three points and $\infty\in \cL$ without loss of generality.\\
In order to prove the first claim consider a vertex $\alpha(x_0,r)\in \cT_K(\cL)$ that we may assume of the form
$t(x_0,x_1,\infty)$ for some $x_0$ and $x_1\in \cL$. Since $\cL$ is compact and
$\{y\in K \ | \ |y-x_0|=r\}$ is closed, their intersection $\{y
\in \cL \ | \ |y-x_0|=r\}$ is compact. Now, given any
$t\in k^*$, the set $\psi^{-1}(\{t\})$ is an open subset (the previous proof shows it is an open ball).
Then, if the point had infinite directions arriving to it, the image
of $\psi$ would be infinite so the compact set $\{y \in \cL \ | \
|y-x_0|=r\}$ would be covered by an infinite number of disjoint open
subsets and we would get a contradiction.\\
To get the second claim we can reduce us to show it for a path $P(\alpha(x,r),\alpha(x,r'))$ with $0<r\leq r'$. We have to show that the set
$$
S_{r,r'}:=\{s\in[r,r']|\ \exists y\in\cL:|y-x|=s\}
$$
is finite. Consider the set
$$
\{y\in\cL|\ r\leq|y-x|\leq r'\}=\cL\bigcap\left(B(x,r')\setminus\mathring{B}(x,r)\right)=\bigcup_{s\in S_{r,r'}}\{y\in\cL|\ |y-x|=s\}
$$
Since it is a closed in $\cL$, then it is compact. Further, the subsets
$$
\cL_{x,s}:=\{y\in\cL|\ |y-x|=s\}=\bigcup_{y\in\cL_{x,s}}\left(\cL\cap\mathring{B}(y,s)\right)
$$
are open, so we can get a finite covering by them, and this implies necessarily that $S_{r,r'}$ is finite.
\end{proof}

\begin{defn}
With the hypotheses of definition~\ref{tree} we say that
$\cT_K(\cL)$ is perfect if for any $\alpha\in\cT_K(\cL)$ and for any
$r\in\mathbb{R}_{>0}$ there exists $\alpha'\in\cT_K(\cL)$ with
valence greater than 2 and such that $d(\alpha,\alpha')>r$.
\end{defn}
One can show  that this definition is compatible with the one of
perfect set, so $\cT_K(\cL)$ is perfect if and only if $\cL$ is
perfect (all the points in $\cL$ are accumulation points). For
example, if $\cL$ is a finite set, then $\cT_K(\cL)$ is not perfect,
since it has just a finite number of vertices of valence greater
than $2$.

\begin{defn}
We will call a topological (oriented) edge $\me:=\me_{\alpha,\beta}$
(of $\cT_K(\cL)$) a non trivial path
$P(\alpha,\beta)\subset\cT_K(\cL)$, such that all its interior
points have valence two in $\cT_K(\cL)$. We will call the length of
$\me$ the distance $d(\alpha,\beta)$, and we will denote it by
$l(\me)$.
\end{defn}
Given any
compact subset $\cL\subset {\bP^1}^*(K)$, we consider $\cL^*:=\{x^*\  |
\ x\in \cL\}\subset\bP^1(K)$, which is also a compact subset.

\begin{defn}\label{open}
Let $\me$ be the topological edge of $\cT_K(\cL)$ induced by the
path $P(\alpha,\beta)$. We may define a compact open set of $\cL^*$
associated to it as
$$
\cB(\me):=\cB(\alpha,\beta):=\{x\in\cL^*|\ \alpha\not\in
P(x^*,\beta)\}.
$$
\end{defn}

Note that if $\alpha=\alpha(x,r)$, $\alpha'=\alpha(x,s)$ and $x\in
K$, either $r<s$ and so $\cB(\alpha,\alpha')=\cL^*\setminus
\mathring{B}(x,s)$, or $r>s$ and then
$\cB(\alpha,\alpha')=\cL^*\bigcap B(x,s)$. The following lemma is
elementary.

\begin{lem}\label{oprops} These sets satisfy the next properties:
\begin{itemize}
\item All together are an open basis of the topology of $\cL^*$ in the strong sense, meaning that any compact open set of $\cL^*$ is a finite disjoint union
of them.
\item If $\overline{\me}$ is the opposite of a topological oriented edge $\me$, then $\cB(\overline{\me})=\cL^*\setminus \cB(\me)$.
\item If $v$ is a vertex en $\cT_K(\cL)$, the open sets $\cB(\me)$ for $\me$ topological edges with source $v$ covering the different directions from the vertex in the tree are pairwise disjoint and $\bigsqcup \cB(\me)=\cL^*$.
\item Let $\me=P(\alpha,\beta)$, $\me'=P(\alpha',\beta')$ be two topological edges of $\cT_K(\cL)$. Then
$$
\cB(\me)\bigcap\cB(\me')\left\{
\begin{array}{l}
=\cB(\me)=\cB(\me'),\mbox{ if there is another topological edge containing both}\\
\qquad\qquad\qquad\quad\mbox{ with the same orientation}\\
=\emptyset,\;\mbox{ if }\alpha,\alpha'\in P(\beta,\beta')\\
=\cB(\me)\subset\cB(\me'),\mbox{ if }\alpha\in P(\beta,\alpha')\mbox{ and }\alpha'\not\in P(\beta,\beta')\\
=\cB(\me')\subset\cB(\me),\mbox{ if }\alpha'\in P(\beta',\alpha)\mbox{ and }\alpha\not\in P(\beta',\beta)\\
\left\{\begin{array}{l}
\neq\emptyset\mbox{ and}\\
\subsetneq\cB(\me),\cB(\me')
\end{array},\right.
\begin{array}{l}\mbox{ if }P(\alpha,\alpha')\mbox{ is not a topological edge}\\
\mbox{ and }\alpha,\alpha'\not\in P(\beta,\beta')\end{array}
\end{array}\right.
$$
\end{itemize}
\end{lem}




\section{The retraction map}\label{ret}

We build the retraction map $\red_{\cL}:({\bP^1_{K}}^*)^{an} \longrightarrow \overline{\cT_K(\cL)}$ generalizing the reduction map constructed by Werner in \cite{Wer04} to the subtrees introduced in the previous section, which, on the other hand, gives the complete description over all the Berkovich analytic points of the reduction map named in \cite[2.3.]{Das05}. Further, we do not restrict to a local field.

Through this section $L|K$ will be an arbitrary extension of valued
complete fields.

Given any compact subset with at least two points $\cL\subset {\bP^1}^*(K)$ we define $\Omega_\cL(L):={\bP^1}^*(L)\setminus \cL$. We also define the diameter of $\cL$ as
$$
d_\cL=\left\{\begin{array}{l}
\inf\{r\geq0|\ \cL\subset B(x,r)\mbox{ for some }x\in\cL\}\mbox{ if }\infty\not\in\cL\\
+\infty\mbox{ if }\infty\in\cL
\end{array}\right.
$$
Note that we may fix $x\in\cL$ and the definition is independent of the chosen point $x$.

\begin{defn} Let $\cL\subset{\bP^1}^*(K)$ be as just above. We define the retraction map
$\red_{\cL}:{\bP^1}^*(L) \to \overline{\cT_K(\cL)}$ to be
$$
\red_{\cL}(x)=\left\{\begin{array}{l}
x\mbox{, if }x\in\cL\\
\alpha(x,\inf\{s\ge 0 \ | \ B(x,s)\cap \cL\ne \emptyset\})\mbox{, if }B(x,d_\cL)\cap\cL\neq\emptyset\mbox{ and }x\not\in\cL\\
\alpha(y,d_\cL)\mbox{ for any }y\in\cL\mbox{, if }B(x,d_\cL)\cap\cL=\emptyset\\
\end{array}\right.
$$
for $x\neq\infty$, and
$$
\red_{\cL}(\infty)=\left\{\begin{array}{l}
\alpha(y,d_\cL)\mbox{ for any }y\in\cL\mbox{, if }\infty\not\in\cL\\
\alpha(\infty,0)\mbox{, if }\infty\in\cL
\end{array}\right.
$$
We also define
$\red_{\cL}:\Omega_\cL(L) \to \cT_K(\cL)$ as the restriction.
\end{defn}

\begin{obs}
The retraction map leaves fixed the points of $\cL$. On the other
hand, if $x\not\in\cL$, the point $\red_\cL(x)$ is the only point of
the path $P(\alpha(x,0),\red_\cL(x))\subset\overline{\cT_K}$ which
is in $\cT_K(\cL)$.
\end{obs}

Now we want to extend this map to $\red_{\cL}:\overline{\cT_{L}}
\longrightarrow \overline{\cT_K(\cL)}$. First, if
$\alpha\in\overline{\cT_K(\cL)}$, then $\red_{\cL}(\alpha)=\alpha$.

Next, consider $\alpha\in
\cT_{L}\smallsetminus\overline{\cT_K(\cL)}$.  Then
$\alpha=\alpha(x,r)$ for some $x\in L\smallsetminus \cL$ and some
$r> 0$.

If $B(x,r)\bigcap\cL=\emptyset$, we define
$\red_{\cL}(\alpha):=\red_{\cL}(x)$. We only need to show that
$\red_{\cL}(\alpha)$ does not depend on the chosen $x$.

When $B(x,d_\cL)\cap\cL\neq\emptyset$,  $\red_{\cL}(x)=\alpha(x,s)$
and $s>r$ since $\alpha(x,r)\not\in\overline{\cT_K(\cL)}$. Hence, if
$\alpha(x,r)=\alpha(y,r)$, then $\alpha(x,s)=\alpha(y,s)$.
Otherwise, it is clear.

In the other case, $B(x,r)\bigcap\cL\neq\emptyset$, we have
$\infty\not\in\cL$ and $\cL\subset B(x,r)$ (so $r>d_\cL$). Then we
define $\red_{\cL}(\alpha):=\red_{\cL}(\infty)$.

\begin{cor}\label{retraction}
The retraction map is a retraction. As a consequence, if ${\Gamma\subset\PGL_2(K)}$ acts on $\cL$, it is $\Gamma$-equivariant.
\end{cor}
\begin{proof}
It follows from the previous remark and construction that the map is
a retraction in the strict sense. The consequence is due to the fact
that the projective linear group acts continuously on $\cT_K$ and
$\Gamma$ leaves $\cT_K(\cL)$ invariant.
\end{proof}

Next, let us recall that $\bC_K$ embeds isometrically into a
spherically complete nonarchimedean field $\bK$, since it admits a
maximally complete extension by \cite[Thm.~24]{Kru32}, and this
condition is equivalent to spherical completeness by
\cite[Thm.~4]{Kap42}. We know by \cite[\S1.4]{Ber90} that
$({\bP^1_{\bK}}^*)^{an}$ has no type IV points so we get
$$
\red_{\cL}:({\bP^1_{\bK}}^*)^{an}=\overline{\cT_{\bK}} \longrightarrow \overline{\cT_K(\cL)}
$$

Note that from the beginning of the formalization of the retraction
map, each time that we define it taking an infimum
($\red_\cL(\alpha)=\alpha(x,\inf\{...\})$) we get this element is
inside the tree $\cT_K(\cL)$ since $\cL$ is compact.

The following lemma is clear from the properties of the retraction
map.

\begin{lem}\label{extend}
If we have two subsets $\cL'\subset \cL\subset {\bP^1}^*(K)$ as above, then
$$
\red_{\cL'}(\alpha)=\red_{\cL'}(\red_{\cL}(\alpha))\text{ for any }\alpha \in \overline{\cT_K}.
$$
\end{lem}

\begin{lem}\label{red2points}
For any two points $y_0,\ y_1$ in ${\bP^1}^*(K)$, with respective
representatives in $(K^2)^*$ given by $\omega_0,\ \omega_1$ and
having dual basis $v_0,\ v_1$, and for any  $\alpha\in
\overline{\cT_{L}}$, the point $\red_{\{y_0,y_1\}}(\alpha)$ is the
seminorm $\eta$ diagonalized by $v_0$ and $v_1$ up to equivalence,
with $\eta(v_i)=\alpha(v_i)$ for $i=0$ and $1$, that is
$[\red_{\{y_0,y_1\}}(\alpha)]=[\alpha_{(v_0,v_1),(\alpha(v_0),\alpha(v_1))}]$.
\end{lem}
\begin{proof}
If $\alpha\in\bA_{\{y_0,y_1\}}$ there is nothing to prove. From now
on we assume this is not the case.

If one of the two points, let us assume $y_1$, is $\infty$, then, writing $\alpha=\alpha(x,r)$,
$$
\red_{\{y_0,\infty\}}(\alpha)=\alpha(x,|y_0-x|)=\alpha(y_0,|y_0-x|)=\alpha_{(v_0,v_1),(1,|x-y_0|)}
$$
Now we compute
$$
\alpha(x,r)(v_0)=\alpha(x,r)(1,0)=\max\{1,0\}=1
$$
$$
\alpha(x,r)(v_1)=\alpha(x,r)(y_0,1)=\max\{|y_0-x|,r\}=|x-y_0|
$$
since $|x-y_0|>r$ due to $\alpha\not\in\bA_{\{y_0,y_1\}}$.

Next, suppose $y_0,y_1\neq\infty$, and then we can take
$\omega_i=(1,-y_i)$ for $i=0,1$, so
$$\displaystyle{v_0=\left(\frac{y_1}{y_1-y_0},\frac{1}{y_1-y_0}\right) \mbox{ and } v_1=\left(\frac{y_0}{y_0-y_1},\frac{1}{y_0-y_1}\right)}.$$
Furthermore, either $\{y_0,y_1\}\subset B(x,r)$ or
$B(x,r)\bigcap\{y_0,y_1\}=\emptyset$.

In the first case
$$
\red_{\{y_0,y_1\}}(\alpha)=\red_{\{y_0,y_1\}}(\infty)=\alpha(y_0,|y_0-y_1|\})=\alpha_{(v_0,v_1),(1,1)}
$$
We just need to show that $\alpha(v_0)=\alpha(v_1)$. We have
$$
\alpha(x,r)(v_0)=\alpha(x,r)\left(\frac{y_1}{y_1-y_0},\frac{1}{y_1-y_0}\right)=\max\left\{\left|\frac{y_1-x}{y_1-y_0}\right|,\left|\frac{r}{y_1-y_0}\right|\right\}
$$
and, identically,
$\displaystyle{\alpha(x,r)(v_1)=\max\left\{\left|\frac{y_0-x}{y_0-y_1}\right|,\left|\frac{r}{y_0-y_1}\right|\right\}}$.
Since the condition $\{y_0,y_1\}\subset B(x,r)$ tells us that
$r\geq|y_0-x|,|y_1-x|$ we get the required equality
$\alpha(v_0)=\alpha(v_1)$.

In the second case, which is satisfied
$B(x,r)\bigcap\{y_0,y_1\}=\emptyset$, we have
$$
\red_{\{y_0,y_1\}}(\alpha)=\red_{\{y_0,y_1\}}(x)=
$$
$$
=\left\{\begin{array}{l}
\alpha(x,\min\{|x-y_0|,|x-y_1|\})\mbox{, if }B(x,|y_0-y_1|)\cap\{y_0,y_1\}\neq\emptyset\\
\alpha(y_0,|y_0-y_1|)\mbox{, if }B(x,|y_0-y_1|)\cap\{y_0,y_1\}=\emptyset
\end{array}\right.
$$
after noticing that $d_{\{y_0,y_1\}}=|y_0-y_1|$. The equality $B(x,|y_0-y_1|)\cap\{y_0,y_1\}=\emptyset$ is equivalent to say that
$|y_0-x|=|y_1-x|>|y_0-y_1|$ and $\alpha(y_0,|y_0-y_1|)=\alpha_{(v_0,v_1),(1,1)}$.
All the rest of the proof for this situation works exactly equal as above taking into account that the
condition $B(x,r)\bigcap\{y_0,y_1\}=\emptyset$ implies $|y_0-x|(=|y_1-x|)>r$.\\
Finally, when $B(x,|y_0-y_1|)\cap\{y_0,y_1\}\neq\emptyset$ we have $|y_i-x|\leq|y_0-y_1|$ for $i=0,1$ and at least for one $i$, $|y_i-x|=|y_0-y_1|$; assume this equality for $y_1$. Then, on one hand we get
$$
\alpha(x,\min\{|x-y_0|,|x-y_1|\})=\alpha(x,|x-y_0|)=\alpha(y_0,|x-y_0|)=\alpha_{(v_0,v_1),(1,\frac{|x-y_0|}{|y_0-y_1|})}
$$
On the other hand we have
$$
\alpha(x,r)(v_0)=\max\left\{\left|\frac{y_1-x}{y_1-y_0}\right|,\left|\frac{r}{y_1-y_0}\right|\right\}=\left|\frac{y_1-x}{y_1-y_0}\right|
$$
and
$$
\alpha(x,r)(v_1)=\max\left\{\left|\frac{y_0-x}{y_0-y_1}\right|,\left|\frac{r}{y_0-y_1}\right|\right\}=\left|\frac{y_0-x}{y_0-y_1}\right|
$$
since $B(x,r)\bigcap\{y_0,y_1\}=\emptyset$. Therefore, maintaining and employing the assumption $|y_1-x|=|y_0-y_1|\geq|y_0-x|$, we obtain
$$
\alpha_{(v_0,v_1),(\alpha(v_0),\alpha(v_1))}=\alpha_{(v_0,v_1),(\frac{|y_1-x|}{|y_1-y_0|},\frac{|y_0-x|}{|y_0-y_1|}}=\alpha_{(v_0,v_1),(1,\left|\frac{x-y_0}{y_0-y_1}\right|)},
$$
and so the claimed equality. Note that if we had assumed $|y_0-x|=|y_0-y_1|$ we would have got
$$
\alpha(x,\min\{|x-y_0|,|x-y_1|\})=\alpha(y_1,|x-y_1|)=\alpha_{(v_0,v_1),(\frac{|x-y_1|}{|y_0-y_1|},1)}=\alpha_{(v_0,v_1),(\alpha(v_0),\alpha(v_1))}
$$
too.
\end{proof}

\begin{lem}
Let $\cL\subset{\bP^1}^*(K)$ be a compact subset with at least two points. For any two seminorms $\alpha,\alpha'\in\overline{\cT_{L}}$ such that $\displaystyle{\alpha_{|K[X_0,X_1]}=\alpha'_{|K[X_0,X_1]}}$, then $\red_\cL(\alpha)=\red_\cL(\alpha')$.
\end{lem}
\begin{proof}
If $\cL=\{y_0,y_1\}$ the claim is true due to the last lemma. Otherwise, we always can find two points $y_0,y_1\in\cL$ such that $\red_\cL(\alpha),\red_\cL(\alpha')\in\bA_{\{y_0,y_1\}}$. Then, using this hypothesis for the outer equalities together with lemmas~\ref{extend}~and~\ref{red2points} for the interior equalities, we get
$$
\red_\cL(\alpha)=\red_{\{y_0,y_1\}}(\red_\cL(\alpha))=\red_{\{y_0,y_1\}}(\alpha)=\red_{\{y_0,y_1\}}(\alpha')=\red_{\{y_0,y_1\}}(\red_\cL(\alpha'))=\red_\cL(\alpha')
$$
\end{proof}

Finally, recall that we have a retraction map
$\red_{\cL}:\overline{\cT_{L}} \longrightarrow
\overline{\cT_K(\cL)}$ with two important particular cases:
$$
\red_{\cL}:\overline{\cT_{K}} \longrightarrow \overline{\cT_K(\cL)}
$$
and
$$
\red_{\cL}:({\bP^1_{\bK}}^*)^{an}=\overline{\cT_{\bK}} \longrightarrow \overline{\cT_K(\cL)}
$$
Now we want extend the first retraction map to $\red_{\cL}:({\bP^1_{K}}^*)^{an} \longrightarrow \overline{\cT_K(\cL)}$. Note first that $({\bP^1_{K}}^*)^{an}\cong({\bP^1_{\bC_K}}^*)^{an}/Gal(\bC_K|K)$ by \cite[Cor.~1.3.6]{Ber90} so we may assume for a while $K=\bC_K$ in order to define the extension.

Then, by remark~\ref{types} we only have to do this for the points of type IV. Let us take such a seminorm point $\alpha\in({\bP^1_{K}}^*)^{an}$. It is a limit of ball seminorms $\{\alpha(x_i,r_i)\}_{i\in\bN}$ such that
$$
r_{i+1}\leq r_i,\qquad B(x_{i+1},r_{i+1})\subset B(x_i,r_i)
$$
$$
r:=\lim_{i\rightarrow\infty}{r_i}>0\mbox{ and }\bigcap_{i\in\bN}{B(x_i,r_i)}=\emptyset
$$
We consider the balls of the same center and radio with points in the spherical completion $\bK$, that is $B_\bK(x_i,r_i):=\{ y\in \bK \ | \ |y-x_i|\le r_i\}$. Denote the associated seminorms in $({\bP^1_{\bK}}^*)^{an}$ by $\alpha_\bK(x_i,r_i)$.\\
Therefore, on one hand we have $\alpha_\bK(x_i,r_i)_{|K[X_0,X_1]}=\alpha(x_i,r_i)$ and on the other hand we obtain $\displaystyle{\bigcap_{i\in\bN}{B_\bK(x_i,r_i)}\neq\emptyset}$ so it is a ball $B_\bK(\hat{x},r)$ which has an associated seminorm $\alpha_\bK(\hat{x},r)\in({\bP^1_{\bK}}^*)^{an}$. Thus we get
$$
\alpha=\lim_{i\rightarrow\infty}{\alpha(x_i,r_i)}=\lim_{i\rightarrow\infty}{\alpha_\bK(x_i,r_i)_{|K[X_0,X_1]}}=\alpha_\bK(\hat{x},r)_{|K[X_0,X_1]}
$$
Finally, we may take $\red_\cL(\alpha):=\red_\cL(\alpha_\bK(\hat{x},r))$ which is well defined by last lemma above.

\begin{obs}
This construction of  $\red_{\cL}:({\bP^1_{K}}^*)^{an} \longrightarrow \overline{\cT_K(\cL)}$ and
the lemma~\ref{red2points} allows us to note that when $\ds{\overline{\cT_K(\cL)}=\overline{\cT_{K}}}$,
this definition coincides with the given by Werner in \cite{Wer04}.
\end{obs}




\section{Graphs, their models and harmonic cochains}\label{graphs}

We give a general definition of harmonic cochains over any weighted graph and we prove the isomorphism between harmonic measures on any compact subset ${\cL^*\subset\bP^1(K)}$ and harmonic cochains on the associated tree $\cT_K(\cL)$, claimed in \cite[Ex.~2.1.1]{vdP92}.

A weighted graph $\mathfrak{G}$ is a non empty set
$V=V(\mathfrak{G})$ called vertex set together with an oriented edge
set $E=E(\mathfrak{G})$, a weight function $\ell:E\lra\mb{R}_{>0}$,
an edge assignment map $s\times t:E\lra V\times V$ which makes
correspond to each edge $e$ a pair $(s(e),t(e))$, where $s(e)$ is
called the source of $e$ and $t(e)$ the target of $e$, and a
bijection $\mo:E\lra E$ verifying $\ell(\mo(e))=\ell(e)$, $s\times
t(\mo(e))=(t(e),s(e))$ and $\mo(e)\neq e$. The edge $\mo(e)$ is
called the opposite of $e$ and denoted by $\bar e$ (cf. \cite{BF11}
and \cite{Ser80}).

The topological realization of a weighted graph $\mathfrak{G}$ is a metric graph $G:=|\mathfrak{G}|$, for which the lenght of their edges is given by the weight of the edges of $\mathfrak{G}$ (the same definitions and notations that we have for a weighted graph work for a metric graph). Reciprocally, given a metric graph $G$, a model for it is any weighted graph $\mathfrak{G}$ such that $G$ is obtained as its topological realization, that is $G\cong|\mathfrak{G}|$. A minimal model is one in which all the vertices have valence greater than 2.

We will consider the free abelian group $\bZ[E(\mathfrak{G})]$ generated by the oriented edges of $\mathfrak{G}$.

Given a weighted graph $\mathfrak{G}$, and a vertex $v$, we denote
by $\Star(v)$ the set of edges of $\mathfrak{G}$ with source $v$. Recall that an harmonic cochain is a
morphism $c:\bZ[E(\mathfrak{G}))]\to \bZ$ verifying
\begin{itemize}
\item $c(\bar e)=-c(e)$ for any $e \in E(\mathfrak{G})$, and
\item $\displaystyle{c\Big(\sum_{e\in \Star(v)}
e\Big) =0}$ for any vertex $v \in V(\mathfrak{G})$.
\end{itemize}
We denote the set of harmonic cochains of $\mathfrak{G}$ by $\HC(\mathfrak{G},\bZ)$.

Observe that, if we subdivide an oriented edge $e$ in two oriented
edges $e_1$ and $e_2$, then the properties tell that any
harmonic cochain verifies that $c(e_1)=c(e_2)$. Hence, given a (locally finite) metric graph and two arbitrary models for it, there is a
canonical isomorphism between their harmonic cochains, so we can define them for the metric graph $G=|\mathfrak{G}|$, and we can write $\HC(G,\bZ):=\HC(\mathfrak{G},\bZ)$.

Let $\mathfrak{G}=(V,E)$ be a weighted graph and $\mathfrak{H}$ be a finite weighted subgraph of $\mathfrak{G}$. We define
$$
\Star(\mathfrak{H}):=\{e\in E|\ s(e)\in \mathfrak{H}, e\not\in E(\mathfrak{H})\}
$$
Note that this generalizes the definition of $\Star(v)$ for a vertex $v$.

\begin{lem}\label{harstar} Let $\mathfrak{H}$ be a finite weighted subgraph of $\mathfrak{G}$. Then, any harmonic cochain $c$ satisfies $\displaystyle{c\Big(\sum_{e\in
\Star(\mathfrak{H})} e\Big) =0}$.
\end{lem}
\begin{proof}
First observe the following properties of stars:
$$
\Star(\mathfrak{H})\sqcup E(\mathfrak{H})=\Star(V(\mathfrak{H}))=\bigsqcup_{v\in V(\mathfrak{H})}{\Star(v)}
$$
Next note that an edge belongs to $\mathfrak{H}$ if and only if its opposite also do. Then, taking into consideration the first equality of stars and the previous remark, because of the first property of the harmonic cochains we get
$$
c\Big(\sum_{e\in \Star(\mathfrak{H})} e\Big) = \sum_{e\in \Star(\mathfrak{H})} c(e) = \sum_{e\in \Star(V(\mathfrak{H}))} c(e)
$$
and because of the second equality of stars and the second property of harmonic cochains we finish as follows:
$$
\sum_{e\in \Star(V(\mathfrak{H}))} c(e) = \sum_{v\in V(\mathfrak{H})} \sum_{e\in \Star(v)} c(e) = 0
$$
\end{proof}

Let $\mathfrak{T}=(V,E)$ be a model for $\cT=\cT_K(\cL)$. For any
edge $e$ there exists a topological edge $\me$ (both oriented) such
that the first (its topological realization) is contained in the
second (as an oriented path in $\cT$), that is $|e|\subset\me$. Then
we define the open set $\cB(e):=\cB(\me)$. This is well defined due
to the lemma~\ref{oprops}.

\begin{prop}
Let $\mathfrak{T}=(V,E)$ be a model for $\cT=\cT_K(\cL)$ and let $F\subset E$ be a well oriented finite set of edges, meaning that it satisfies the following hypothesis:
\begin{itemize}
\item it cannot exist a topological edge $\me$ of $\cT$ and edges $e,e'\in F$ such that both are contained in $\me$, $e$ is oriented like $\me$ and $e'$ is oriented like the topological opposite edge $\overline{\me}$.
\end{itemize}
Take the source vertices of F, $\sigma:=\sigma(F):=\{s(e)|\ e\in F\}$ and denote by $\mathfrak{T}_\sigma$ the subtree generated by $\sigma$. Then
\begin{enumerate}
\item The open sets $\{\cB(e)\}_{e\in F}$ are pairwise disjoint if and only if $F\bigcap E(\mathfrak{T}_\sigma)=\emptyset$, which means $|F|\subset |\Star(\mathfrak{T}_\sigma)|$.
\item The equality $\displaystyle{\bigcup_{e\in F}\cB(e)=\cL^*}$ occurs if and only if $\Star(\mathfrak{T}_\sigma)\subset F$.
\end{enumerate}
\end{prop}
\begin{proof}
We will show the claims by induction on the cardinal of vertices $n=\#V(\mathfrak{T}_\sigma)$.

If $n=1$, then $\mathfrak{T}_\sigma=\{v\}=\sigma(F)$, $F\subset \Star(v)$, the sets $\cB(e)$ with $e\in \Star(v)$ are pairwise disjoint and $\displaystyle{\bigsqcup_{e\in F}{\cB(e)}=\cL^*}$ if and only if $F=\Star(v)$.

Next, assume $n>1$ and let $v\in \sigma=\sigma(F)$ be a vertex with valence $1$ in $\mathfrak{T}_\sigma$. Consider the non empty set $F_v:=\{e\in F|\ s(e)=v\}$, proper in $F$ since $n>1$, and let $e_v$ be the edge of $\mathfrak{T}_\sigma$ with target $t(e_v)=v$. Then, if $F':=(F\setminus F_v)\cup\{e_v\}$, we get the next remarkable properties:
\begin{itemize}
\item $\sigma':=\left(\sigma\setminus\{v\}\right)\cup\{s(e_v)\}=\sigma(F')$,
\item and $\#V(\mathfrak{T}_{\sigma'})=n-1$, so we may apply the induction hypothesis on $F'$.
\item $E(\mathfrak{T}_{\sigma'})=E(\mathfrak{T}_\sigma)\setminus \{e_v,\overline{e_v}\}$, so $$F'\cap E(\mathfrak{T}_{\sigma'})=(F\setminus F_v)\bigcap \big(E(\mathfrak{T}_\sigma)\setminus\{e_v,\overline{e_v}\}\big).$$
\item $\Star(\mathfrak{T}_{\sigma'})=(\Star(\mathfrak{T}_\sigma)\setminus \Star(v))\cup\{e_v\}$.
\item $\Star(\mathfrak{T}_{\sigma})=(\Star(\mathfrak{T}_{\sigma'})\setminus \{e_v\})\cup (\Star(v)\setminus\{\overline{e_v}\})$.
\end{itemize}

Suppose that $F\bigcap E(\mathfrak{T}_\sigma)=\emptyset$. Therefore
$F'\bigcap E(\mathfrak{T}_{\sigma'})=\emptyset$. Then, by induction
hypothesis, the open sets $\{\cB(e)\}_{e\in F'}$ are pairwise
disjoint and, in particular, $\cB(e_v)\cap\cB(e)=\emptyset$ for all
$e\in F\setminus F_v$. Recall now that
$$\displaystyle{\cB(e_v)=\bigsqcup_{e\in \Star(v)\setminus\overline{e_v}}\cB(e)}$$ and that $F_v\subset \Star(v)$.
But, $e_v$ and $\overline{e_v}$ are edges of $\mathfrak{T}_\sigma$,
so $F\bigcap E(\mathfrak{T}_\sigma)=\emptyset$ implies that $e_v,
\overline{e_v}\not\in F$, and therefore we get that the sets
$\{\cB(e)\}_{e\in F}$ are also pairwise disjoint.

Now assume that $F\bigcap E(\mathfrak{T}_\sigma)\neq\emptyset$.
Then, either $F'\bigcap E(\mathfrak{T}_{\sigma'})\neq\emptyset$, or
$F'\bigcap E(\mathfrak{T}_{\sigma'})=\emptyset$ but $$\emptyset\neq
F\bigcap E(\mathfrak{T}_\sigma)\subset \{e_v,\overline{e_v}\}.$$

In this last case, $F'\bigcap E(\mathfrak{T}_{\sigma'})=\emptyset$
and $F\bigcap E(\mathfrak{T}_\sigma)\subset \{e_v,\overline{e_v}\}$,
when $e_v\in F$, then $\cB(e_v)\cap\cB(e)\neq\emptyset$ for any
$e\in F_v\neq\emptyset$. In the case $\overline{e_v}\in F$, the fact
that $e_v\in F'$ and so that $\cB(e_v)\cap\cB(e)=\emptyset$ for any
$e\in F\setminus F_v$ (by induction on $F'$), together with
$\cB(\overline{e_v})=\cL^*\setminus\cB(e_v)$, implies that
$\cB(\overline{e_v})\cap\cB(e)\neq\emptyset$ for any $e\in
F\setminus F_v$.

If $F'\bigcap E(\mathfrak{T}_{\sigma'})\neq\emptyset$, the sets
$\{\cB(e)\}_{e\in F'}$ are not pairwise disjoint, and the collection
of sets $\{\cB(e)\}_{e\in F}$ include the same except maybe
$\cB(e_v)$, besides the $\{\cB(e)\}_{e\in F_v}$.

Therefore, if there are $e,e'\in F\setminus F_v$ such that
$\cB(e)\cap\cB(e')\neq\emptyset$ we get the claim. Otherwise there
is an $e_0\in F\setminus F_v$ such that
$\cB(e_0)\cap\cB(e_v)\neq\emptyset$ and $e_v\not\in F$. By
definition of $e_v$ we have that $s(e_v)\in P(t(e_v),s(e_0))$. Then,
taking in consideration the lemma~\ref{oprops} we get $s(e_0)\not\in
P(t(e_v),t(e_0))$ (and
$\cB(e_0)\cap\cB(e_v)=\cB(e_v)\subset\cB(e_0)$), since otherwise we
would have $s(e_0)\in P(t(e_v),t(e_0))$ and, as a consequence,
$\cB(e_0)\cap\cB(e_v)=\emptyset$.

Take now an edge $e_1\in F_v$. Assume first $e_1\neq\overline{e_v}$.
Then we obtain that $\cB(e_1)\subset\cB(e_v)\subset\cB(e_0)$ and
that the sets $\{\cB(e)\}_{e\in F}$ are not pairwise disjoint as we
wanted. To finish the proof of the the first equivalence, we just
have to deal with the case $F_v=\{\overline{e_v}\}$. Since $F$ is
well oriented, there is some vertex of valence three in $\cT_K(\cL)$
between $s(e_0)$ and $t(e_v)$ (excluding them). Then
$\cB(e_0)\cap\cB(\overline{e_v})\neq\emptyset$ by
lemma~\ref{oprops}.

Recalling the properties we have noted above, we get that $\Star(\mathfrak{T}_\sigma)\subset F$ implies $\Star(\mathfrak{T}_{\sigma'})\subset F'$, so, by hypothesis, $\displaystyle{\bigcup_{e\in F'}\cB(e)=\cL^*}$. By definition, we know that each edge of $F'$ is an edge of $F$ except at most $e_v$, but we have that $\Star(v)\setminus\{\overline{e_v}\}\subset \Star(\mathfrak{T}_\sigma)\subset F$ and $\displaystyle{\cB(e_v)=\bigsqcup_{e\in \Star(v)\setminus\overline{e_v}}\cB(e)}$, so
$$
\cL^*=\bigcup_{e\in F'}\cB(e)\subset\bigcup_{e\in F}\cB(e)=\cL^*.
$$

Suppose that $\Star(\mathfrak{T}_\sigma)\not\subset F$. This means that there is an edge $e\in \Star(\mathfrak{T}_\sigma)\setminus F$, in particular with $s(e)\in \sigma=\sigma(F)$. We may assume that the vertex $v$ we chose above in order to apply the induction method is different from $s(e)$. It is clear that $e\not\in F'$, and by the assumption $e\in \Star(\mathfrak{T}_{\sigma'})$, so $\Star(\mathfrak{T}_{\sigma'})\not\subset F'$ and $\displaystyle{\bigcup_{e\in F'}\cB(e)\neq\cL^*}$.

Finally, as we have seen before, we have
$$
\bigcup_{e\in F}\cB(e)=\bigcup_{e\in F\setminus F_v}\cB(e)\cup\bigcup_{e\in F_v}\cB(e)\subset\bigcup_{e\in F'}\cB(e)\cup\cB(e_v)\subset\bigcup_{e\in F'}\cB(e)\subsetneq\cL^*
$$

\end{proof}

\begin{cor} Let $\{\me_i\}_{i\in I}$ be a finite set
of topological oriented edges in $\cT_K(\cL)$ such that the open
sets $\cB(\me_i)$ for $i\in I$ are pairwise disjoint. Then $
\bigsqcup_{i\in I} \cB(\me_i) = \cL^* \Leftrightarrow
\{\me_i\}_{i\in I}=\Star(\mathfrak{T})$ for the finite subtree
$\mathfrak{T}$ with source vertices $\{s(\me_i)\}_{i\in I}$, or
$\{\me_i\}_{i\in I}=\{\me_1,\me_2\}$ existing a topological edge
$\me$ in $\cT_K(\cL)$ such that $\me_1\subset\me$ and
$\me_2\subset\overline{\me}$.
\end{cor}

In order to get another point of view for the harmonic cochains we have to define the harmonic measures on a suitable compact space.

\begin{defn}
Let $X$ be a compact space such that the compact open subsets form a
basis for the topology. A $\mb{Z}$-valued measure $\mu$ on $X$ is a finitely additive function on the (disjoint) compact subsets
of $X$. The set of $\mb{Z}$-valued measures on $X$ is denoted $\mathscr{M}(X,\mb{Z})$.
\end{defn}

\begin{defn} Let $\mu\in\mathscr{M}(X,\mb{Z})$ be a $\mb{Z}$-valued
measure on $X$. We say that $\mu$ is harmonic if the total
volume $\mu(X)$ is $0$. We denote the set of harmonic measures by $\mathscr{M}(X,\bZ)_0$.
\end{defn}

Note that as much $\mathscr{M}(X,\mb{Z})$ as $\mathscr{M}(X,\bZ)_0$ are abelian groups.

\begin{cor}\label{HMC} Any harmonic cochain $c$  of the metric tree $\cT_K(\cL)$
determines a unique harmonic measure $\mu(c)$ in $\mathscr{M}(
\cL^*,\mb{Z})_{0}$ by defining $\mu(c)(\cB(\me))=c(\me)$ for any
topological oriented edge $\me$ in $\cT_K(\cL)$. This induces an
isomorphism between $\mathscr{M}( \cL^*,\mb{Z})_{0}$ and
$\HC(\cT_K(\cL),\bZ)$.
\end{cor}
\begin{proof}
Essentially, all we have to check is that the map
${\HC(\cT_K(\cL),\bZ)\longrightarrow\mathscr{M}( \cL^*,\mb{Z})_{0}}$
given by the description above is well defined.

First, it is enough to characterize a measure over the sets $\cB(e)$ because of these form a basis for the topology of $\cL^*$.

Next, take a model $\mathfrak{T}=(V,E)$ for $\cT_K(\cL)$. We just have to see that for any open compact set $\mathcal{U}\subset\cL^*$ and for any partition $\mathcal{U}=\bigsqcup_{e\in I}{\cB(e)}$ with $I\subset E$ finite, the sum $\sum_{e\in I}{c(e)}$ is invariant. Let us take two finite partitions of $\cU$:
$$
\cU=\bigsqcup_{e\in I}{\cB(e)}=\bigsqcup_{e\in I'}{\cB(e)}
$$
Since $\mathcal{U}$ is open and compact so it is the complement $\mathcal{V}=\cL^*\setminus\mathcal{U}$ and we can consider another finite partition $\displaystyle{\mathcal{V}=\bigsqcup_{e\in\tilde{I}}{\cB(e)}}$, $\tilde{I}\subset E$. Then we have
$$
\cL^*=\cU\sqcup\mathcal{V}=\bigsqcup_{e\in I}{\cB(e)}\sqcup\bigsqcup_{e\in\tilde{I}}{\cB(e)}=\bigsqcup_{e\in I'}{\cB(e)}\sqcup\bigsqcup_{e\in\tilde{I}}{\cB(e)}
$$
Therefore, by the previous corollary, we get $I\sqcup\tilde{I}=\Star(\mathfrak{T})$ and $I'\sqcup\tilde{I}=\Star(\mathfrak{T}')$ for certain finite subtrees of $\cT$ (or any or both disjoint unions can be the degenerated case, which the reader can do as an easy exercise). Then we have
$$
\sum_{e\in I}{c(e)}+\sum_{e\in\tilde{I}}{c(e)}=\sum_{e\in I\sqcup\tilde{I}}{c(e)}=\sum_{e\in \Star(\mathfrak{T})}{c(e)}=0
$$
and
$$
\sum_{e\in I'}{c(e)}+\sum_{e\in\tilde{I}}{c(e)}=\sum_{e\in I'\sqcup\tilde{I}}{c(e)}=\sum_{e\in \Star(\mathfrak{T}')}{c(e)}=0
$$
after apply lemma~\ref{harstar}, so we get
$$
\sum_{e\in I}{c(e)}=\sum_{e\in I'}{c(e)}
$$
as we wanted to prove.

Once we have the map well defined, it follows immediately from the definition that it is an isomorphism of abelian groups. Indeed, the kernel has to be zero and the same definition provides the exhaustivity.
\end{proof}




\section{Multiplicative Integrals}\label{MI}

The following definition was introduced by Longhi \cite{Lon02} as a
generalization of Darmon \cite{Dar01}.

\begin{defn}
Let $X$ be a compact space such that the compact open subsets form a
basis for the topology. Let $G$ be a complete topological abelian
group (written multiplicatively) such that a basic system of
neighbourhoods of the identity consists of open subgroups. Let
$\ds{f:X\lra G}$ be a continuous function and let
$\mu\in\mathscr{M}(X,\mb{Z})$ be a $\mb{Z}$-valued measure on $X$. The \textbf{multiplicative integral} of $f$ with respect to
$\mu$ is defined as
$$
\mint_{X}{fd\mu}:=\mint_{X}{f(t)d\mu(t)}:=
\lim_{\substack{\rightarrow\\\mc{C}_\al}}{\ds{\prod_{\substack{\mc{U}_n^\al\in\mc{C}_\al\\t_n^\al\in\mc{U}_n^\al}}{f(t_n^\al)^{\mu(\mc{U}_n^\al)}}}}
$$
where the limit is taken over the direct system of finite covers
$\ds{\mc{C}_\al=\mc{C}_\al(X)}$ of $X$ by disjoint open compact subsets
$\ds{\mc{U}_n^\al}$, and the $\ds{t_n^\al}$ are arbitrary points in
them.
\end{defn}
It is well defined since the limit exists and it does not depend on the choice of the $t_n^\al$'s. (\cite[Prop.~5]{Lon02})

\begin{prop}\label{propertiesmultiplicativeintegral} For any measure $\mu\in\mathscr{M}(X,\mb{Z})$, we
have
\begin{enumerate}
\item For any compact open subset $U$ of $X$, and for any $\gamma\in G$, denote by $\chi_{U,\gamma}(t)$
the function sending $x\in X$ to $\gamma$ if $x\in U$, and to $1$
otherwise. Then $\displaystyle{\mint_{X}{\chi_{U,\gamma} d\mu}=\gamma^{\mu(U)}}$.
\item If $\ds{f,g:X\lra G}$ are continuous functions on $X$, then
$$\mint_{X}{(f\cdot g) d\mu}=\left(\mint_{X}{f d\mu}\right)\left(\mint_{X}{ g d\mu}\right)$$
\end{enumerate}
\end{prop}

Note that for any harmonic measure $\mu$ and any constant function
$\ds{f:X\lra G}$ such that $f(x)=\lambda$ for all $x\in X$, we have
$\displaystyle{\mint_{X}{fd\mu}=1}$.

Now, let $\cL$ be a compact subset of ${\bP^1}^*(K)$ with at least two points and let $L|K$ be an arbitrary complete extension of fields. We
get from them the set $\cL^*\subset\bP^1(K)$, the
space $\Omega_\cL(L)$ and the tree $\cT_K(\cL)$. With these objects we give the next definitions and lemmas.

\begin{defn}\label{FutS}
Let $\cP$ be a finite set of points in $\Omega_\cL(L)$, and consider
$D:=\sum_{p\in \cP} m_p p$ a divisor of degree zero. We denote by
$f_D$ the element of $\mathscr{M}aps(\cL^*, L^*)/L^*$ which is
defined up to scalars as follows: if we choose representatives
$v_p\in (L^2)^*$ for any $p\in \cP$ and $v_q\in K^2$ for $q$, then
$f_D(q)=\prod_{p\in \cP} v_p(v_q)^{m_p}$ does not depend on $v_q$.
Any other election of the vectors $v_p$ change $f_D$ to $\lambda
f_D$ for some $\lambda\in L^*$.

Similarly, if $\cA$ is a finite set of points in $\cT_{L}$, and
consider $D:=\sum_{[\alpha]\in \cA} m_{[\alpha]} [\alpha]$ a degree
zero divisor, then we denote by $|f|_D$ the element of
$\mathscr{M}aps(\cL^*, \bR_{>0})/\bR_{>0}^*$ being defined up to
scalars by $|f|_D(q)=\prod_{\alpha\in \cA} \alpha(q)^{m_{[\alpha]}
}$ (remind that the points $[\alpha]$ are classes modulo homothety
of diagonalizable seminorms $\alpha$).

We note that we will be flexible when using these notations, not making difference between the map and the class of the map.

We note also that any representant of $f_D$ can be seen as a map
which extends to a meromorphic function on $\bP^1$ with divisor $D$.

\end{defn}

\begin{obs}
We can see the degree zero divisor $0$ as the divisor $0p$ for any ${p\in\Omega_\cL(L)}$. Therefore, as $m_p=0$, we get $f_0\equiv1$ and $|f|_0\equiv1$.
\end{obs}

As a particular case, if we consider the divisor $D:=\alpha(x,s)-\alpha(x,r)$ in $\cT_K(\cL)$, where $s>r$, then we have
$$
|f|_{D}(q)=\left\{\begin{array}{ll} \frac sr & \mbox{ if }
q\in B(x^*,r) \\
\frac{s}{|q-x^*|} & \mbox{ if } q\in B(x^*,s)\setminus B(x^*,r) \\
1 & \mbox{ if } q\not\in B(x^*,s)
\end{array} \right.
$$
for any $q\in \cL^*$.

 Observe that, if the path from $\alpha(x,r)$
to $\alpha(x,s)$ is a topological edge, then $\cL \cap (B(x,s-\epsilon)\setminus
B(x,r))$ is empty for any $s-r>\epsilon>0$ (and so the corresponding intersection with $\cL^*$), and then
$|f|_{D}(q)=1$ or $\frac sr$ for any $q\in \cL^*$.

\begin{prop}\label{frd} Let $\cA$ be a finite set of points in $\cT_K$,
let ${D:=\sum_{\alpha\in \cA} m_{\alpha} \alpha}$ be a degree zero
divisor and consider its retraction $\red_{\cL}(D):=\sum_{\alpha\in \cA} m_{\alpha}
\red_{\cL}(\alpha)$. Then
$|f|_{D}=|f|_{\red_{\cL}(D)}$ in $\mathscr{M}aps(\cL^*, \bR_{>0})/\bR_{>0}^*$.
\end{prop}

\begin{proof}
First of all, observe that in the case $\cL=\{y_0,y_1\}$ this is a
consequence of lemma~\ref{red2points}.

Now we do the general case. Fix $x\in \cL$ and consider any point
$y\in \cL$, $x\ne y$. Take $\cL':=\{y,x\}\subset \cL$. Using the
previous case twice (and taking some representatives) we get that
$${|f|_{D}}_{|\cL'^*}={|f|_{\red_{\cL'}(D)}}_{|\cL'^*}={|f|_{\red_{\cL'}(\red_{\cL}(D))}}_{|\cL'^*}={|f|_{\red_{\cL}(D)}}_{|\cL'^*}$$
by applying lemma~\ref{extend}. Since this equality is satisfied for all $\cL'$ with $x$ fixed, it is satisfied for $\cL$ too (if we looked to the maps representing these classes modulo homothety, it would appear some scalar at the end of the equality which would not depend on $\cL'$ or on $y$ due to the fixed $x$).
\end{proof}


\begin{defn}\label{defint} Given any degree 0 divisor
$D=\sum_{i\in I} m_i p_i$ with support in $\Omega_\cL(L)$ (i.e. $m_i\in \bZ$, $p_i\in \Omega_\cL(L)$,
being $I$ a finite set and with $\sum_{i\in I}m_i=0$) we choose
$v_i$ in $(L^2)^*$ representatives of the $p_i\in{\bP^1}^*(L)$ and consider
the map up to scalars $\ds{f_D\in\mathscr{M}aps(\cL^*, L^*)/K^*}$ given by a representant $\prod_{i\in
I} v_i(x)^{m_i}$ (which depend on the $v_i$'s). Let $\mu\in\mathscr{M}(
\cL^*,\mb{Z})_{0}$ be a $\mb{Z}$-valued harmonic measure on
$\cL^*$.

We define $$\displaystyle{\mint_{\cL^*,D}{d\mu}:= \mint_{\cL^*}{f_D
d\mu}\in L^*},$$ which is well defined since the integral does not
depend on $f_D$ but only on $D$. Indeed, although the representant
of $f_D$ depend on the elections of the representatives in $(L^2)^*$
of the points in ${\bP^1}^*(L)$, the multiplicative integral does
not, since the measure is harmonic.

In general, when some $\cL$ was fixed previously -as along this section-, we will omit its corresponding set, writing $$\mint_D{d\mu}:=\mint_{\cL^*,D}{d\mu},$$ meanwhile we will specify the other sets over which we will integrate.
\end{defn}

Note also that when $D=0$, since $f_0\equiv1$, we have $\ds{\mint_0{d\mu}=1}$.

Therefore, this definition gives us a morphism of groups
$$
\xymatrix@R=.1pc{
\bZ[\Omega_\cL(L)]_0\ar[rr]^(.42){\displaystyle{\mint_\bullet {d}}}&&\displaystyle{\mathrm{Hom}(\mathscr{M}(\cL^*,\mb{Z})_{0},{L}^*)}\\
D\ar@{|->}[rr]&&\displaystyle{\mint_D{d}:\mu\longmapsto \mint_{D}{d\mu}}
}
$$

\begin{lem}
Let $\Gamma\subset \PGL_2(K)$ be a subgroup acting on $\cL$ and so on $\cL^*$. Then, the map $\displaystyle{\mint_\bullet{d}}$ is $\Gamma$-equivariant.
\end{lem}
\begin{proof}
We want to see that $\displaystyle{\mint_{\gamma\cdot D}{d}=\gamma\cdot\mint_D{d}}$ for any $\gamma\in\Gamma$. That is to say that for any $\gamma\in\Gamma$ and $\mu\in\mathscr{M}(\cL^*,\mb{Z})_{0}$ we have
$$
\mint_{\cL^*}{f_{\gamma D}d\mu}=\mint_{\gamma\cdot D}{d\mu}=\gamma\cdot\mint_D{d\mu}=\mint_D{d(\gamma^{-1}\mu)}=\mint_{\cL^*}{f_Dd(\gamma^{-1}\mu)}
$$
Let us to compute the first integral:
$$
\mint_{\cL^*}{f_{\gamma D}d\mu}=
\lim_{\substack{\rightarrow\\\mc{C}_\al(\cL^*)}}{\ds{\prod_{\substack{\mc{U}_n^\al\in\mc{C}_\al(\cL^*)\\t_n^\al\in\mc{U}_n^\al}}{f_{\gamma D}(t_n^\al)^{\mu(\mc{U}_n^\al)}}}}=
\lim_{\substack{\rightarrow\\\mc{C}_\al(\cL^*)}}{\ds{\prod_{\substack{\mc{U}_n^\al\in\mc{C}_\al(\cL^*)\\t_n^\al\in\mc{U}_n^\al}}{(\gamma f_{D})(t_n^\al)^{\mu(\mc{U}_n^\al)}}}}=
$$
$$
=\lim_{\substack{\rightarrow\\\mc{C}_\al(\cL^*)}}{\ds{\prod_{\substack{\mc{U}_n^\al\in\mc{C}_\al(\cL^*)\\t_n^\al\in\mc{U}_n^\al}}{f_D(\gamma^{-1} t_n^\al)^{\mu(\mc{U}_n^\al)}}}}=
\lim_{\substack{\rightarrow\\\mc{C}_\al(\cL^*)}}{\ds{\prod_{\substack{\mc{U}_n^\al\in\mc{C}_\al(\cL^*)\\t_n^\al\in\mc{U}_n^\al}}{f_D(t_n^\al)^{\mu(\gamma\mc{U}_n^\al)}}}}=
$$
$$
=\lim_{\substack{\rightarrow\\\mc{C}_\al(\cL^*)}}{\ds{\prod_{\substack{\mc{U}_n^\al\in\mc{C}_\al(\cL^*)\\t_n^\al\in\mc{U}_n^\al}}{f_D(t_n^\al)^{(\gamma^{-1}\mu)(\mc{U}_n^\al)}}}}=\mint_{\cL^*}{f_Dd(\gamma^{-1}\mu)}
$$
Therefore we get the claimed compatibility of the action of $\Gamma$ with the map.
\end{proof}

\begin{defn}  Given any degree 0 divisor $D=\sum_{i\in I} m_i \alpha_i$ with support in
$\cT_K(\cL)$ (i.e. $m_i\in \bZ$, $\alpha_i\in \cT_K(\cL)$, with $I$ a finite
set and $\sum_{i\in I}m_i=0$) consider the map up to scalars
$\ds{|f|_D\in\mathscr{M}aps(\cL^*, \bR_{>0})/\bR_{>0}^*}$ given by a representant $\prod_{i\in I}
\alpha_i(x)^{m_i}$. Let
$\mu\in\mathscr{M}(\cL^*,\mb{Z})_{0}$ be a $\mb{Z}$-valued harmonic
measure on $\cL^*$.\\
We define
$$\displaystyle{\left|\mint_{\cL^*}\right|_{D}{d\mu}:=\mint_{\cL^*}{|f|_D d\mu}\in \bR_{>0}}$$
since, as above, the value of the integral only depends on $D$, and not on the representant of $|f|_D$, because of the harmonicity of the measure.

We will follow the same rule that above with respect to $\cL^*$, omiting it when it is a given fixed set and specifying only in case of need:
$$
\left|\mint\right|_{D}{d\mu}:=\left|\mint_{\cL^*}\right|_{D}{d\mu}.
$$

\end{defn}

\begin{lem}\label{AVC} Let $\cP$ be a finite set of points in $\Omega_\cL(L)$,
and let $D:=\sum_{p\in \cP} m_p p$ be a degree zero divisor. Denote by
$\alpha_D:=\sum_{p\in \cP} m_p \alpha_p$, where $\alpha_p$ is the
seminorm associated to $p$. Then $|f_D|=|f|_{\alpha_{D}}$
in $\mathscr{M}aps(\cL^*, \bR_{>0})/\bR_{>0}^*$.
\end{lem}

\begin{proof} Take $q\in\cL^*$ and representatives as in definition~\ref{FutS}. For the sake of simplicity we will assume  all the points $p$ and $q$ are non infinite (then we can choose $v_q=(q,1)$ and $v_p=(1,-p)$).
$$
|f_D|(q)=|f_D(q)|=|\prod_{p\in \cP} v_p(v_q)^{m_p}|=\prod_{p\in \cP} {|q-p|^{m_p}}=|\prod_{p\in \cP} {\alpha_p(q)}|=|f|_{\alpha_D}(q)
$$
by having into account for the fourth equality the remark~\ref{frem}.\end{proof}

This has as an immediate consequence the next result.

\begin{cor}\label{absI}
For any degree 0 divisor
$D=\sum_{i\in I} m_i p_i$ with support in $\Omega_\cL(L)$, consider
the divisor $\red_{\cL}(D):=\sum_{i\in I}
m_i \red_{\cL}(p_i)$ on $\cT_K(\cL)$. Then, for any $\mb{Z}$-valued harmonic
measure $\mu\in\mathscr{M}(\cL^*,\mb{Z})_{0}$ on $\cL$, we have
$$\left|\mint_{D}{d\mu}\right|=\left|\mint\right|_{\red_\cL(D)}{d\mu}.$$
\end{cor}
\begin{proof}
Applying the previous lemma and proposition~\ref{frd} we obtain
$$
\left|\mint_{D}{d\mu}\right|=\mint_{\cL^*}{|f_D|d\mu}=\mint_{\cL^*}{|f|_{\alpha_D}d\mu}=\mint_{\cL^*}{|f|_{\red_\cL(D)}d\mu}=\left|\mint\right|_{\red_\cL(D)}{d\mu}
$$
\end{proof}

\begin{lem} Given $x\in \cL$, for any
two points $\alpha(x,r),\alpha(x,s)\in \cT_K(\cL)$, with $s>r$,
such that the path $P(\alpha(x,r),\alpha(x,s))$ is a topological edge, then
$$\left|\mint\right|_{\alpha(x,s)-\alpha(x,r)}{d\mu}=\left(\frac sr\right)^{\mu(B(x^*,r)\cap\cL^*)}$$
\end{lem}

\begin{proof} We have
$$
|f|_D(q)=\left\{\begin{array}{ll} \frac sr & \mbox{ if }
q\in B(x^*,r) \\
1 & \mbox{ if } q\not\in B(x^*,r)
\end{array} \right.
$$
and these are the only two possibilities. Hence
$\displaystyle{|f|_D(q)=\chi_{U,\frac sr}}$ for $U=B(x^*,r)$ in the
notation of Proposition~\ref{propertiesmultiplicativeintegral}.

Now, if we denote by $D=\alpha(x,s)-\alpha(x,r)$, and by applying
Proposition~\ref{propertiesmultiplicativeintegral}, we get
$$
\left|\mint\right|_{D}{d\mu}= \mint_{\cL^*}{|f|_D d\mu}=
 \mint_{\cL^*}{\chi_{U,\frac sr} d\mu}=\left(\frac sr\right)^{\mu(B(x^*,r)\cap\cL^*)}.
$$
\end{proof}

The following result generalizes \cite[Lem. 4.2]{Das05} to the case
that $K$ is any complete non-archimedean field (and not just a local
field).

\begin{lem}\label{logint}
For any $\alpha,\alpha'\in \cT_K(\cL)$ such that $P(\alpha,\alpha')$ is a topological edge, then
$$
\mathit{v}_K\left(\mint_{\alpha'-\alpha}{d\mu}\right)=-\log{\left|\mint\right|_{\alpha'-\alpha}{d\mu}}=d(\alpha,\alpha')\mu(\cB(\alpha,\alpha'))
$$
where the first equality is by corollary~\ref{absI}.
\end{lem}
\begin{proof}
If $\alpha=\alpha(x,r),\ \alpha'=\alpha(x,s)$ are like in the
previous lemma , then the claim is immediate consequence of that
together with definition~\ref{open}. One only has to observe that
$B(x^*,r)\cap\cL^*= \mathring{B}(x^*,s)\cap\cL^*$, so
$$
\mu(B(x^*,r)\cap\cL^*)=\mu(\mathring{B}(x^*,s)\cap\cL^*)=-\mu(\cL^*\setminus \mathring{B}(x^*,s))=-\mu(\cB(\alpha,\alpha'))
$$
We get the identity in a similar way when $r>s$.

Otherwise
$$
-\log{\left|\mint\right|_{\alpha'-\alpha}{d\mu}}=-\log{\left|\mint\right|_{\alpha'\vee\alpha-\alpha}{d\mu}}-\log{\left|\mint\right|_{\alpha'-\alpha\vee\alpha'}{d\mu}}=
$$
$$
=d(\alpha,\alpha'\vee\alpha)\mu(\cB(\alpha,\alpha'\vee\alpha))+d(\alpha\vee\alpha',\alpha')\mu(\cB(\alpha\vee\alpha',\alpha'))=
$$
$$
=d(\alpha,\alpha'\vee\alpha)\mu(\cB(\alpha,\alpha'))+d(\alpha\vee\alpha',\alpha')\mu(\cB(\alpha,\alpha'))=d(\alpha,\alpha')\mu(\cB(\alpha,\alpha'))
$$
since $\cB(\alpha,\alpha')=\cB(\alpha,\alpha'\vee\alpha)=\cB(\alpha\vee\alpha',\alpha')$.
\end{proof}

We may show this result in a more expressive way writing the
topological edge as $\me$ and defining its boundary $\partial \me$
as the difference of its target minus its source -as usual in
homology theory (cf. below in section~\ref{tropical}).

Recall that by corollary~\ref{HMC} we have $\ds{\mathscr{M}(
\cL^*,\mb{Z})_{0}\cong \HC(\cT_K(\cL),\bZ)}$ in such a way that to
each measure $\mu$ corresponds an harmonic cochain $c_\mu$ such that
${c_\mu(\me)=\mu(\cB(\me))}$. So, by abuse of notation we may write
$\mu(\me)=\mu(\cB(\me))$.

Therefore, we may write the lemma as
$$
\mathit{v}_K\left(\mint_{\partial\me}\right)=l(\me)\mu(\me).
$$




\section{The Poisson Formula}\label{poisson}

In this section we will show in our context the Poisson formula of
Longhi in \cite[Thm.~6]{Lon02}. To show this, we recall and study in
detail a map introduced by van der Put in \cite[Thm.~2.1]{vdP92}
assigning a mesure to any invertible analytic function.

Let $\cL$ be a compact set with at least two points and consider the
abelian group of measures $\mathscr{M}(\cL^*,\mb{Z})_{0}$, as in the
previous section. For any two different points $a,b\in\cL^*$ we
define the measure $\mu_{a,b}$ by
$$
\mu_{a,b}(\cU):=\left\{\begin{array}{l}
1\mbox{ if }a\in\cU,\ b\not\in\cU\\
-1\mbox{ if }b\in\cU,\ a\not\in\cU\\
0\mbox{, otherwise}
\end{array}\right.
$$

In particular, on the open compact subsets $\cB(e)\subset\cL^*$, which determine the measure because of being a basis, we note that
$$
\mu_{a,b}(e):=\left\{\begin{array}{l}
1\mbox{ if }e\in P(b^*,a^*)\\
-1\mbox{ if }e\in P(a^*,b^*)\\
0\mbox{, otherwise}
\end{array}\right.
$$

For any  $a,b\in\cL^*$ we take representatives $\tilde{a},\tilde{b}\in K^2$ and for any complete extension $L|K$ we define
the function $\ds{\omega_{\ta-\tb}:\Omega_\cL(L)\lra L^*}$ as $$\omega_{\ta-\tb}(z):=\frac{\ta(z)}{\tb(z)}=\frac{z(\ta)}{z(\tb)}.$$
Note that identifying $z$ with $(1,-z)$ or $(0,1)$ if it is $\infty$, this is an analytic function on $\Omega_\cL(L)$ depending on $a,b$
up to a constant.

Let us write for any $p,q\in\cL$, $u_{p,q}(z):=\omega_{\tilde{p}^*-\tilde{q}*}$ for suitable representants, so we can put
$$
u_{p,q}(z):=\frac{z-p}{z-q}
$$
where we consider the usual convention when some of the two points are $\infty$ (\cite[Ch.~2(2.2)]{GvdP80}), that is
$$u_{p,q}(z):=\left\{\begin{array}{l}
1\mbox{ if }p=q=\infty\\
z-p\mbox{ if }p\neq\infty=q\\
\ds{\frac{1}{z-q}}\mbox{ if }p=\infty\neq q
\end{array}\right.
$$

On the other hand, let us recall part of the definition~\ref{defint}. For any degree 0 divisor $D=\sum_{i\in I} m_i p_i$ with
support in $\Omega_\cL(L)$ we could build as above a map up to scalars
$\ds{f_D\in\mathscr{M}aps(\cL^*, L^*)/L^*}$. Let us fix an element $b_0\in\cL$. Along this section
we will choose a representant of $f_D$ satisfying $f_D(b_0)=1$, so $f_D$ will be a well defined function.

We write the usual notation $\mathcal{O}(\Omega_\cL)$ for the analytic functions on the analytic space $\Omega_\cL:=({\bP^1_{K}}^*)^{an}\setminus\cL$, and we write $\mathcal{O}(\Omega_\cL)^*$ for the ones which vanish nowhere. Then we have $\ds{\omega_{\ta-\tb}\in\mathcal{O}(\Omega_\cL)^*}$.

Let $\me$ be a topological edge of $\cT_K(\cL)$ induced by a path
$P(\alpha(x,r),\alpha(x,s))$ with $x\in\cL$ and $r\leq s$. Then we
define the (closed) annulus associated to $\me$ as
$R(\me):=R_x(r,s):=B(x,s)\setminus\mathring{B}(x,r)$, and the open
annulus associated to $\me$ as
$\mathring{R}(\me):=\mathring{R}_x(r,s):=\mathring{B}(x,s)\setminus
B(x,r)$.

The following result is shown by Thuiller in \cite[Lemme
2.2.1]{Thu05}.

\begin{lem}\label{PoiL}
Given $x\in \cL$, and any two points $\alpha(x,r),\alpha(x,s)\in \cT_K(\cL)$, with $r<s$, such that the path $P(\alpha(x,r),\alpha(x,s))$ is a topological edge (i.e. $\mathring{R}_x(r,s)\cap\cL=\emptyset$), for any $\omega\in\mathcal{O}(\Omega_\cL)^*$ there exists $k\in\bZ$ such that for any interior path $P'=P(\alpha(x,r'),\alpha(x,s'))\subset P(\alpha(x,r),\alpha(x,s))$ ($r\leq r'\leq s'\leq s$) satisfying $R_x(r',s')\cap\cL=\emptyset$, the function $|\omega(z)(z-x)^{-k}|$ is constant on $R_x(r',s')$.
\end{lem}
\begin{proof}

For any $0<r'\leq s'$ let us consider $R_x(r',s')^{an}$, the Berkovich analytic annulus associated to $R_x(r',s')$. Now we can assume without any problem that $x=0$. Then, we have the isomorphism
$$
\mathcal{O}(R_0(r',s')^{an})\cong K\langle r'T^{-1},s'^{-1}T\rangle
$$
where
$$
K\langle r'T^{-1},s'^{-1}T\rangle=
$$
$$
=\left\{\sum_{n=-\infty}^{\infty}{a_nT^n}:\ |a_n|r'^n\rightarrow0\mbox{ as }n\rightarrow-\infty,\ |a_n|s'^n\rightarrow0\mbox{ as }n\rightarrow\infty\right\}
$$
We will prove first the case $r'=s'=1$. We have $\omega\in\mathcal{O}(\Omega_\cL)^*$ and then the restriction of $\omega$ is a unit in $K\langle T^{-1},T\rangle$. Such an element can be expressed as $\omega=c\cdot\omega_1$ for $c\in K^*$ such that $\lVert\omega\rVert_{R_0(1,1)}=|c|$ and $\omega_1\in\mathcal{O}_K\langle T^{-1},T\rangle^*$. Therefore, the reduction of $\omega_1$ to $k[T^{-1},T]^*$ is also invertible so it is of the form $bT^n$ for $b\in k,\ n\in\bZ$, and so we deduce that we can write $\omega_1=\tilde{b}T^n+\omega_2=\tilde{b}T^n(1+\omega'_2)$ with $\tilde{b}\in\mathcal{O}_K^*,\ \omega_2\in\mathfrak{m}_K\langle T^{-1},T\rangle,\ \omega'_2=\tilde{b}^{-1}T^{-n}\omega_2$ and $\lVert\omega'_2\rVert_{R_0(1,1)}=\lVert\omega_2\rVert_{R_0(1,1)}<1$, so $|\omega(z) z^{-n}|=|c\tilde{b}||1+\omega'_2(z)|=|c\tilde{b}|$.

Observe that writing $\omega=\sum_{n\in\bZ}{a_nT^n}$ the supremum norm can be expressed by $\lVert\omega\rVert_{R_0(1,1)}:=\max\{|a_m|\}$ and this is reached at just one $m$, which is $n$.

From now on we consider the case $r'<s'$.
Now $\omega$ is a unit $\sum_{n\in\bZ}{a_nT^n}$ in $K\langle r'T^{-1},s'^{-1}T\rangle$, so for any $r''\in[r',s']$, the image of $\omega$ by the restriction homomorphism $\ds{K\langle r'T^{-1}s'^{-1},T\rangle\longrightarrow K\langle r''T^{-1},r''^{-1}T\rangle}$ is also a unit. Next note that after a non archimedean extension $K'|K$ we have $r''\in|K'^*|$ so there is an isomorphism $\ds{K'\langle r''T^{-1},r''^{-1}T\rangle\cong K'\langle T^{-1},T\rangle}$.
\end{proof}

We say that a sequence of functions $(\omega_n)_n$ in $\mathcal{O}(\Omega_\cL)^*$ converge uniformly to a function $\omega\in\mathcal{O}(\Omega_\cL)^*$ if for each edge $e$ of $\cT_K(\cL)$ and for all $\epsilon>0$ there exists an $n_0=n(e,\epsilon)$ such that for any $N\geq n_0$ we have $\ds{\lVert\omega-\omega_N\rVert_{R(|e|)}<\epsilon}$ (recall that $|e|$ means the topological realization of $e$).
We will write $\displaystyle{\lim_{N\to\infty}{\omega_N}=\omega}$.

\begin{thm}\label{FMeas}
There exists a morphism $\displaystyle{\tilde{\mu}:\mathcal{O}(\Omega_\cL)^*\longrightarrow\mathscr{M}(\cL^*,\mb{Z})_{0}}$ with kernel $K^*$ and such that commutes with limits in the following sense: if $\displaystyle{\lim_{N\to\infty}{\omega_N}=\omega}$, then $\displaystyle{\tilde{\mu}(\omega)=\lim_{N\rightarrow\infty}{\tilde{\mu}(\omega_N)}}$.
\end{thm}
\begin{proof}
Let us consider $\omega\in\mathcal{O}(\Omega_\cL)^*$. We have to
define $\tilde{\mu}(\omega)$ over each (oriented) edge $e$ of a
model of $\cT_K(\cL)$. By lemma~\ref{oprops} we may assume that
$|e|$ or $|\overline{e}|$ is contained in a topological edge given
by $P(\alpha(x,r),\alpha(x,s))$ with $r<s$ and $x\in\cL\cap K$.
Depending on if this happens with $e$ or $\overline{e}$, we define
$\tilde{\mu}(\omega)(e):=k$ or $\tilde{\mu}(\omega)(e):=-k$
respectively, where $k$ is the integer obtained in the above lemma.
Henceforth we will work on this edge to prove its properties.

First, $\tilde{\mu}(\omega)$ is a measure because of the definition and the residue theorem (\cite[Thm.~2.3.3 (2)]{FvdP04}).

Form the way we have defined the map $\tilde{\mu}$ it is clear that
it is a morphism and that $K^*$ is inside its kernel. From the
definition of $\tilde{\mu}$, the fact that $\Omega_\cL$ is connected
implies that if $\tilde{\mu}(\omega)=0$, then the absolute value of
$\omega$ is a constant, and since bounded analytic functions on
$\Omega_\cL$ are constant (\cite[Ch.~4~Cor.~(2.5)]{GvdP80}), we get
$\Ker(\tilde{\mu})=K^*$.

And now let us see the commutativity with limits in the sense we told. We want to check the equality $\ds{\tilde{\mu}(\omega)(e)=\lim_{N\rightarrow\infty}{\tilde{\mu}(\omega_N)}(e)}$ for any edge $e$ that we can take as above.

We know by hypothesis that for any $\epsilon>0$ there exists an $n_0=n(e,\epsilon)$ such that for any $N\geq n_0$ we have $\lVert\omega-\omega_N\rVert_{R_x(r,s)}<\epsilon$. Note that if we just take $\epsilon=\inf_{z\in R_x(r,s)}{\{|\omega(z)|\}}$, which is strictly positive since $R_x(r,s)$ is compact, then for any $z\in R_x(r,s)$ we get $|\omega(z)-\omega_N(z)|<|\omega(z)|$ and so $\lvert \omega_N(z)\rvert=\lvert \omega(z)\rvert$, therefore $\tilde{\mu}(\omega_N)(e)=\tilde{\mu}(\omega)(e)$.
\end{proof}

\begin{prop}\label{FMeas2}
The morphism $\displaystyle{\tilde{\mu}:\mathcal{O}(\Omega_\cL)^*\longrightarrow\mathscr{M}(\cL^*,\mb{Z})_{0}}$ satisfies the following properties:
\begin{enumerate}
\item For any two different points $a,b\in\cL^*$,
$$\displaystyle{\tilde{\mu}\left(\omega_{\ta-\tb}\right)=\mu_{b,a}}$$ independently of the chosen representants of $a$ and $b$. In particular, for any $p,q\in\cL$ we have $\displaystyle{\tilde{\mu}\left(u_{p,q}\right)=\mu_{q^*,p^*}}$.
\item It is natural in the sense that if $\cL\subset\cL'$ are both compacts, it commutes with restriction maps:
$$
\xymatrix{
\mathcal{O}(\Omega_\cL)^*\ar[rr]^{\displaystyle{\tilde{\mu}}}\ar[dd]&&\mathscr{M}(\cL^*,\mb{Z})_{0}\ar[dd]\\
&&\\
\mathcal{O}(\Omega_{\cL'})^*\ar[rr]^{\displaystyle{\tilde{\mu}}}&&\mathscr{M}(\cL'^*,\mb{Z})_{0}
}
$$
In particular it does not depend on $\cL$, since given any compacts $\cL_1,\ \cL_2$, the definition coincides in $\cL_1\cap\cL_2$.
\item It commutes with the action of $\PGL_2(K)$, that is, for each $\displaystyle{\gamma\in \PGL_2(K)}$ the diagram
$$
\xymatrix{
\mathcal{O}(\Omega_\cL)^*\ar[rr]^{\displaystyle{\tilde{\mu}}}\ar[dd]^{\displaystyle{\gamma_*}}&&\mathscr{M}(\cL^*,\mb{Z})_{0}\ar[dd]^{\displaystyle{\gamma_*}}\\
&&\\
\mathcal{O}(\Omega_{\gamma\cL})^*\ar[rr]^{\displaystyle{\tilde{\mu}}}&&\mathscr{M}(\gamma\cL^*,\mb{Z})_{0}
}
$$
 is commutative, where $\gamma_*(\omega)=\gamma\cdot\omega$ and $\gamma_*(\mu)=\gamma\cdot\mu$. (Note that $\Omega_{\gamma\cL}=\gamma\Omega_{\cL}$ and $\gamma\cdot(\cL^*)=(\gamma\cdot\cL)^*$.)
\end{enumerate}
\end{prop}

\begin{proof}
First, we want to see $\displaystyle{\tilde{\mu}\left(\omega_{\ta-\tb}\right)(e)=\mu_{b,a}(e)}$.

If $a,b\in\cB(e)=\cL^*\setminus \mathring{B}(x^*,s)$, for $z\in R_x(r,s)$ we have
$$|\omega_{\ta-\tb}(z)|=\left|\frac{z(\ta)}{z(\tb)}\right|=\frac{|x-a^*|}{|x-b^*|}
$$
(taking into account the above convention if $a$ or $b$ are $\infty$), which is a constant, so ${\tilde{\mu}(\omega_{\ta-\tb})(e)=0=\mu_{b,a}(e)}$.

If $a,b\in\cB(\overline{e})=\cL^*\cap B(x^*,r)$, $z\in R_x(r,s)$ satisfies $|z(\ta)|=|z-a^*|=|z-x|=|z-b^*|=|z(\tb)|$, so we also get a constant ($|\omega_{\ta-\tb|R_x(r,s)}|\equiv 1$) and the equality as above.

Finally, assuming $a\in\cB(e)=\cL^*\setminus \mathring{B}(x^*,s),\ b\in\cB(\overline{e})=\cL^*\cap B(x^*,r)$, then
$$
|\omega_{\ta-\tb}(z)|=\left|\frac{z(\ta)}{z(\tb)}\right|=\frac{|x-a^*|}{|z-b^*|}=\frac{|x-a^*|}{|x-z|}\frac{|x-z|}{|z-b^*|}=|x-a^*|\cdot|z(x^*)|^{-1},
$$
therefore $\ds{\tilde{\mu}(\omega_{\ta-\tb})(e)=-1=\mu_{b,a}(e)}$
(once more, one should consider the case in which $a$ is $\infty$,
but we would get a similar result).

Second, the naturality is a direct consequence of the definition of the $\tilde{\mu}$ through the above lemma.

The third property is equivalent to say
$\gamma\cdot\tilde{\mu}(\omega)(e)=\tilde{\mu}(\gamma\cdot\omega)(e)$
for all $\omega\in\cO(\Omega_\cL)^*$ and
$e\in\cT_K(\gamma\cdot\cL)$, and the left side of the equality is
$\tilde{\mu}(\omega)(\gamma^{-1}\cdot e)$. Then, this also follows
from the definition by means of the lemma and from the isomorphism
$\displaystyle{\gamma^*:\cO(R(|e|))\stackrel{\cong}\longrightarrow\cO(R(|\gamma^{-1}e|))}$,
by which $\gamma^*(\omega)=\gamma^{-1}\cdot\omega$.
\end{proof}

As Longhi remarks (\cite{Lon02}), we may compute a multiplicative
integral on $\cL^*$ by means of fixing a vertex $v_0\in\cT_K(\cL)$
and defining $l_{v_0}(e)$ as the number of intermediate vertices
between $v_0$ and $e$ in a previously fixed model for $\cT_K(\cL)$.
Then we have
$$
\mint_{\cL^*}{fd\mu}=\lim_{n\rightarrow\infty}{\ds{\prod_{\substack{l_{v_0}(e)=n\\t_e\in\cB(e)}}{f(t_e)^{\mu(e)}}}}
$$

\begin{thm}[Poisson Formula]\label{PF}
Let $u\in\mathcal{O}(\Omega_\cL)^*$ and $z_0\in\Omega_\cL$. Then, for any $z\in\Omega_\cL$ the next identity is satisfied:
$$
\frac{u(z)}{u(z_0)}=\mint_{z-z_0}{d\tilde{\mu}(u)}
$$
\end{thm}
\begin{proof}
We follow the proof of \cite[Thm.~6]{Lon02}.

The partial products
$$
\prod_{\substack{l_{v_0}(e)=N\\t_e\in\cB(e)}}{f_{z-z_0}\left(t_e\right)^{\tilde{\mu}(u)(e)}}
$$
converge uniformly on $\Omega_\cL$ so the integral built with them is a nowhere vanishing analytic function of $z$.
Since by the previous theorem the kernel of $\tilde{\mu}$ is $\ds{K^*}$, in order to prove the identity it is enough to see that $\ds{\tilde{\mu}(u(z))=\tilde{\mu}\left(\mint_{\cL^*}{f_{z-z_0}(t)d\tilde{\mu}(u)(t)}\right)}$. Further, note that
$$
f_{z-z_0}(t_e)=f_{z-z_0}(t_e)/f_{z-z_0}(b_0)=\frac{\tilde{z}(\tilde{t_e})}{\tilde{z_0}(\tilde{t_e})}\frac{\tilde{z_0}(\tilde{b_0})}{\tilde{z}(\tilde{b_0})}=\frac{\tilde{z}(\tilde{t_e})}{\tilde{z}(\tilde{b_0})}\frac{\tilde{z_0}(\tilde{b_0})}{\tilde{z_0}(\tilde{t_e})}=c\cdot\omega_{\tilde{t_e}-\tilde{b_0}}(z),
$$
$$
c\in {K(\tilde{z_0})}^*
$$

Therefore we have $\ds{\tilde{\mu}(f_{z-z_0}(t_e))=\mu_{b_0,t_e}}$ also by the previous theorem. Then, by the commutativity of $\tilde{\mu}$ and limits we obtain
$$
\tilde{\mu}\left(\mint_{\cL^*}{f_{z-z_0}(t)d\tilde{\mu}(u)(t)}\right)=\tilde{\mu}\Big(\lim_{N\rightarrow\infty}{\ds{\prod_{\substack{\l_{v_0}(e)=N\\t_e\in\cB(e)}}{f_{z-z_0}\left(t_e\right)^{\tilde{\mu}(u)(e)}}}}\Big)=
$$
$$
=\lim_{N\rightarrow\infty}{\ds{\sum_{\substack{\l_{v_0}(e)=N\\t_e\in\cB(e)}}{\tilde{\mu}(u)(e) \tilde{\mu}\left(f_{z-z_0}(t_e)\right)}}}=\lim_{N\rightarrow\infty}{\ds{\sum_{\substack{\l_{v_0}(e)=N\\t_e\in\cB(e)}}{\tilde{\mu}(u)(e)\mu_{b_0,t_e}}}}
$$

Let us evaluate on an edge $e'$ of a fixed model of $\cT_K(\cL)$. We
may assume $e'$ points away from $b_0$, so
$b_0\in\cB(\overline{e'})$. We have $e'\in P(b_0^*,t_e^*)$ if and
only if $t_e\in\cB(e')$, so we get
$$
\tilde{\mu}\left(\mint_{\cL^*}{f_{z-z_0}(t_e)d\tilde{\mu}(u)(t)}\right)(e')=\lim_{N\rightarrow\infty}{\ds{\sum_{\substack{l_{v_0}(e)=N\\t_e\in\cB(e)}}{\tilde{\mu}(u)(e)
\mu_{b_0,t_e}(e')}}}=
$$
$$
=\lim_{N\rightarrow\infty}{\ds{\sum_{\substack{\l_{v_0}(e)=N\\t_e\in\cB(e)\cap\cB(e')}}{-\tilde{\mu}(u)(e)}}}=\tilde{\mu}(u)(e')
$$
where the last equality is due to harmonicity applied to the sum
independent of $N\geq l_{v_0}(e')$. \end{proof}

\begin{cor}[Extended Poisson Formula]\label{EPF}
Take $u\in\mathcal{O}(\Omega_\cL)^*$. Then, given any degree 0
divisor $D=\sum{m_p p}$ of $\Omega_\cL$, we have
$$
\prod_{p\in\Supp(D)}{u(p)^{m_p}}=\mint_{D}{d\tilde{\mu}(u)}
$$
\end{cor}

\begin{cor}\label{exh}
The morphism
${\tilde{\mu}:\cO(\Omega_\cL)^*\longrightarrow\cM(\cL^*,\bZ)_0}$ is
surjective and for each $z_0\in\Omega_\cL$ it has a section
${\cI_{z_0}:\cM(\cL^*,\bZ)_0\longrightarrow\cO(\Omega_\cL)^*}$. As a
consequence we get a (non-unique, non-canonical) isomorphism
${\cO(\Omega_\cL)^*\cong K^*\times\cM(\cL^*,\bZ)_0}$.
\end{cor}
\begin{proof}
Let us take an harmonic measure $\mu\in\cM(\cL^*,\bZ)_0$. Let $z_0\in\Omega_\cL$ be any point. Then, as along the proof of the Poisson formula, we see that the function
$$
\cI_{\mu,z_0}(z):=\mint_{z-z_0}{d\mu}
$$
is analytic on $\Omega_\cL$, and once more, the same steps with $\mu$ instead of $\tilde{\mu}(u)$ prove that ${\tilde{\mu}(\cI_{\mu,z_0})=\mu}$. Then, we define the section by $\cI_{z_0}(\mu):=\cI_{\mu,z_0}$ and we chech that it is a morphism of groups:
$$
\cI_{z_0}(\mu+\mu')(z)=\mint_{z-z_0}{d(\mu+\mu')}=\mint_{z-z_0}{d\mu}\mint_{z-z_0}{d\mu'}=\left(\cI_{z_0}(\mu)\cI_{z_0}(\mu')\right)(z)
$$
Finally, by theorem~\ref{FMeas} we got the short exact sequence
$$
0\longrightarrow K^*\longrightarrow\cO(\Omega_\cL)^*\longrightarrow\cM(\cL^*,\bZ)_0\longrightarrow0
$$
which, with the section morphism, gives the asserted isomorphism by elementary homological algebra.
\end{proof}




\section{Schottky Groups and their limit sets}\label{schottky}

Along this section we recall Schottky groups and their main properties, and we build the Mumford curve for which we want to give its Jacobian, and its associated graph. The main novelty is the ``Berkovich analytification'' of some results in \cite{GvdP80}.

Given any $\gamma\in \PGL_2(K)$, we say that $\gamma$ is hyperbolic
if the (two) eigenvalues of $\gamma$ have two distinct absolute
values. Note that in this case, due to the completeness of $K$, the
eigenvalues are in $K$. Hence a $\gamma\in\PGL_2(K)$ is hyperbolic
if and only if it is conjugated to an element of $\PGL_2(\cO_K)$
represented by a matrix
$\left(\begin{array}{cc}q&0\\0&1\end{array}\right)$ with $q\in K$,
$|q|<1$ (look at \cite[Ch.~1~Lem.~I.1.4]{GvdP80}). From this we get
that if $\gamma$ is hyperbolic,
$$
\{x\in{\bP^1}^*(\bC_K)|\ \gamma x=x\}\subset{\bP^1}^*(K).
$$

Given any subgroup $\Gamma\subset \PGL_2(K)$, we denote by
$\cL_{\Gamma}$ the set of limit points of $\Gamma$ in the dual projective line, i.e. the set of
points $x\in {\bP^1}^*(\bC_K)$ such that there exists an infinite set
$\{\gamma_n\}_{n\in \bZ_{\ge0}}\subset \Gamma$ and $y\in {\bP^1}^*(\bC_K)$
with $\lim_{n\to \infty} \gamma_n\cdot y=x$. Observe that this set is
closed, and it contains the set $\displaystyle{\Sigma_{\Gamma^{*}}}$ of the points
$x\in {\bP^1}^*(\bC_K)$ such that there exists $\gamma\in \Gamma$, not of
finite order, satisfying $\gamma\cdot x=x$ (since $x=\lim_{n\to \infty}
\gamma^n\cdot x$). Observe also that $\Gamma$ acts on $\cL_{\Gamma}$.

Recall that a subgroup $\Gamma\subset \PGL_2(K)$ is discontinuous if
the set of limit points $\cL_{\Gamma} \ne {\bP^1}^*(\bC_K)$, and for any $p\in
{\bP^1}^*(\bC_K)$, the closure of the orbit $\overline{\Gamma p}$ is
compact.

A subgroup $\Gamma\subset \PGL_2(K)$ is a Schottky group if it is
discontinuous, torsion free (so all its elements
$\gamma\neq1_\Gamma$ are hyperbolics) and finitely generated. Then
$\Gamma$ is a free group of finite rank $g(\Gamma)$.

The following lemma is well known, but we didn't find an explicit
reference.

\begin{lem} Let $\Gamma$ be a Schottky group. Then
\begin{enumerate}
\item If $g(\Gamma)=1$, so $\Gamma=<\gamma>\cong \bZ$, then
$\cL_{\Gamma}=\{y_0,y_1\}=\Sigma_{\Gamma^*}$.
\item If $g(\Gamma)>1$, so $\Gamma$ is not abelian, then
$\cL_{\Gamma}$ is compact, perfect (without isolated points),
$\cL_{\Gamma}=\overline{\Gamma P}$ if $P\in \cL_{\Gamma}$ and
$\cL_{\Gamma}=\overline{\Gamma P}\setminus \Gamma P$ if $P\not\in
\cL_{\Gamma}$.
\item The set $\cL_\Gamma$ always has at least two points.
\item In any case, $\cL_{\Gamma}=\overline{\Sigma_{\Gamma^*}}$.
\end{enumerate}
\end{lem}
\begin{proof}
The three first claims are proved at paragraphs 1.5 and 1.6 of \cite[Ch.~1]{GvdP80}. Then we have $\overline{\Sigma_{\Gamma^*}}\subset\cL_\Gamma=\overline{\Gamma P}$ for any $P\in\Sigma_{\Gamma^*}$. But the set of fixed points it is clearly $\Gamma$-invariant, since if we have $\gamma_xx=x$ then $(\gamma\gamma_x\gamma^{-1})\gamma x=\gamma x$ for any $\gamma\in\Gamma$, so $\cL_\Gamma=\overline{\Gamma P}\subset\overline{\Sigma_{\Gamma^*}}$.
\end{proof}

We will need the two notions we will recall now. A half-line in a tree $\cT$ is an infinite subtree whose topological realization is homeomorphic to $[0,+\infty)$. Two half-lines are equivalent if they differ in a finite subgraph of the union. An end of $\cT$ is an equivalence class of half-lines.

\begin{lem}\label{quoGraph}
A Schottky group $\Gamma$ acts freely on $\cT_K(\cL_{\Gamma})$ (with the induced left action by $\PGL_2(K)$ on $\cT_K$),
and the quotient $G_{\Gamma}:=\Gamma\backslash\cT_K(\cL_{\Gamma})$ is a
finite metric graph.
Moreover, if $\cL'\subset {\bP^1}^*(K)$ is the union of $\cL_\Gamma$ and a finite set of orbits of points by the action of $\Gamma$, then there exists a finite connected graph $G_{\cL'}$ such that
$$
G_\Gamma\subset G_{\cL'}\subset\Gamma\backslash\cT_K(\cL')\text{ and }(\Gamma\backslash\cT_K(\cL'))\setminus G_{\cL'} = \bigsqcup_{\mathcal{R}_{\cL'}}{(0,+\infty)}
$$
where $\mathcal{R}_{\cL'}=\Gamma\backslash(\cL'\setminus\cL_\Gamma)$ is a finite set.
\end{lem}
\begin{proof}
The fact that $\Gamma$ acts freely on $\cT_K(\cL_{\Gamma})$ is a
consequence of all its non-neutral elements are hyperbolic with two
fixed points at the ends of the tree.

For the rest of the proof, we
are inspired by the proof given in \cite[Ch.~1~Lem.~(3.2)]{GvdP80}.
Let $B_\Gamma$ be a finite set of generators of $\Gamma$ and their
inverses containing the identity $\id_\Gamma$ too. Take $w\in
\cT_K(\cL_\Gamma)$ and a finite subtree
$\mathfrak{T}_w\subset\cT_K(\cL_\Gamma)$ containing $B_\Gamma\cdot
w$. Then,
$$
\mathfrak{T}=\bigcup_{\gamma\in\Gamma}{\gamma\cdot\mathfrak{T}_w}
$$
is a subtree of $\cT_K(\cL_\Gamma)$. The only thing we have to
verify is that it is connected, that is given
$\gamma,\gamma'\in\Gamma$ and $p\in\gamma\cdot\mathfrak{T}_w$,
$p'\in\gamma'\cdot\mathfrak{T}_w$ there exists a path in
$\mathfrak{T}$ between $p$ and $p'$. Through operating by $\gamma'$
on the path, we may suppose $\gamma'=\id_\Gamma$. Also, by an
induction process it is enough to show this when $\gamma\in
B_\Gamma$. So, with these hypotheses, we have $p'$ and $\gamma w$
connected by a path in $\mathfrak{T}_w$, and $\gamma w$ and $p$
connected by a path in $\gamma\cdot \mathfrak{T}_w$.

Now we will show $\mathfrak{T}=\cT_K(\cL_\Gamma)$, from what we will
get consequently the finiteness of the quotient.

Let $v$ be any vertex of $\cT_K(\cL_\Gamma)$ and consider a
half-line through $v$ starting at $w$, whose end corresponds to a
limit point $z\in\cL_\Gamma$ since $\cL_\Gamma$ is compact.
Therefore, there exists a sequence
$\{\gamma_n\}_{n\in\bN}\subset\Gamma$ with $\gamma_0=\id_\Gamma$
such that for any $z_0\in{\bP_K^1}^*\setminus\cL_\Gamma$,
$\lim_{n\to\infty}{\gamma_nz_0}=z$ (we may assume that the half-line
considered has as end a fixed point for some $\gamma\in\Gamma$ and
take the sequence of powers of $\gamma$ or $\gamma^{-1}$). Then the
fragments $P(\gamma_n w,\gamma_{n+1}w)$ belong to $\mathfrak{T}$,
and they form the unique half-line starting at $w$ in the direction
$z$, so $v\in\mathfrak{T}$.

For the second part, recall that
$\cT_K(\cL_\Gamma)\subset\cT_K(\cL')$ and that we have the
retraction map
$${\red_{\cL_\Gamma}:\Omega_{\cL_\Gamma}\longrightarrow\cT_K(\cL_\Gamma)}.$$
Choose a $p\in\Omega_{\cL_\Gamma}$ such that $\Gamma\cdot p$ is one
of the orbits added to $\cL_\Gamma$ to form $\cL'$. Take the open
path $L_p:=\mathring{P}(\red_{\cL_\Gamma}(p),p)$ and then observe
that $L_p\cap\cT_K(\cL_\Gamma)=\emptyset$. Now it is clear that
$$
\Gamma\backslash\cT_K(\cL') = G_\Gamma \bigsqcup\left(\bigcup_{\pi_\Gamma(p)\in \mathcal{R}_{\cL'}}{\pi_\Gamma(L_p)}\right)
$$
but the $\pi_\Gamma(L_p)$ have not to be disjoint. Nevertheless,
note that for any $\gamma\in\Gamma\setminus\{\id_\Gamma\}$ the
intersection $L_{\gamma p}\cap L_p$ is empty, since otherwise,
$\red_{\cL_\Gamma}(p)$ would be a fixed vertex for $\gamma$, which
contradicts the first claim of the result. Take now another
$q\in\Omega_{\cL_\Gamma}$ such that
$\pi_\Gamma(q)\in\mathcal{R}_{\cL'}$ and
$\pi_\Gamma(q)\neq\pi_\Gamma(p)$. It may happen that for some
$\gamma\in\Gamma$ (by the previous consideration, for at most one
$\gamma$) we have $L_p\cap L_{\gamma q}\neq\emptyset$. In that case,
in which $\red_{\cL_\Gamma}(p)=\red_{\cL_\Gamma}(\gamma q)$, let
$v_{pq}$ be the vertex of valence 3 in the tree $L_p\cup L_{\gamma
q}$. Next, let $v_p$ be one vertex of $L_p$ such that all the
possible $v_{pq}$ with $\pi_\Gamma(q)\in\mathcal{R}_{\cL'}$ are in
the path $P(\red_{\cL_\Gamma}(p),v_p)$. Finally take
$$
G_{\cL'}:=\Gamma\backslash\left(\cT_K(\cL_\Gamma)\bigcup_{\pi_\Gamma(p)\in\mathcal{R}_{\cL'}}{P(\red_{\cL_\Gamma}(p),v_p)}\right)
$$
and the claim is immediate.
\end{proof}


\begin{thm} Let $\Gamma$ be a Schottky group and consider $\cL:=\cL_{\Gamma}$
and $\Omega:=\Omega_\cL=({\bP^1}^*)^{an}\setminus \cL$. Then $\Gamma$
acts on $\Omega$ and $C_{\Gamma}:=\Gamma\backslash \Omega$ is a proper
analytic space and so it is isomorphic to the analytification of a smooth
projective algebraic curve of genus $g(\Gamma)$.
\end{thm}
\begin{proof}
You can see the proof with more detail in \cite[Ch.~2~and~3]{GvdP80}. Here, we will sketch it.

We will suppose that $G_{\Gamma}$ has a model without loops. This is
possible after a finite extension of the base field, if necessary.
The general case can be done by means of Galois descent.

We consider the projection $\pi_\Gamma:\cT_K(\cL)\lra G_\Gamma$ and
a metric graph model for $\cT_K(\cL)$ given by a pair of sets
($V,E)$. The collection of vertices $V$ is formed by points of the
form $t(x_0,x_1,x_2)$ for $x_0,x_1,x_2 \in \bP^1(K)$ such that it
includes all the points of valency greater than $2$, it is
$\Gamma$-invariant and the metric graph model for $G_\Gamma$ given
by $\pi_\Gamma(V)$ has no loops. Recall that the set of open edges
for the model of $\cT_K(\cL)$ is the set of connected components of
$\cT_K(\cL)\setminus V$, and the edges are obtained from the open
ones adjoining the adherent vertices. We will denote this set by
$E$.

Consider now the restriction  to $\Omega_\cL$ of the retraction map,
that is ${\red_\cL:\Omega_\cL\longrightarrow\cT_K(\cL)}$. To each
$e\in E$, we take $U(e):=\red_{\cL}^{-1}(e)$, and, similarly, to a
vertex $v\in V $ we take $U(v):=\red_{\cL}^{-1}(v)$. Then, the sets
$U(e)$ and $U(v)$ are strictly affinoid and from them we get back
$\Omega$ by gluing $U(e)$ with $U(e')$ through $U(v)$ when the edges
$e,e'$ have $v$ as a common vertex.

Since the retraction map $\red_{\cL}$ is $\Gamma$-equivariant, given two
edges $e , e'\in E$ such that $\pi_\Gamma(e)=\pi_\Gamma(e')$ so
there exists $\gamma\in \Gamma$ such that $\gamma\cdot e=e'$, then
$\gamma\cdot U(e)=U(e')$, and similarly for vertices. Therefore, gluing
as before but taking into account these identifications, or what is
the same, gluing according to the graph $G_\Gamma$ we get the
analytic space $C_\Gamma$, which is reduced and separated.

To prove that $C_\Gamma$ is proper we are going to show that it is compact and
its boundary (over $K$) is empty (\cite[Def.~4.2.13.~(ii)]{Tem15}).

The compactness is because we can express $C_\Gamma$  as a finite
union of affinoids: the preimages of the stars of the vertices of
$G_\Gamma$, which is a finite set.

To show that the boundary is empty, take any $x\in C_\Gamma$. We
want to show there exists $x\in U$ affinoid such that
$x\not\in\partial U$. Consider the image of $x$ by the induced
retraction map in the quotients,
$$
\red_{\cL,\Gamma}:C_\Gamma\longrightarrow G_\Gamma.
$$
Now, $\red_{\cL,\Gamma}(x)$ is an interior point of a $\Star(v)$ for
some vertex $v$ in the fixed model of $G_\Gamma$ (if
$\red_{\cL,\Gamma}(x)$ is a vertex we take $v=\red_{\cL,\Gamma}(x)$;
otherwise $v$ is any vertex of the edge to which
$\red_{\cL,\Gamma}(x)$ belongs). Then,
$\red_{\cL,\Gamma}^{-1}(\Star(v))$ is the affinoid we are looking
for.

Consider the following commutative diagram:
$$
\xymatrix{
\Omega\ar^{\displaystyle{\red_\cL}}[rr]\ar^{\displaystyle{\pi_\Gamma}}[dd]&&\cT_K(\cL)\ar^{\displaystyle{\pi_\Gamma}}[dd]\\
&&\\
C_\Gamma\ar^{\displaystyle{\red_{\cL,\Gamma}}}[rr]&& G_\Gamma
}
$$
Choose a vertex $\tilde{v}$ in $\cT_K(\cL)$ such that $\pi_\Gamma(\tilde{v})=v$. Then $\pi_\Gamma$ gives an isomorphism
$$
\pi_\Gamma:\Star(\tilde{v})\stackrel{\sim}\longrightarrow \Star(v)
$$
since there are no loops in $G_\Gamma$ and the action of $\Gamma$ in $\cT_K(\cL)$ is free. It is clear that
$$
\red_\cL^{-1}(\Star(\tilde{v}))=\bigcup_{\tilde{v}=s(e)} U(e)
$$
and hence, by construction of $C_\Gamma$, $\pi_\Gamma$ also induces an isomorphism
$$
\pi_\Gamma:\red_\cL^{-1}(\Star(\tilde{v}))\stackrel{\sim}\longrightarrow \red_{\cL,\Gamma}^{-1}(\Star(v))
$$
Now recall that $\partial U(e)=\{s(e),t(e)\}\subset U(e)$, since
$U(e)$ is an annulus, therefore
$$\partial (\red_\cL^{-1}(\Star(\tilde{v})))=\{t(e)|\ s(e)=\tilde{v}\}.$$
So we get $\partial
(\red_{\cL,\Gamma}^{-1}(\Star(v)))=\{\pi_\Gamma(t(e))|\
s(e)=\tilde{v}\}\not\ni x$ as we wished.
\end{proof}

\begin{cor}\label{corexh}
If there exists a model of $G_\Gamma$ which is without loops, then
the map $\Omega_\cL(K)\longrightarrow C_\Gamma(K)$ is surjective.
\end{cor}
\begin{proof}
 Choose such a model. By the previous proof we have
 $$
 C_\Gamma(K)=\bigcup_{e\in E(G_\Gamma)}{U(e)(K)},\qquad\Omega_\cL(K)=\bigcup_{\tilde{e}\in E}{U(\tilde{e})(K)}
 $$
 with the same notation. We may assume $\pi_\Gamma(\tilde{e})=e$ so we conclude $U(\tilde{e})=U(e)$.
\end{proof}




\section{The Jacobian of a tropical graph via integration}\label{tropical}

We give a proof of \cite[Thm.~6.4~(2)]{vdP92} from the different
perspective given by multiplicative integrals. This result was
generalized by Baker and Rabinoff in \cite[Thm.~2.9]{BR15}.

Recall the definition of the Jacobian of a finite metric graph (or
more generally, of a tropical curve) (see, for example
\cite[Def.~4.1.4]{CV10}). We only consider metric graphs $G$ with
all vertices of valence greater than or equal to $2$. By the
introduction of section~\ref{graphs}, for any edge $e$ of $G$ we
have a length $\ell(e)\in \bR_{>0}$.

We choose an orientation for each edge of $G$, and we consider the
free abelian group $\bZ[E(G)]$ generated by the oriented edges of
some model for $G$ and the map $\partial:\bZ[E(G)] \to\bZ[V(G)]$
given by $\partial(e)=t(e)-s(e)$, where $t(e)$ is the target of $e$
and $s(e)$ is the source. Then $H_1(G,\bZ)=\Ker(\partial)$. The
following result is well known.

Let $G$ be a metric graph. Consider the paring $(\ ,\
)_G:\bZ[E(G))]\times \bZ[E(G))]\to \bR$ defined by $(e,e')=0$ if
$e'\ne e$ and $e'\ne \overline{e}$ (the opposite edge of $e$),
$(e,e)=\ell(e)$ and $(e,\overline{e})=-\ell(e)$.

\begin{lem}  The pairing $(\ ,\ )_G$
determines a symmetric positive definite bilinear map $(\ ,\
)_G:H_1(G,\bZ)\times H_1(G,\bZ)\to \bR$.
\end{lem}

The Jacobian of $G$ is the torus given by $H_1(G,\bR)/H_1(G,\bZ)$
together with the metric determined by $(\ ,\ )_G$.

Now, suppose $\Gamma$ is a Schottky group in $\PGL_2(K)$, and
$\cL:=\cL_{\Gamma}$ is its set of limit points. For any $\gamma\in\Gamma$ we denote by $\ds{\lim_{n\rightarrow\infty}{\gamma^n z_0}=:z_\gamma^+\in\cL}$ and $\ds{\lim_{n\rightarrow-\infty}{\gamma^n z_0}=:z_\gamma^-\in\cL}$ its attractive and repulsive fixed points respectively.

Further, $\Gamma$ acts
on $\cT=\cT_K(\cL)$, and there is just an apartment fixed by $\gamma$ which is $\bA_{\{z_\gamma^-,z_\gamma^+\}}$. We will denote it by $\bA_\gamma$.

The quotient $G_\Gamma:=\Gamma\backslash\cT$ is a finite metric
graph. Moreover $\pi:\cT\to\Gamma\backslash\cT$ is the universal covering,
and the fundamental group of $\Gamma\backslash\cT$ is canonically isomorphic to
$\Gamma$. Hence the map $\varpi:\Gamma \to H_1(G_\Gamma,\bZ)$
defined sending $\gamma\in \Gamma$ to
$\pi(P(\alpha,\gamma(\alpha))$, where $\alpha$ is any point of $\bA_\gamma$ (for simplicity, a vertex of any model)
and $P(\alpha,\gamma(\alpha))$ is the oriented path from $\alpha$ to
$\gamma(\alpha)$, determines an isomorphism between the
abelianization $\Gamma^{ab}$ and $H_1(G_\Gamma,\bZ)$.

We denote by $$(\ , \ )_{\Gamma}:\Gamma^{ab}\times \Gamma^{ab}\to
\bR$$ the bilinear map given by
$$(\gamma, \gamma'
)_{\Gamma}=(\varpi(\gamma),\varpi(\gamma'))_{\cT/\Gamma}.$$

\begin{lem}
Any finite metric graph $G$ satisfies $H_1(G,\bZ)\cong \HC(G,\bZ)$.
\end{lem}

\begin{proof}
Take a model $\mathfrak{G}$ for $G$. We want to prove
$H_1(G,\bZ)\cong \HC(\mathfrak{G},\bZ)$. Given a cycle
$\displaystyle{z=\sum_{e\in E(G)}{n_e\cdot e\in
H_1(\mathfrak{G},\bZ)\subset\bZ[E(G)]}}$, we associate to it an
harmonic cochain $c(z)$ defined by $c(z)(e):=n_e$ and $c(z)(\bar
e):=-n_e$ for any $e\in E(G)$. Reciprocally, for each harmonic
cochain $c$ we get a cycle $z_c:=\sum_{e\in E(G)}{c(e)\cdot e}$.
This correspondence defines the bijection.
\end{proof}

\begin{obs}
Note that any Schottky group $\Gamma$ acts on $\HC(\cT,\bZ)$ so that $\HC(\Gamma\backslash\cT,\bZ)\cong \HC(\cT,\bZ)^\Gamma$.
\end{obs}

\begin{thm}\label{Niso}
The map $\displaystyle{\mu: \Gamma^{ab} \to\mathscr{M}( \cL^*,\mb{Z})_{0}^\Gamma}$ defined by
$$
\mu_\gamma(\me):=\mu(\gamma)(\me):=\frac{(\pi(\me),\varpi(\gamma))_{\cT/\Gamma}}{l(\me)}.
$$
over a (topological) edge $\me$ is a natural isomorphism such that
for any $\gamma,\gamma'\in \Gamma$, we have
$$
(\gamma, \gamma')_{\Gamma}=-\log{\left|\mint\right|_{\gamma\alpha-\alpha}{d\mu_{\gamma'}}}
$$
where $\alpha\in \cT_K(\cL)$ is any point.
\end{thm}

\begin{proof}
The previous results together with section~\ref{graphs} give the composition of isomorphisms
$$
\xymatrix@R=.1pc{
\displaystyle{\Gamma^{ab}}\ar[r]_(.35){\displaystyle{\cong}}&\displaystyle{H_1(G_\Gamma,\bZ)}\ar[r]_(.45){\displaystyle{\cong}}&\displaystyle{\HC(G_\Gamma,\bZ)}\ar[r]_(.42){\displaystyle{\cong}}&\displaystyle{\HC(\cT_K(\cL_\Gamma),\bZ)^\Gamma}\ar[r]_(.55){\displaystyle{\cong}}&\displaystyle{\mathscr{M}( \cL^*,\mb{Z})_{0}^\Gamma}\\
\displaystyle{\gamma}\ar@{|->}[r]&\displaystyle{\varpi(\gamma)}\ar@{|->}[r]&\displaystyle{c(\varpi(\gamma))}\ar@{|->}[rr]&&\displaystyle{\mu(c(\varpi(\gamma)))}
}
$$
which assigns to $\gamma\in\Gamma^{ab}$ the harmonic cochain defined by
$$
\mu(c(\varpi(\gamma)))(\me)=c(\varpi(\gamma))(\me)=\frac{(\pi(\me),\varpi(\gamma))_{\cT/\Gamma}}{l(\me)}
$$

Since the set of points of valence greater than 2 in the path from
$\alpha$ to $\gamma'\alpha$ is finite (by corollary~\ref{locft}),
then we get the equality
$$
(\gamma, \gamma')_{\Gamma}=-\log{\left|\mint\right|_{\gamma\alpha-\alpha}{d\mu_{\gamma'}}}
$$
decomposing the path linearly, applying  the lemma~\ref{logint} and
the multiplicativity of the integral with respect to the path and
taking into account the definition of the map $\mu$.
\end{proof}




\section{The discrete cross ratio}

In this section we recall some results relating the cross ratio of 4
points in $\bP^1(\bC_K)$ with the tree they generate. Recall that,
given four points $a_1,a_2,z_1,z_2\in{\bP_K^1}^*(\bC_K)$, the cross
ratio is defined as
$$
\left(\begin{array}{c}
a_1:z_1\\
a_2:z_2
\end{array}\right)=\frac{(a_1-z_1)(a_2-z_2)}{(a_1-z_2)(a_2-z_1)}
$$
Note that formally
$$
\left(\begin{array}{c}
a_1:z_1\\
a_2:z_2
\end{array}\right)=\left(\begin{array}{c}
z_1:a_1\\
z_2:a_2
\end{array}\right)=\left(\begin{array}{c}
a_2:z_2\\
a_1:z_1
\end{array}\right)
$$
and given a fifth point $z_3\in{\bP_K^1}^*(\bC_K)$,
$$
\left(\begin{array}{c}
a_1:z_1\\
a_2:z_2
\end{array}\right)\left(\begin{array}{c}
a_1:z_2\\
a_2:z_3
\end{array}\right)=\left(\begin{array}{c}
a_1:z_1\\
a_2:z_3
\end{array}\right)
$$

The next lemma is known, at least the particular cases and when $K$
is local (\cite{MD73}, \cite{BDG04}), but we prefer to expose a
general and new proof using our results.

\begin{lem}\label{val}
Let $a_1,a_2,z_1,z_2\in{\bP_K^1}^*(\bC_K)$ be four points such that $a_1\neq a_2$ and $z_1\neq z_2$. Then
$$
\mathit{v}_K\left(\left(\begin{array}{c}
a_1:z_1\\
a_2:z_2
\end{array}\right)\right)=\left(\bA_{\{a_1,a_2\}},\bA_{\{z_1,z_2\}}\right)_{\cT_{\bC_K}}.
$$
\end{lem}
\begin{proof}
To begin, recall the definition of the first term,
$$
\mathit{v}_K\left(\left(\begin{array}{c}
a_1:z_1\\
a_2:z_2
\end{array}\right)\right)=-\log\left|\frac{(a_1-z_1)(a_2-z_2)}{(a_1-z_2)(a_2-z_1)}\right|.
$$
If $a_i=z_j$ for some $i,j$ it is clear that the valuation of the
cross ratio and the intersection pairing of the apartments are
identically $\pm\infty$ with the sign depending on the combination.
Next we will considerate the case in which the four points are
distinct.

Let us suppose first that one of the four points is $\infty$. By the absolute symmetry among them, we can put $z_2=\infty$. Then, on one hand we have
$$
\mathit{v}_K\left(\left(\begin{array}{c}
a_1:z_1\\
a_2:z_2
\end{array}\right)\right)=-\log\left|\frac{a_1-z_1}{a_2-z_1}\right|
$$
On the other hand we will compute the intersection of $\bA_{\{a_1,a_2\}}$ with $\bA_{\{z_1,\infty\}}$. Note that we may write $\alpha(z_1,r)$ with $r\in\bR_{>0}$ for the points of the second apartment. Let us assume without loss of generality that $|a_1-z_1|<|a_2-z_1|$, so we see that the intersection between the apartments goes from the point $\alpha(z_1,|z_1-a_1|)$ to the point $\alpha(z_1,|z_1-a_2|)$ and the distance between them, which is the length of the intersection, and it is the product of the pairing (with positive sign because the assumption), is
$$
\left|\log\frac{|a_2-z_1|}{|a_1-z_1|}\right|=-\log\left|\frac{a_1-z_1}{a_2-z_1}\right|=\mathit{v}_K\left(\left(\begin{array}{c}
a_1:z_1\\
a_2:z_2
\end{array}\right)\right)
$$
as we wanted to see.

To finish the proof we have to deal with the case in which none of the four points is $\infty$. Let us define the compact set $\cL':=\{a_1,a_2\}$ and the radii
$$
\begin{array}{l}
r_1:=d(z_1,\cL')=\mbox{min}(|z_1-a_1|,|z_1-a_2|)\mbox{, and}\\
r_2:=d(z_2,\cL')=\mbox{min}(|z_2-a_1|,|z_2-a_2|).
\end{array}
$$
Once more, we can do the assumption $r_1\leq r_2$ without loss of generality. We will consider three cases:

We suppose first $|a_1-a_2|\geq r_2\geq r_1$.

On one hand it can occur that there is an $i\in\{1,2\}$ such that
$r_1=|z_1-a_i|$ and $r_2=|z_2-a_i|$. Then, the starting and ending
points of the intersection between $\bA_{\{a_1,a_2\}}$ and
$\bA_{\{z_1,z_2\}}$ are $\alpha(a_i,r_1)$ and $\alpha(a_i,r_2)$
respectively (so the intersection pairing is the distance with
positive sign), or the intersection is empty or just a point if
$r_1=r_2$. Anyway,
$$
\left(\bA_{\{a_1,a_2\}},\bA_{\{z_1,z_2\}}\right)_{\cT_{\bC_K}}=d(\alpha(a_i,r_1),\alpha(a_i,r_2))=\left|\log\frac{r_2}{r_1}\right|=-\log\frac{r_1}{r_2}
$$
If $i=1$, $r_1\leq|z_1-a_2|\leq\max\{r_1,|a_1-a_2|\}$ so $|z_1-a_2|=|a_1-a_2|$, and $r_2\leq|z_2-a_2|\leq\max\{r_2,|a_1-a_2|\}$ so $|z_2-a_2|=|a_1-a_2|$. If $i=2$, the same computation gives a similar result. In any case we always get
$$
\mathit{v}_K\left(\left(\begin{array}{c}
a_1:z_1\\
a_2:z_2
\end{array}\right)\right)=-\log\left|\frac{(a_1-z_1)(a_2-z_2)}{(a_1-z_2)(a_2-z_1)}\right|=-\log\frac{r_1 |a_1-a_2|}{r_2 |a_1-a_2|}=-\log\frac{r_1}{r_2}.
$$

On the other hand, writing $\{i,j\}=\{1,2\}$ we have $r_1=|a_i-z_1|$ and $r_2=|a_j-z_2|$. We may assume $i=1$ and $j=2$. The starting and ending points of the intersection are $\alpha(a_1,r_1)$ and $\alpha(a_2,r_2)$. So we have
$$
\left(\bA_{\{a_1,a_2\}},\bA_{\{z_1,z_2\}}\right)_{\cT_{\bC_K}}=d(\alpha(a_1,r_1),\alpha(a_2,r_2))=
$$
$$
=d(\alpha(a_1,r_1),\alpha(a_1,|a_1-a_2|))+d(\alpha(a_2,r_2),\alpha(a_2,|a_1-a_2|))=-\log\frac{r_1 r_2}{|a_1-a_2|^2}
$$
(Note that if we assumed $i=2$ and $j=1$, the intersection pairing
would be minus the distance.)

 Further,
$r_2\geq|a_1-z_2|\geq\max\{|a_1-a_2|,r_2\}\geq|a_1-a_2|$ so
$|a_1-z_2|=|a_1-a_2|$ and identically $|a_2-z_1|=|a_1-a_2|$.
Therefore
$$
\mathit{v}_K\left(\left(\begin{array}{c}
a_1:z_1\\
a_2:z_2
\end{array}\right)\right)=-\log\left|\frac{(a_1-z_1)(a_2-z_2)}{(a_1-z_2)(a_2-z_1)}\right|=-\log\frac{r_1 r_2}{|a_1-a_2|^2}.
$$

In second place we suppose $r_2>|a_1-a_2|\geq r_1$. We can assume $r_1=|z_1-a_1|$. Let us observe that $r_2=|z_2-a_1|=|z_2-a_2|$. The starting and ending points of the intersection are $\alpha(a_1,r_1)$ and $\alpha(a_1,|a_1-a_2|)$, so
$$
\left(\bA_{\{a_1,a_2\}},\bA_{\{z_1,z_2\}}\right)_{\cT_{\bC_K}}=d(\alpha(a_1,r_1),\alpha(a_1,|a_1-a_2|))=-\log\frac{r_1}{|a_1-a_2|}
$$
(Note that if we assumed $r_1=|z_1-a_2|$, the distance would appear with a minus, and so we would get the inverse value.)

Since we have $|z_1-a_2|=|a_1-a_2|$ by an argument as above, we get
$$
\mathit{v}_K\left(\left(\begin{array}{c}
a_1:z_1\\
a_2:z_2
\end{array}\right)\right)=-\log\left|\frac{(a_1-z_1)(a_2-z_2)}{(a_1-z_2)(a_2-z_1)}\right|=
$$
$$
=-\log\frac{r_1 r_2}{r_2|a_1-a_2|}=-\log\frac{r_1}{|a_1-a_2|}.
$$

Finally, the third case is $r_1\geq r_2>|z_1-z_2|$. In this case the intersection of the apartments is empty so the intersection pairing of the apartments is zero, and since $|z_1-a_1|=|z_1-a_2|$ and $|z_2-a_1|=|z_2-a_2|$, the valuation of the cross ratio vanishes as well.
\end{proof}

\begin{cor}
Let $\cL\subset{\bP_K^1}^*(K)$ be a compact set with at least two points.
If $a_1,a_2,z_1,z_2$ are in $\cL$ or even in ${\bP_K^1}^*(K)$, the pairing can be done in $\cT_K$,
$$
\mathit{v}_K\left(\left(\begin{array}{c}
a_1:z_1\\
a_2:z_2
\end{array}\right)\right)=\left(\bA_{\{a_1,a_2\}},\bA_{\{z_1,z_2\}}\right)_{\cT_{K}}
$$
while if $a_1,a_2\in\cL$ and $z_1,z_2\in\Omega_{\cL}(\bC_K)$, we may restrict to $\cT_K(\cL)$:
$$
\mathit{v}_K\left(\left(\begin{array}{c}
a_1:z_1\\
a_2:z_2
\end{array}\right)\right)=\left(\bA_{\{a_1,a_2\}},P\big(\red_{\cL}(z_1),\red_{\cL}(z_2)\big)\right)_{\cT_K(\cL)}.
$$

\end{cor}




\section{A peculiar symmetry}\label{APS}

In this section we study some properties of the action of $\Gamma$
on $\cT_K$, a relation among the measures, and a symmetry among
multiplicative integrals which can be useful to generalize the well
known symmetry between theta functions.

Let $\Gamma\subset \PGL_2(K)$ be a Schottky group, and let
$\cL:=\cL_\Gamma\subset{\bP^1}^*(K)$ be its set of limit points. We
are going to show a new result which will led to a proof of the
symmetry of bilinear pairing defining the jacobian of the Mumford
curve $C_{\Gamma}$.

 We assume that $\Omega_\cL(K)\neq\emptyset$ and contains at least the
closures of two $\Gamma$-orbits of points. This is possible after a
finite extension of $K$, meanwhile $\cL$ remains invariant.

Let us define for any $p\in\Omega_\cL(K)$ the compact set
$\ds{\cL_p:=\cL\cup\overline{\Gamma\cdot p}\subset {\bP^1}^*(K)}$
and for any $\gamma,\delta\in\Gamma$ the analytic function
$$
u_{\gamma,\delta,p}(z):=u_{\gamma\delta p,\gamma p}(z)=\frac{z-\gamma\delta p}{z-\gamma p}\in\mathcal{O}(\Omega_{\cL_p})
$$
Consider now a point $q\in\Omega_{\cL_p}(K)$, that is $q\in\Omega_\cL(K)$ such that $\ds{\overline{\Gamma\cdot p}\cap\overline{\Gamma\cdot q}=\emptyset}$. Then, for any $\rho\in\Gamma$, applying the invariance of the cross ratio we obtain
$$
\frac{u_{\gamma,\delta,p}(\rho q)}{u_{\gamma,\delta,p}(q)}=\frac{u_{\gamma^{-1},\rho,q}(\delta p)}{u_{\gamma^{-1},\rho,q}(p)}
$$
Recall from the section on the Poisson formula the equality of measures $$\tilde{\mu}(u_{\gamma,\delta,p})=\tilde{\mu}(u_{\gamma\delta p,\gamma p})=\mu_{\gamma p^*,\gamma\delta p^*},$$ and then
$$
\frac{u_{\gamma,\delta,p}(\rho q)}{u_{\gamma,\delta,p}(q)}=\mint_{\cL_p^*}{f_{\rho q-q}(t)d\mu_{\gamma p^*,\gamma\delta p^*}}
$$
Therefore, putting together the two last ideas we have
$$
\mint_{\cL_p^*}{f_{\rho q-q}(t)d\mu_{\gamma p^*,\gamma\delta p^*}}=\frac{u_{\gamma,\delta,p}(\rho q)}{u_{\gamma,\delta,p}(q)}=\frac{u_{\gamma^{-1},\rho,q}(\delta p)}{u_{\gamma^{-1},\rho,q}(p)}=\mint_{\cL_q^*}{f_{\delta p-p}(t)d\mu_{\gamma^{-1} q^*,\gamma^{-1}\rho q^*}}
$$

For any $\delta\in\Gamma$, using theorem~\ref{Niso} one defines a
measure $\ds{\mu_\delta\in\mathscr{M}( \cL^*,\mb{Z})_{0}}$, while we
just defined, for each $\gamma\in\Gamma$, a measure $\ds{\mu_{\gamma
p^*,\gamma\delta p*}\in\mathscr{M}( \cL_p^*,\mb{Z})_{0}}$. Note that
$\cL^*\subset\cL_p^*$ and $\cT_K(\cL)\subset\cT_K(\cL_p)$. We
consider compatible models for these trees, meaning that the model
of $\cT_K(\cL_p)$ restricts to the model of $\cT_K(\cL)$.

\begin{prop}\label{measum}
With the above notations, for any edge $e$ of $\cT_K(\cL_p)$ and $\cT_K(\cL)$ we have
$$
\sum_{\gamma\in\Gamma}{\mu_{\gamma p^*,\gamma\delta p^*}}(e)=-\mu_\delta(e)
$$
and for any edge of $\cT_K(\cL_p)$ which is not inside $\cT_K(\cL)$, then
$$
\sum_{\gamma\in\Gamma}{\mu_{\gamma p^*,\gamma\delta p^*}}(e)=0.
$$
\end{prop}

In order to prove the proposition, we observe first that
$$
\sum_{\gamma\in\Gamma}{\mu_{\gamma p^*,\gamma\delta p^*}}(e)=\sum_{\gamma\in\Gamma}{\mu_{p^*,\delta p^*}}(\gamma^{-1}e)=\sum_{\gamma\in\Gamma}{\mu_{p^*,\delta p^*}}(\gamma e)=\sum_{\{\gamma\in\Gamma|\ \gamma e\in|P(p,\delta p)|\}}{\mu_{p^*,\delta p^*}}(\gamma e)
$$
(where the bars for $|P(p,\delta p)|$ mean that we are considering
just the underlying sets, without orientation) and we proceed by
steps. The first step, which is the main one, lies essentially on
the following lemma.

\begin{lem}
For any $\delta\in\Gamma$ and $p\in\Omega_\cL(K)$ we have
$|\bA_{\{p,\delta
p\}}|\cap|\bA_{\{\delta^2p,\delta^3p\}}|=\emptyset$ and
$\bA_{\{p,\delta
p\}}\cap\bA_{\{\delta^{-1}p,\delta^2p\}}=\bA_{\{p,\delta
p\}}\cap\bA_\delta\subset\bA_{\{p,\delta p\}}\cap\cT_K(\cL)$.
\end{lem}
\begin{proof}
Since $\delta$ is hyperbolic it has the form $\delta=\delta'\left(\begin{array}{cc}q&0\\0&1\end{array}\right)\delta'^{-1}$ with $|q|<1$. Consider $p':=\delta^{-1}p\in\Omega_\cL$. Then, if we prove the equalities of the lemma for $\left(\begin{array}{cc}q&0\\0&1\end{array}\right)$ and $p'$ instead of $\delta$ and $p$, then allowing $\delta$ act on the apartments we will get the claims. So we may assume $\delta=\left(\begin{array}{cc}q&0\\0&1\end{array}\right)$ with $|q|<1$. In particular, we have $\bA_\delta=\bA_{\{\infty,0\}}$.

And now, we want to show $|\bA_{\{p,q
p\}}|\cap|\bA_{\{q^2p,q^3p\}}|=\emptyset$. Let us observe that
$|q^3p|<|q^2p|<|qp|<p$, so
\begin{eqnarray*}
\bA_{\{p,q p\}}\cap\bA_{\{\infty,0\}}=P\left(\alpha(0,|p|),\alpha(0,|qp|)\right)\\
\bA_{\{q^2p,q^3p\}}\cap\bA_{\{\infty,0\}}=P\left(\alpha(0,|q^2p|),\alpha(0,|q^3p|)\right)
\end{eqnarray*}
Therefore, if the intersection $|\bA_{\{p,\delta p\}}|\cap|\bA_{\{\delta^2p,\delta^3p\}}|$ was non empty it should occur in $\bA_{\{\infty,0\}}$ since the total space is a tree, but it is clear that $$
P\left(\alpha(0,|p|),\alpha(0,|qp|)\right)\cap P\left(\alpha(0,|q^2p|),\alpha(0,|q^3p|)\right)=\emptyset,
$$
an so we get the first claim.

In order to obtain the second claim we will prove
$\left(\bA_{\{p,qp\}},\bA_{\{q^{-1}p,q^2p\}}\right)_{\cT_{\bC_K}}=\left(\bA_{\{p,qp\}},\bA_{\{\infty,0\}}\right)_{\cT_{\bC_K}}$.
Applying the lemma~\ref{val} we see it is enough to check that
$$
\mathit{v}_K\left(\left(\begin{array}{c}
p:q^{-1}p\\
qp:q^2p
\end{array}\right)\right)=
\mathit{v}_K\left(\left(\begin{array}{c}
p:\infty\\
qp:0
\end{array}\right)\right).
$$
So we compute:
$$
\mathit{v}_K\left(\left(\begin{array}{c}
p:q^{-1}p\\
qp:q^2p
\end{array}\right)\right)=-\log\frac{|p-q^{-1}p||qp-q^2p|}{|p-q^2p||qp-q^{-1}p|}=-\log\frac{|q^{-1}p||qp|}{|p||q^{-1}p|}=$$
$$=-\log\frac{|qp|}{|p|}=\mathit{v}_K\left(\left(\begin{array}{c}
p:\infty\\
qp:0
\end{array}\right)\right)
$$
\end{proof}

Next, and under the hypotheses of the previous lemma, it allows us
to subdivide the apartment $\bA_{p,\delta p}$ in three paths:
$\bA_{\{p,\delta p\}}=S_{p,\delta p}\cup I_{p,\delta p}\cup
T_{p,\delta p}$, where
\begin{eqnarray*}
S_{p,\delta p}=P(p,t(p,\delta p,\delta^{-1}p))\\
I_{p,\delta p}=P(t(p,\delta p,\delta^{-1}p),t(p,\delta p,\delta^2 p))\\
T_{p,\delta p}=P(t(p,\delta p,\delta^2 p),\delta p)
\end{eqnarray*}
Since the first part of the lemma tells that $|\bA_{\{\delta^{-1}p, p\}}|\cap|\bA_{\{\delta p,\delta^2p\}}|=\emptyset$, this implies that $|S_{p,\delta p}|\cap|T_{p,\delta p}|=\emptyset$, the intersections of the interior of the paths are empty and the paths are well defined subpaths of $\bA_{\{p,\delta p\}}$ with the same orientation.\\
The second part of the lemma implies that $I_{p,\delta
p}\subset\cT_K(\cL)$. With this tools, we proceed to get the next
step:

\begin{lem}
Let $e$ be an edge of $\cT_K(\cL_p)$ and consider the sets
\begin{eqnarray*}
\Gamma_S^e:=\{\gamma\in\Gamma|\ \gamma e\in |S_{p,\delta p}|\}\\
\Gamma_I^e:=\{\gamma\in\Gamma|\ \gamma e\in |I_{p,\delta p}|\}\\
\Gamma_T^e:=\{\gamma\in\Gamma|\ \gamma e\in |T_{p,\delta p}|\}
\end{eqnarray*}
so that we have the decomposition $\{\gamma\in\Gamma|\ \gamma
e\in|P(p,\delta p)|\}=\Gamma_S^e\sqcup\Gamma_I^e\sqcup\Gamma_T^e$.
Then:
\begin{enumerate}
\item There is a bijection $\Gamma_S^e\longleftrightarrow\Gamma_T^e$ which reverses the orientation of the edge in $\bA_{\{p,\delta p\}}$, that is, if $\gamma'$ corresponds to a $\gamma$ such that $\gamma e$ is in $S_{p,\delta p}$ with the same orientation, the edge $\gamma' e$ is in $T_{p,\delta p}$ with the opposite orientation.
\item If $e$ is not inside $\cT_K(\cL)$, then $\Gamma_I^e=\emptyset$
\end{enumerate}
\end{lem}
\begin{proof}

\begin{enumerate}
\item The bijection is defined by
\begin{eqnarray*}
\Gamma_S^e\longrightarrow\Gamma_T^e\\
\gamma\mapsto\ \delta\gamma\
\end{eqnarray*}
Thus, if the oriented edge $\gamma e$ is in
$$S_{p,\delta p}=P(p,t(p,\delta p,\delta^{-1}p))= P(p,\delta^{-1}p)\cap P(p,\delta p),$$
the oriented edge $\delta\gamma e$ is in
$$\delta P(p,\delta^{-1}p)\cap \delta P(p,\delta p)=P(\delta p, p)\cap P(\delta p,\delta^2 p)=T_{p,\delta p}.$$
In general, the orientation of $\gamma e$ with respect to $S_{p,\delta p}$ and $P(p,\delta p)$ is the same as the orientation of $\delta\gamma e$ with respect to $T_{p,\delta p}$ and $P(p,\delta p)$ so the opposite to the orientation of $\gamma e$. Clearly, the inverse map is $\gamma\mapsto\delta^{-1}\gamma$.
\item The result is clear from the remark previous to the lemma. If $e$ is not inside $\cT_K(\cL)$, there is no $\gamma e$ inside $\cT_K(\cL)$ for $\gamma\in\Gamma$, but $\Gamma_I^e=\{\gamma\in\Gamma|\ \gamma e\in |I_{p,\delta p}|\subset|\cT_K(\cL)|\}$ so $\Gamma_I^e=\emptyset$.
\end{enumerate}
\end{proof}

\begin{proof}[Proof of proposition~\ref{measum}]
Let us see first the second claim. If $e$ is not in $\cT_K(\cL)$ we have
$$
\sum_{\gamma\in\Gamma}{\mu_{\gamma p^*,\gamma\delta p^*}}(e)=\sum_{\{\gamma\in\Gamma|\ \gamma e\in|P(p,\delta p)|\}}{\mu_{p^*,\delta p^*}}(\gamma e)=
$$
$$
=\sum_{\gamma\in\Gamma_S^e}{\mu_{p^*,\delta p^*}}(\gamma e)+\sum_{\gamma\in\Gamma_I^e}{\mu_{p^*,\delta p^*}}(\gamma e)+\sum_{\gamma\in\Gamma_T^e}{\mu_{p^*,\delta p^*}}(\gamma e)
$$
Because of the second part of the previous lemma the second summation is zero and because of the first part and the definition of $\mu_{p^*,\delta p*}$ the sum of the other two summations vanishes, so $\sum_{\gamma\in\Gamma}{\mu_{\gamma p^*,\gamma\delta p^*}}(e)=0$ as we wanted to see.\\
We assume now that $e$ is in $\cT_K(\cL)$. We have the same equalities that before and also the cancellation of the two extreme summations so
$$
\sum_{\gamma\in\Gamma}{\mu_{\gamma p^*,\gamma\delta p^*}}(e)=\sum_{\gamma\in\Gamma_I^e}{\mu_{p^*,\delta p^*}}(\gamma e)
$$
and we want to prove this is equal to
$$
-\mu_\delta(e)=-\frac{(\pi(e),\varpi(\delta))_{\cT/\Gamma}}{\ell(e)}=-\frac{(\pi(e),\pi(P(\alpha,\delta\alpha)))_{\cT/\Gamma}}{\ell(e)}
$$
where $\cT=\cT_K(\cL)$) and $\alpha$ is any vertex in $\bA_\delta$.
We take $\alpha=t(p,\delta p,\delta^{-1}p)$, so we have
$\delta\alpha=t(p,\delta p,\delta^2p)$ and
$$
\mu_\delta(e)=\frac{(\pi(e),\pi(P(\alpha,\delta\alpha)))_{\cT/\Gamma}}{\ell(e)}=\sum_{\substack{|\gamma e|\subset |P(\alpha,\delta\alpha)|\\\gamma\in\Gamma}}{\frac{(\gamma e,P(\alpha,\delta\alpha))_\cT}{\ell(e)}}=
$$
$$
=\sum_{\gamma\in\Gamma_I^e}{\frac{(\gamma e,P(\alpha,\delta\alpha))_\cT}{\ell(e)}}=-\sum_{\gamma\in\Gamma_I^e}{\mu_{p^*,\delta p^*}}(\gamma e)=-\sum_{\gamma\in\Gamma}{\mu_{\gamma p^*,\gamma\delta p^*}}(e)
$$
where for the third equality we use the definition of $\alpha$ and the fact that the action of $\Gamma$ on $\cT$ is free, and for the fourth equality we use the definition of $\mu_{p^*,\delta p*}$.
\end{proof}

\begin{cor}
With the above notations we have
$$
\prod_{\gamma\in\Gamma}{\mint_{\cL_p^*}{f_{\rho q-q}^{-1}(t)d\mu_{\gamma p^*,\gamma\delta p^*}}}=\mint_{\cL^*}{f_{\rho q-q}(t)d\mu_{\delta}}
$$
\end{cor}
\begin{proof}
It is direct from the proposition, taking into account that the inverse of the function $f_{\rho q-q}$ appears due to the negative sign in the equality $\sum_{\gamma\in\Gamma}{\mu_{\gamma p^*,\gamma\delta p^*}}(e)=-\mu_\delta(e)$.
\end{proof}

\begin{cor}\label{APR}
Let $\Gamma\subset \PGL_2(K)$ be a Schottky group, and let $\cL:=\cL_\Gamma\subset{\bP^1}^*(K)$ be its set of limit points. For any $\rho,\delta\in\Gamma$ and for any $p,q\in\Omega_\cL(K)$ such that $\ds{\overline{\Gamma\cdot p}\cap\overline{\Gamma\cdot q}=\emptyset}$ we get
$$
\mint_{\rho q-q}{d\mu_{\delta}}=\mint_{\delta p-p}{d\mu_{\rho}}
$$
\end{cor}
\begin{proof}
Taking into account the last observation previous to the proposition and the corollary above we get
$$
\mint_{\rho q-q}{d\mu_{\delta}}=\mint_{\cL^*}{f_{\rho q-q}(t)d\mu_{\delta}}=\prod_{\gamma\in\Gamma}{\mint_{\cL_p^*}{f_{\rho q-q}^{-1}(t)d\mu_{\gamma p^*,\gamma\delta p^*}}}=
$$
$$
=\prod_{\gamma\in\Gamma}{\mint_{\cL_q^*}{f_{\delta p-p}^{-1}(t)d\mu_{\gamma^{-1} q^*,\gamma^{-1}\rho q^*}}}=\mint_{\cL^*}{f_{\delta p-p}(t)d\mu_{\rho}}=\mint_{\delta p-p}{d\mu_{\rho}}
$$
\end{proof}




\section{Automorphic Forms}\label{automorphic}

The main goal of this section is to prove theorem~\ref{AutTh} using
the analytic theory developed along this paper and some results of
\cite{BPR13}, like propositions 2.5, 2.10 and the slope formula
theorem (5.15), instead of using \cite[Ch.~2~(3.2)]{GvdP80}, whose
proof requires more analytic tools.

Let $G$ be a metric graph.
\begin{defn}
We call a tropical function on $G$ a continuous function
$f:G\longrightarrow\bR$ such that there exists a model
$\mathfrak{G}$ of $G$ that for each edge $e\in E(\mathfrak{G})$ the
restriction
$$
f_{\big|\vert e\vert}:\vert e\vert\longrightarrow\bR
$$
is linear with integral slope, where by linear we mean that for
every isometric embedding $[a,b]\longrightarrow|e|$, the composition
$[a,b]\longrightarrow|e|\longrightarrow\bR$ is linear.

Note that this is equivalent to say that for each model of $G$ the function $f$ is piecewise linear (with integral slopes) on each edge.
\end{defn}

Suppose now that $G$ is locally finite. Given a tropical function
$f$ on $G$ and a model $\mathfrak{G}$ of $G$ such that $f$ verifies
the ``edge-linearity'' condition stated on previous definition, we
can associate to it a cochain $D_f$ on the edges of $\mathfrak{G}$
defined by taking $D_f(e)$ to be the slope of $f$ on $e$.

We call $f$ a harmonic function if $D_f$ is an harmonic cochain.

\begin{obs}
If $f$ is harmonic, $f_{\big|\vert e\vert}$ is linear for any edge of any model of $G$.
\end{obs}

Next, let $\Gamma$ be a group with a left action on a metric graph $G$.
\begin{defn}
A tropical function $f$ on $G$ is called an automorphic form for $\Gamma$ if
$$
\forall\ \gamma\in\Gamma\ \exists\ c_f(\gamma)\in\bR:\ f(z)=c_f(\gamma)+f(\gamma z)\ \forall z\in G
$$
\end{defn}

\begin{lem}
Let $G$ be a locally finite metric graph on which acts a group
$\Gamma$. Let $f$ be an automorphic form for $\Gamma$. Then there
exists a model $\mathfrak{G}$ of $G$, on which acts $\Gamma$, such
that $f$ is linear on its edges, the cochain $D_f$ is
$\Gamma$-invariant and so induces a cochain $\overline{D_f}$ on
$\Gamma\backslash G$.
\end{lem}
\begin{proof}
Since $f$ is tropical there exists a model of $G$ such that $f$ is linear on its edges. Now, the minimal $\Gamma$-invariant model refining the previous satisfies the claims of the lemma immediately, and $D_f$ is $\Gamma$-invariant because $f$ is automorphic for $\Gamma$.
\end{proof}

\begin{lem}\label{BPR210}
Let $G$ be a locally finite metric graph on which acts a group $\Gamma$. Assume there exists a finite connected graph $G'\subset G/\Gamma$ such that
$$
(\Gamma\backslash G)\setminus G'=\bigsqcup_{i\in I}{L_i}\text{ where }I\text{ is finite and }L_i\cong(0,\infty)\ \forall i\in I
$$
such that its closure inside $\Gamma\backslash G$ is $\overline{L_i}\cong[0,\infty)$ (we are choosing an orientation on $L_i$).

Then, any harmonic function on $G$ being an automorphic form for $\Gamma$ verifies:
\begin{enumerate}
\item For any $i\in I$, the restricted cochain is constant: ${\overline{D_f}}_{|L_i}\equiv m_i\in\bZ$.
\item $\displaystyle{\sum_{i\in I}{m_i}=0}$.
\end{enumerate}
\end{lem}
\begin{proof}
We take a suitable model of $G$ -since $f$ is harmonic, it only has to be $\Gamma$-invariant-. Since $D_f$ is harmonic, so it is $\overline{D_f}$. Now, given two adjacent oriented edges $e,e'$ of $L_i$, due to the hypothesis on $G$ and $G'$ harmonicity implies $\overline{D_f}(e)+\overline{D_f}(\overline{e''})=0$, so $\overline{D_f}(e)=\overline{D_f}(e')$, and this extends obviously to any edge of $L_i$, so the first claim rests proved.

The second claim is a direct consequence of the lemma~\ref{harstar}.
\end{proof}

From now on, let $\Gamma$ be a fixed Schottky group, $\cL=\cL_\Gamma$ the set of fixed points of $\Gamma$, and $\Omega_\cL$ as defined above. Let $L|K$ be a field extension.
\begin{defn}
We will say that a $\ds{\bC_K}$-valued meromorphic function $f\neq0$ on $\ds{\Omega_{\cL}}$ is an automorphic form for $\Gamma$ with automorphy factor $\ds{c_f:\Gamma\longrightarrow \mb{C}_K^*}$ if
$$f(z)=c_f(\alpha)f(\alpha z)\ \forall z\in\Omega_\cL \forall \alpha\in\Gamma.$$
We will call it $L$-automorphic if $c_f$ takes values in $\ds{L^*}$.

Let us denote the set of $L$-automorphic forms on $\Omega_\cL$ by $\cA_\Gamma(L)$.
\end{defn}

\begin{obs}
By definition, $c_f$ is a group morphism.
\end{obs}

\begin{prop}\label{autMI}
Given a point $z_0\in\Omega_\cL(K)$ and a $\Gamma$-invariant measure ${\mu\in\cM(\cL^*,\bZ)_0^\Gamma}$ the function on $\Omega_\cL$
$$
\cI_{\mu,z_0}(z):=\mint_{z-z_0}{d\mu}
$$
is an analytic and automorphic form for $\Gamma$ with automorphy factor independent of $z_0$.
\end{prop}
\begin{proof}
We already know it is analytic, as shown in the proof of theorem~\ref{PF} and remarked in its corollary~\ref{exh}.

In order to see that it is automorphic for $\Gamma$ let us show first that the integral
$$
\mint_{ p -\gamma p}{d\mu}
$$
does not depend on $p\in\Omega_\cL$. Indeed, given $p,q\in\Omega_\cL$ we have
$$
\frac{\ds{\mint_{ p-\gamma p}{d\mu}}}{\ds{\mint_{q-\gamma q}{d\mu}}}=\frac{\ds{\mint_{p-q}{d\mu}}}{\ds{\mint_{\gamma p-\gamma q}{d\mu}}}=1
$$
due to the $\Gamma$-equivariance of the integration and to the $\Gamma$-invariance of $\mu$.

Therefore,
$$
\frac{\ds{\cI_{\mu,z_0}(z)}}{\ds{\cI_{\mu,z_0}(\gamma z)}}=\frac{\ds{\mint_{z-z_0}{d\mu}}}{\ds{\mint_{\gamma z-z_0}{d\mu}}}=\mint_{z-\gamma z}{d\mu}\in K^*
$$
is its automorphy factor.
\end{proof}

\begin{prop}\label{autexh}
For any $c\in \mathrm{Hom}(\Gamma^{ab},L^*)$ there exists an $L$-automorphic form $f$ such that $c=c_f$.
\end{prop}
\begin{proof}
Let us consider the group $\cM(\Omega_\cL)^*$ of non-zero
meromorphic functions on $\Omega_\cL$ and its quotient $Q$ by the
constants, so we have the short exact sequence
$$
0\longrightarrow L^*\longrightarrow\cM(\Omega_\cL)^*\longrightarrow Q\longrightarrow 0
$$
After taking invariants under $\Gamma$ we find the exact sequence
$$
\cM(C_\Gamma)\longrightarrow Q^\Gamma\longrightarrow \mathrm{Hom}(\Gamma^{ab},L^*)\longrightarrow H^1(\Gamma,\cM(\Omega_\cL)^*)
$$
We end the proof recalling that $H^1(\Gamma,\cM(\Omega_\cL)^*)=0$ by \cite[Cor.~5.3]{vdP92} -since $C_\Gamma$ is algebraic-, and noting that $Q^\Gamma$ coincides with the group of $L$-automorphic forms modulo the constants.
\end{proof}
We may express this telling that the morphism
$$
\cA_\Gamma(L)\longrightarrow \mathrm{Hom}(\Gamma^{ab},L^*)
$$
is surjective.

Let us formalize the notion of infinite divisor as in \cite[\S2]{MD73}.
\begin{defn} We call a function $\textbf{D}:\Omega_\cL(\bC_K)\longrightarrow\bZ$ an infinite $L$-divisor on $\Omega_\cL$ verifying the following properties:
\begin{itemize}
\item $\textbf{D}(z_1)=\textbf{D}(z_2)$ if $z_1=\Gamma z_2$.
\item The set $\Supp(D):=\{z\in\Omega_\cL|\ \textbf{D}(z)\neq0\}$ has no limit points in $\Omega_\cL$ and there
is a finite extension $L'|L$ such that $\Supp(D)\subset \Omega_\cL(L')$.
\end{itemize}
We write such a divisor in the usual form
$$
D=\sum_{n_z=\textbf{D}(z)\neq0}{n_zz}.
$$
\end{defn}
We will say that such an infinite divisor has finite representant $\tilde{D}$ if this is a finite divisor (that is it has finite support) such that
$$
D=\sum_{\gamma\in\Gamma}{\gamma\tilde{D}}=:\Gamma\tilde{D}
$$

We consider now the zeroes and poles of the automorphic forms. Note that if $z$ is a zero (resp. pole) of order $n$ of $f\in\cA_\Gamma$, for each $\gamma\in\Gamma$, $\gamma z$ is a zero (resp. pole) of order $n$ of $f$ too.
\begin{prop}
Let $f$ be a meromorphic function and $e$ an edge of a model of $\cT_K(\cL)$. Then, the set of zeroes and poles of $f$ restricted to $U(e)$ is finite.
\end{prop}
\begin{proof}
First, a meromorphic function is the quotient of analytic functions so we may assume that $f$ is analytic and we only have to show that it has a finite number of zeroes. But this is proved in \cite[Prop.~3.3.6]{FvdP04} as a consequence of the fact that the affinoid $U(e)$ is a disjoint union of closed discs, the Mittag-Leffler decomposition theorem and the Weierstrass preparation theorem.
\end{proof}

\begin{cor}\label{ftAut}
The set of zeros and poles of an automorphic form $f$ on $\Omega_\cL$ for $\Gamma$ is a finite union of orbits of points of $\Omega_\cL$.
\end{cor}
\begin{proof}
Consider a model for $\cT_K(\cL)$ and denote the set of its edges $E$. Consider also a set of
 edges $E_\Gamma\subset E$ in bijection by $\pi_\Gamma$ with the edges on the induced model on $G_\Gamma$. Since the quotient graph is
 finite so it is the set $E_\Gamma$, and since this is a set of representatives of the graph $G_\Gamma$,
 $$
 \bigcup_{\gamma\in\Gamma}\gamma\cdot E_\Gamma = E
 $$
 Therefore, the affinoids $\gamma U(E_\Gamma)$ with $\gamma\in\Gamma$ cover all $\Omega_\cL$, where
$$
U(E_\Gamma):=\bigcup_{e\in E_\Gamma}{U(e)}.
$$
Now, because of the previous proposition, the set $S_\Gamma(f)$ of zeroes and poles of $f$ on $U(E_\Gamma)$ is finite. And since this set is $\Gamma$-invariant and the orbit of $U(E_\Gamma)$ covers $\Omega_\cL$, the orbit of $S_\Gamma(f)$ is the set of zeroes and poles of $f$ and it is a finite union of orbits of points.
\end{proof}

Let us denote $S(f)$ the set of zeroes and poles of an automorphic form $f$ on $\Omega_\cL$, and $\cL_f:=\cL_\Gamma\cup S(f)$. The set $\cL_f$ is compact, due to the previous proposition and the fact that $\Gamma$ is a Schottky group.

Note that $f$ has neither zeroes nor poles on $\Omega_{\cL_f}$, so $f\in\cO(\Omega_{\cL_f})^*$.

\begin{thm}
Let $f$ be an automorphic form for $\Gamma$ on $\Omega_\cL$. Then $$F=-\log|f|_{|\cT_K(\cL_f)}$$ is a harmonic and automorphic form for $\Gamma$ on $\cT_K(\cL_f)$.
\end{thm}
\begin{proof}
The first thing we have to check is that $F$ is tropical, that is, given a model of $\cT_K(\cL_f)$ and an edge $e$ of this model, the restriction of $F$ on $|e|$ is piecewise linear on it.

Since we are going to apply lemma~\ref{PoiL}, we recall the notation used in it. We may suppose that the topological realization of the edge is $|e|=P(\alpha(x,r),\alpha(x,s))$ with $x\in\cL_f$, $r<s$ and such that its associated annulus satisfies $R_x(r,s)\cap\cL_f=\emptyset$. We also do not loss generality assuming $x=0$. Now we consider an isometric embedding
$$
\exp:[r_0,s_0]\longrightarrow P(\alpha(0, \exp(r_0)),\alpha(0,\exp(s_0)))\text{ where }r=\exp(r_0),\ s=\exp(s_0)
$$
By the cited lemma, we know that $|f(z)|=r|z^k|$ for some $r\in\bR_{>0}, k\in\bZ$ on that path, and $z=\exp(w)$ for $w\in[r_0,s_0]$. Therefore
$$
F(\exp(w))=-\log|f(z)|=-k\log|z|-\log(r)=-kw-\log(r),
$$
so we get the hoped linearity with integral slope $k$, and so $F$ becomes tropical.

In the previous computation we got $D_F(e)=-k$. Recall also the map
$$
\tilde{\mu}:\cO(\Omega_{\cL_f})^*\longrightarrow\mathscr{M}(\cL_f^*,\mb{Z})_{0}
$$
by which $\tilde{\mu}(f)(e)=k$. Therefore $D_F=-\tilde{\mu}(f)$, so this is a harmonic cochain and $F$ is harmonic.

Finally we will show that $F$ is automorphic for $\Gamma$ on $\cT_K(\cL_f)\subset\Omega_{\cL_f}\subset\Omega_\cL$. Since $f$ is automorphic on $\Omega_\cL$ we have that for all $z\in\Omega_\cL$ and $\gamma\in\Gamma$, $f(z)=c_f(\gamma)f(\gamma z)$ with $c_f(\gamma)\in\bC_K^*$. Let us restrict to the case when $z\in \cT_K(\cL_f)$:
$$
F(z)=-\log|f(z)|=-\log|c_f(\gamma)f(\gamma z)|=-\log|c_f(\gamma)|-\log|f(\gamma z)|=
$$
$$
=\mathit{v}_K(c_f(\gamma))+F(\gamma z)\text{ with }\mathit{v}_K(c_f(\gamma))\in\bR
$$
\end{proof}

We maintain the same hypothesis of the theorem. Consider now the
quotient $\Gamma\backslash\cT_K(\cL_f)$. By the
lemma~\ref{quoGraph}, its quotient is the disjoint union of a finite
connected graphs with a finite union of ``ends'' which correspond to
the classes of zeroes and poles of $f$ modulo $\Gamma$ -that is
$\Gamma\backslash S(f)$- by the definition of $\cL_f$. For any $x\in
S(f)$ denote $L_x$ the corresponding end oriented from the interior
to the exterior, like in lemma~\ref{BPR210}. With the previous
theorem, the next completes the slope formula (cf.
\cite[5.15]{BPR13}).

\begin{prop}
With the previous notation we get
$$
\overline{D_F}_{|L_x}\equiv o_x(f)
$$
\end{prop}
\begin{proof}
In order to know the value of $\overline{D_F}_{|L_x}$ we have to
evaluate $D_f$ on any edge $e$ of $L_x$. We can assume its
topological realization is of the form $P(\alpha(x,r),\alpha(x,s))$
with $r<s$. Note that, by the chosen orientation, we have
$\overline{D_F}_{|L_x}=D_f(\overline{e})=-D_F(e)$. Finally, by what
we have seen on the proof of the previous theorem or in
lemma~\ref{PoiL}, we get $D_F(e)=-o_x(f)$, so
$\overline{D_F}_{|L_x}=o_x(f)$.
\end{proof}

Next, we want to build a finite degree zero divisor associated to an
automorphic form on $\Omega_\cL$. In order to get this, we have to
refine the proof of corollary~\ref{ftAut}.

First, we note that there is a ``semi-open'' (connected) tree (open
at some edges, closed at others) in $\cT_K(\cL)$ in bijection with
$G_\Gamma=\Gamma\backslash\cT_K(\cL)$ by the projection map
$\pi_\Gamma$.

To see this, take a maximal tree $T_\Gamma$ of $G_\Gamma$ and a set $E_\Gamma^c$ of adjacent closed edges of $\cT_K(\cL)$ such that its topological realization $|E_\Gamma^c|$ is a tree in bijection with $T_\Gamma$ by $\pi_\Gamma$. Next take a set of open edges $E_\Gamma^o$ of $\cT_K(\cL)$ corresponding to the open edges which form $G_\Gamma\setminus T_\Gamma$, each one of them adjacent to some edge of $E_\Gamma^c$. Then we have that $|E_\Gamma^c\cup E_\Gamma^o|$ is a subtree of $\cT_K(\cL)$ in bijection with $\pi_\Gamma(E_\Gamma^c\cup E_\Gamma^o)=G_\Gamma$, as the one we claimed the existence.

Now take
$$
U(G_\Gamma):=\left(\bigcup_{e\in E_\Gamma^c}{U(e)}\right)\bigcup\left(\bigcup_{\mathring{e}\in E_\Gamma^o}{U(\mathring{e})}\right)
$$
By construction, for $\gamma'\in\Gamma\setminus\{\id_\Gamma\}$, we
have
$$U(G_\Gamma)\bigcap\left(\gamma'\cdot U(G_\Gamma)\right)=\emptyset\text{ and }\bigcup_{\gamma\in\Gamma}{\gamma\cdot U(G_\Gamma)}=\Omega_\cL.$$
Consider also the set $S_\Gamma(f)=S(f)\bigcap U(G_\Gamma)$ (note that in the proof of corollary~\ref{ftAut} we used the same notation but with a slightly different meaning, since $U(E_\Gamma)\neq U(G_\Gamma)$) and the finite divisor
$$
D_f^\Gamma:=\sum_{p\in S_\Gamma(f)}{o_p(f)p}
$$
By the previous remark on unions and intersections on the orbit of $U(G_\Gamma)$ and the structure of $S(f)$, we get that the divisor of $f$ satisfies
$$
\sum_{p\in S(f)}{o_p(f)p}=\sum_{\gamma\in\Gamma}{\gamma\cdot D_f^\Gamma}
$$

\begin{prop}
An automorphic form has associated an infinite divisor with finite representant of degree zero.
\end{prop}
\begin{proof}
Because of the previous considerations, the only we have to proof is that $D_f^\Gamma$ has degree zero, that is
$$
\sum_{p\in S_\Gamma(f)}{o_p(f)}=0
$$
Next note that there is a bijection between $S_\Gamma(f)$ and $\Gamma\backslash S(f)$. Further, by the previous theorem we have
$$
\sum_{p\in S_\Gamma(f)}{o_p(f)}=\sum_{p\in S_\Gamma(f)}{\overline{D_F}_{|L_p}}=\sum_{\pi_\Gamma(p)\in S(f)/\Gamma}{\overline{D_F}_{|L_p}}
$$
Finally, applying the lemma~\ref{BPR210} to the quotient $\cT_K(\cL_f)$, which has as ends the sets $L_p$ with $\pi_\Gamma(p)\in\Gamma\backslash S(f)$ by the lemma~\ref{quoGraph}, we get that this sum is zero, as we wanted to prove.
\end{proof}

In order to go in depth, let us take into consideration a special kind of automorphic forms: theta functions.\\
For any $p,p'\in\Omega_\cL(\bC_K)$, the infinite product
$$
\theta(p-p';z):=\prod_{\gamma\in\Gamma}{\frac{z-\gamma p}{z-\gamma p'}}
$$
defines a meromorphic function on $\Omega_\cL$, clasically called theta function.

Its construction and the properties we report are done in \cite[Ch.~2]{GvdP80}. It is an $L$-automorphic form for $\Gamma$, where $L|K$ is any complete extension of fields such that $p,p'\in\Omega_\cL(L)$. If $p$ and $p'$ are in the same $\Gamma$-orbit, the theta function is analytic. If $\Gamma p\neq\Gamma p'$, then $\theta(p-p';z)$ has simple zeroes at the points of $\Gamma p$ and simple poles at the points of $\Gamma p'$ and no other zeroes or poles. The previous considerations show us that $\theta(p-p';z)$ has associated an infinite divisor on $\Omega_\cL$, which is $\Gamma(p-p')=\Gamma p-\Gamma p'$. Further, for any $\delta\in\Gamma$ and $p\in\Omega_\cL$, the theta function $\theta(p-\delta p;z)$ does not depend on $p$.


Next we prove a simpler version of \cite[Ch.~2~(3.2)]{GvdP80}.
\begin{thm}\label{AutTh}
Let $f$ be an automorphic form on $\Omega_\cL$. There is a finite divisor $\sum_{i=1}^r{(p_i-q_i)}$ which represent the infinite divisor associated to $f$ and such that
$$
f(z)=\widetilde{f}(z)\cdot\theta(p_1-q_1;z)\cdots\theta(p_r-q_r;z)
$$
with $\widetilde{f}$ analytic function without zeroes on
$\Omega_\cL$. Further, if $L$ is a field such that
$p_i,q_i\in\Omega_\cL(L)$, then $f$ is $L$-automorphic.
\end{thm}
\begin{proof}
First, with the notation of the previous proposition take
$$
D_f^\Gamma=\sum_{i=1}^r{(p_i-q_i)}
$$
Second, consider the automorphic form
$$
\theta_{D_f^\Gamma}(z):=\theta(p_1-q_1;z)\cdots\theta(p_r-q_r;z)
$$
By definition, the zeroes and poles of it are the same as the ones
of $f$, so $\widetilde{f}(z):=f(z)/\theta_{D_f^\Gamma}(z)$ is an
analytic function.

The second claim is immediate.
\end{proof}

Therefore we have an infinite divisor on $\Omega_\cL$ for any automorphic form. Indeed, the associated infinite divisor to the form of the theorem is
$$
\Gamma\cdot\sum_{i=1}^{r}{(p_i-q_i)}
$$
As a consequence we get a well defined degree zero divisor on $\Gamma\backslash\Omega_\cL(L)=C_\Gamma(L)$.

Finally let us take into consideration $\delta\in\Gamma$ and the analytic function $\theta(p-\delta p;z)\in\cO(\Omega_\cL)^*$ for
any $p\in\Omega_\cL(K)$ (as in the previous section we assume $\Omega_\cL(K)\neq\emptyset$, if necessary after a finite extension of the base field).

\begin{thm}\label{thMeas}
 The image of $\theta(p-\delta p;z)$ by the
 morphism
 $$
 \displaystyle{\tilde{\mu}:\mathcal{O}(\Omega_\cL)^*\longrightarrow\mathscr{M}(\cL^*,\mb{Z})_{0}}
 $$
 is $\mu_\delta$. Moreover, it maps any (analytic) automorphic form to a $\Gamma$-invariant measure.
\end{thm}
\begin{proof}
In the same way we did in the previous section, we define $\cL_{p}:=\cL\cup\overline{\Gamma\cdot p}\subset {\bP^1}^*(K)$. We recall the analytic functions defined through section~\ref{poisson}.
$$
u_{\gamma p,\gamma\delta p}(z)=\frac{z-\gamma p}{z-\gamma\delta p}\in\cO(\Omega_{\cL_p})^*
$$
so
$$
\theta(p-\delta p;z)=\prod_{\gamma\in\Gamma}{u_{\gamma p,\gamma\delta p}(z)}\text{ on }\Omega_{\cL_p}.
$$
Now, theorem~\ref{FMeas} gives us the map
$$
\tilde{\mu}:\cO(\Omega_{\cL_p})^*\longrightarrow\mathscr{M}(\cL_p^*,\bZ)_0
$$
by which we map the previous functions:
$$
\tilde{\mu}(\theta(p-\delta p;z))=\tilde{\mu}\left(\prod_{\gamma\in\Gamma}{u_{\gamma p,\gamma\delta p}(z)}\right)=\sum_{\gamma\in\Gamma}{\tilde{\mu}(u_{\gamma p,\gamma\delta p}(z))}=\sum_{\gamma\in\Gamma}{\mu_{\gamma\delta p^*,\gamma p^*}}
$$
where the second equality is due to the fact that $\tilde{\mu}$ commutes with limits. Thus, applying results of previous section, this
measure coincides with $-\mu_{\delta^{-1}}=\mu_\delta$ when is restricted to $\cL$, so the image
of $\theta(p-\delta p;z)$ by $\tilde{\mu}$ as an analytic function on $\cL$ is $\mu_\delta$.

For the second claim, let us take an analytic $K$-automorphic form $f\in\cO(\Omega_\cL)^*$. To be automorphic means that for any $\gamma\in\Gamma$, $\gamma\cdot f=c_\gamma f$ for some $c_\gamma\in K^*$. Therefore, applying the $\Gamma$-equivariance of $\tilde{\mu}$ -recall the third part of proposition~\ref{FMeas2} and the $\Gamma$-invariance of $\cL$- we get
$$
\gamma\cdot\tilde{\mu}(f)=\tilde{\mu}(\gamma\cdot f)=\tilde{\mu}(c_\gamma f)=\tilde{\mu}(c_\gamma)+\tilde{\mu}(f)=\tilde{\mu}(f)
$$
Finally, since we can apply this reasoning for any field $K$, this is true for all automorphic forms.
\end{proof}

\begin{cor}
If $f\in\cO(\Omega_\cL)^*$ is an automorphic form, there exists a $\delta\in\Gamma$ such that $\tilde{\mu}(f)=\mu_\delta$.
\end{cor}
\begin{proof}
By the previous theorem we have $\tilde{\mu}(f)\in \mathscr{M}(\cL^*,\mb{Z})_{0}^\Gamma$ and by the isomorphism $\Gamma^{ab}\cong\mathscr{M}(\cL^*,\mb{Z})_{0}^\Gamma$ (theorem~\ref{Niso}) there exists a $\delta\in\Gamma$ such that $\tilde{\mu}(f)=\mu_\delta$.
\end{proof}

We give a new proof of the complete result cited above \cite[Ch.~2~(3.2)]{GvdP80}.

\begin{cor}
All analytic automorphic forms are products of the theta functions of the form $\theta(p-\delta p;z)$ by constants.
\end{cor}
\begin{proof}
This is due to the first claim of the theorem, to the previous corollary and to the fact that the kernel of $\tilde{\mu}$ are the constants.
\end{proof}

We finish this section extending the corollary~\ref{exh}.

\begin{cor}
We have a commutative rectangle of short exact sequences with sections for each $z_0\in\Omega_\cL$
$$
\xymatrix@C=1.5pc@R=1.5pc{
0\ar[rr]&&K^*\ar[rr]&&\ds{\cO(\Omega_\cL)^*}\ar[rr]^(.45){\ds{\tilde{\mu}}}&&\ds{\cM(\cL^*,\bZ)_0}\ar[rr]\ar@/^1pc/[ll]^{\ds{\cI_{z_0}}}&&0\\
&&&&&&&&\\
0\ar[rr]&&K^*\ar[rr]\ar@{=}[uu]&&\ds{\cA_\Gamma\cap\cO(\Omega_\cL)^*}\ar[rr]^(.45){\ds{\tilde{\mu}}}\ar@{^(->}[uu]&&\ds{\cM(\cL^*,\bZ)_0^\Gamma}\ar[rr]\ar@/^1pc/[ll]^{\ds{\cI_{z_0}}}\ar@{^(->}[uu]&&0\\
}
$$
and with (non-canonical) isomorphisms ${\cO(\Omega_\cL)^*\cong K^*\times\cM(\cL^*,\bZ)_0}$ and $${\cA_\Gamma\cap\cO(\Omega_\cL)^*\cong K^*\times\cM(\cL^*,\bZ)_0^\Gamma\cong K^*\times\Gamma^{ab}}.$$
\end{cor}
\begin{proof}
We already had the first exact sequence with its section and the corresponding isomorphism by corollary~\ref{exh}. By theorem~\ref{thMeas} the map $\tilde{\mu}$ restricts to analytic automorphic forms and $\Gamma$-invariant measures. The same occurs to the section due to proposition~\ref{autMI}, so we get the exhaustivity and the isomorphism (using for the last part the theorem~\ref{Niso}).
\end{proof}




\section{The Jacobian and the Abel-Jacobi map}

Using the results of the previous sections, we show that the
jacobian and the Abel-Jacobi map of a Mumford curve can be described
in terms of multiplicative integrals. The main theorem generalize
the result of Dasgupta \cite[Thm.~2.5]{Das05} to any field complete
with respect to a non-archimedean absolute value. We give, however,
a distinct and independent proof. 

Let $\Gamma\subset \PGL_2(K)$ be a Schottky group, let $\cL:=\cL_\Gamma\subset{\bP^1}^*(K)$ be its set of limit points
and let $\Omega_\cL$ be the functor which associates to any complete extension of fields $L|K$ the set of
points $\Omega_\cL(L)$.

Now we are going to do the following steps aimed at building an abelian variety associated to $\Gamma$ in a natural way.

Take into consideration the short exact sequence
$$
0\lra\bZ[\Omega_\cL]_0\lra\bZ[\Omega_\cL]\lra\bZ\lra0
$$
where the first arrow is the injection of divisors of degree zero
and the second arrow is the degree map. Since $\Gamma$ acts on
$\Omega_\cL$, we can take the associated long homology sequence, and
in particular, the morphism
$$
\xymatrix@R=.1pc{
\Gamma^{ab}=H_1(\Gamma,\bZ)\ar[r]& H_0(\Gamma,\bZ[\Omega_\cL]_0)={\bZ[\Omega_\cL]_0}_\Gamma\\
\gamma\ar@{|->}[r]&\gamma p-p
}$$
independent of the chosen $p\in\Omega_\cL$.

Since the map $\displaystyle{\mint_\bullet{d}:\bZ[\Omega_\cL]_0\lra\mathrm{Hom}(\mathscr{M}(\cL^*,\mb{Z})_{0},\bG_{m,K})}$ is $\Gamma$-equivariant we may take $\Gamma$-coinvariants so we obtain
$$
\mint_\bullet{d}:{\bZ[\Omega_\cL]_0}_\Gamma\lra \mathrm{Hom}(\mathscr{M}(\cL^*,\mb{Z})_{0},\bG_{m,K})_\Gamma=\mathrm{Hom}(\mathscr{M}(\cL^*,\mb{Z})_{0}^\Gamma,\bG_{m,K})
$$
and after composing with the connecting map above we get
$$
\xymatrix@R=.1pc{
\Gamma^{ab}\ar[r]&{\bZ[\Omega_\cL]_0}_\Gamma\ar[r]&\mathrm{Hom}(\mathscr{M}(\cL^*,\mb{Z})_{0}^\Gamma,\bG_{m,K})\\
\gamma\ar[rr]&&\displaystyle{\mint_{\gamma p-p}{d}:\mu\mapsto\mint_{\gamma p-p}{d\mu}}
}$$

Note that if $\cL\neq{\bP^1}^*(K)$, then we may take $p\in\Omega_\cL(K)$. This occurs unless $K$ is local and $\Gamma$ is cocompact, in which case, since we may take $p$ in any complete extension $L|K$, we also have $\displaystyle{\mint_{\gamma p-p}{d\mu}\in K^*}$.

Therefore, by theorem~\ref{Niso} we get a pairing
$$
\xymatrix@R=0pc{
\Gamma^{ab}\times\Gamma^{ab}\ar[rr]^{\displaystyle{\mint_{\cL^*}{(\ ,\ )}}}&& K^*\\
(\gamma,\gamma')\ar@{|->}[rr]&&\displaystyle{\mint_{\cL^*}{(\gamma,\gamma')}:=\mint_{\gamma p-p}{d\mu_{\gamma'}}}
}$$
such that, by theorem~\ref{Niso} and corollary~\ref{absI},
$$
\displaystyle{\mathit{v}_K\left(\mint_{\cL^*}(\gamma,\gamma')\right)=(\gamma,\gamma')_\Gamma}
$$
for all $\gamma,\gamma'\in\Gamma$. This equality implies that the pairing is positive definite. Further, using corollary~\ref{APR} we get
$$
\mint_{\cL^*}{(\gamma,\gamma')}=\mint_{\gamma p-p}{d\mu_{\gamma'}}=\mint_{\gamma' p-p}{d\mu_{\gamma}}=\mint_{\cL^*}{(\gamma',\gamma)}
$$
so the pairing is symmetric too.

Summarizing, we have a morphism
$$
\xymatrix@R=0pc{
H_1(\Gamma,\bZ)\ar[rr]^(.35){\ds{\mint_\bullet{d}}}&&\mathrm{Hom}(\mathscr{M}(\cL^*,\mb{Z})_{0}^\Gamma,\bG_{m,K}):=T\\
\gamma\ar@{|->}[rr]&&\displaystyle{\mint_{\gamma p-p}{d}:\mu\mapsto\mint_{\gamma p-p}{d\mu}}
}$$
which is an isomorphism between $H_1(\Gamma,\bZ)\cong\Gamma^{ab}$ and its image $\Lambda$, so that it is a free group of rank $g=\rank(\Gamma)=\mbox{genus}(C_\Gamma)$.\\
Note that, as a consequence of having $\displaystyle{\mint_{\gamma p-p}{d\mu}\in K^*}$ for any $\gamma\in\Gamma$, we get
$$\Lambda\subset T(K)=\mathrm{Hom}(\mathscr{M}(\cL^*,\mb{Z})_{0}^\Gamma,K^*)\cong\mathrm{Hom}(\Gamma^{ab},K^*)\cong (K^*)^g$$
Let us consider now the valuation map applied to this:
$$
\xymatrix@R=0pc{
\displaystyle{(K^*)^g}\ar[r]^{\displaystyle{\mathit{v}_K}}&\bR^g\\
(a_1,\dots,a_g)\ar@{|->}[r]&(\mathit{v}_K(a_1),\dots,\mathit{v}_K(a_g))
}$$

\begin{lem}
The subgroup $\mathit{v}_K(\Lambda)\subset \bR^g$ is a lattice.
\end{lem}

\begin{proof}

Observe the way in which the isomorphism $T(K)\cong(K^*)^g$ works:
$$
\xymatrix@R=1pc{
T(K)\ar[r]&\displaystyle{(K^*)^g}\\
\ds{\mint}\ar@{|->}[r]&\displaystyle{\left(\mint(\mu_{\gamma_1}),\dots,\mint(\mu_{\gamma_g})\right)}
}
$$
where $\ds{\gamma_1,\dots,\gamma_g}$ is a fixed basis of the free group $\Gamma$.
In particular $\Lambda$ seen inside of $\ds{(K^*)^g}$ is the multiplicative subgroup
$$\ds{\left\{\left(\mint_{\cL^*}(\gamma,\gamma_1),\dots,\mint_{\cL^*}(\gamma,\gamma_g)\right)\right\}_{\gamma\in\Gamma}}.$$
After applying the valuation map to this we get
$$
\left(\mathit{v}_K\left(\mint_{\cL^*}(\gamma,\gamma_1)\right),\dots,\mathit{v}_K\left(\mint_{\cL^*}(\gamma,\gamma_g)\right)\right)=
\left((\gamma,\gamma_1)_\Gamma,\dots,(\gamma,\gamma_g)_\Gamma\right)
$$
that is the image of the map
$$
\xymatrix@R=.1pc{
\Gamma^{ab}\ar[rr]&&\mathrm{Hom}(\Gamma^{ab},\bR)\cong\bR^g\\
\gamma\ar@{|->}[rr]&&\displaystyle{\mathit{v}_K\left(\mint_{\cL^*}(\gamma,\cdot)\right)}
}
$$
As $\Gamma$ is generated by $\ds{\gamma_1,\dots,\gamma_g}$, $\displaystyle{\mathit{v}_K(\Lambda)\subset\bR^g}$ is the subgroup generated by
$$
\left((\gamma_1,\gamma_1)_\Gamma,\dots,(\gamma_1,\gamma_g)_\Gamma\right),\dots,\left((\gamma_g,\gamma_1)_\Gamma,\dots,(\gamma_g,\gamma_g)_\Gamma\right)
$$
which, due to the fact that $(\ ,\ )_\Gamma$ is positive definite, is isomorphic to $\bZ^g$ so it is a discrete subgroup, and has maximal rank. Therefore it is a lattice.
\end{proof}

\begin{thm}
The quotient
$$
T^{an}/\Lambda=\mathrm{Hom}(\mathscr{M}(\cL^*,\mb{Z})_{0}^\Gamma,\bG_{m,K})^{an}/\Lambda
$$
is an abelian variety.
\end{thm}
\begin{proof}
By \cite[6.4,~p.~171]{FvdP04} we obtain this quotient
is a proper analytic torus.\\
Note that by means of the identification $\mathscr{M}(\cL^*,\mb{Z})_{0}^\Gamma\cong\Gamma^{ab}$, the torus is defined by the map $\ds{\gamma\mapsto\mint_{\cL^*}{(\gamma,\cdot)}}$.\\
This torus has principal polarization
$$
\xymatrix@R=.1pc{
\Gamma^{ab}\cong\Lambda\ar[rr]^(.3){\ds{\mu^*}}&&X(T)=\mathrm{Hom}_{K-\mathcal{G}rp}(T,\bG_{m,K})\cong\Gamma^{ab}\\
\gamma\ar@{|->}[rr]&&\ds{\mu^*(\gamma):\mint\longmapsto\mint(\mu_\gamma)}
}$$
since
$$\mu^*(\gamma')\left(\mint_{\gamma p-p}{d}\right)=\mint_{\gamma p-p}{d\mu_{\gamma'}}=\mint_{\cL^*}(\gamma,\gamma')
$$
and this form is symmetric and positive definite. Thus, we conclude that $T^{an}/\Lambda$ is an abelian variety (\cite[Thm.~6.6.1]{FvdP04}).
\end{proof}

Our next goal is to get an isomorphism of abelian varieties
$$
Jac(C_\Gamma)\longrightarrow T/\Lambda
$$

In order to show this we are going to use the well known isomorphism $Jac(C_\Gamma)\cong \Div_0(C_\Gamma)/\Prin(C_\Gamma)$. First we will build for any extension of complete fields $L|K$ a map
$$
\Div_0(C_\Gamma)(L)\longrightarrow(T/\Lambda)(L)
$$

Then, let us fix any extension of complete fields $L|K$. Take a divisor $D\in\Div_0(C_\Gamma)(L)$. It can be written as
$$
D=\sum_{p\in C_\Gamma(\bC_L)}{n_pp}\qquad\text{  verifying  }\qquad D^\sigma=D\ \forall\sigma\in Gal(\bC_L|L)
$$
and there exists a finite extension $L'|L$ such that
$\Supp(D)\subset C_\Gamma(L')$ so $D\in\Div_0(C_\Gamma(L'))$. Now,
there is a finite field extension $\tilde{L}|L'$ such that
$G_\Gamma$ has no loops (in fact, this is true for almost any
extension up to a finite number), so by corollary~\ref{corexh}, the
map $\Omega_\cL(\tilde{L})\longrightarrow C_\Gamma(\tilde{L})$ is
surjective and thus, the maps
$$
\Gamma\backslash\Omega_\cL(\tilde{L})\longrightarrow C_\Gamma(\tilde{L})\qquad\text{ and }\qquad
\Gamma\backslash\bZ[\Omega_\cL(\tilde{L})]_0\longrightarrow \Div_0(C_\Gamma(\tilde{L}))
$$
are isomorphisms. Thus, we got a finite extension $\tilde{L}|L$ such that there is a divisor $\tilde{D}\in \bZ[\Omega_\cL(\tilde{L})]_0$ satisfying $\pi_\Gamma(\tilde{D})=D$, that is
$$
\forall\sigma\in Gal(\tilde{L}|L)\ \exists\ \gamma_\sigma\in\Gamma\text{ such that }\tilde{D}^\sigma=\gamma_\sigma\tilde{D}.
$$

The continuous arrows of the diagram

$$
\xymatrix@R=.1pc{
{\bZ[\Omega_\cL(\tilde{L})]_0}_\Gamma\ar@{->>}[dddd]\ar[rr]^(.35){\ds{\mint_\bullet{d}}}&&\mathrm{Hom}(\mathscr{M}(\cL^*,\mb{Z})_{0}^\Gamma,\tilde{L}^*)=T(\tilde{L})\ar@{->>}@/^1pc/[rdddd]&\\
\quad \tilde{D}\ar@{|->}[rr]&&\displaystyle{\mint_{\tilde{D}}{d}:\mu\longmapsto\mint_{\tilde{D}}{d\mu}}\qquad&\\
&&&\\
&&&\\
\Div_0(C_\Gamma(\tilde{L}))\ar@{.>}[rrr]&&& T(\tilde{L})/\Lambda\\
\quad
\pi_\Gamma(\tilde{D})=D\ar@{|.>}[rrr]&&&\displaystyle{\mint_{D}{d}}
}$$ factorize by the dots arrow, since
$\displaystyle{\Gamma\backslash{\bZ[\Omega_\cL(\tilde{L})]_0}\cong{\bZ[\Omega_\cL(\tilde{L})]_0}_\Gamma/H_1(\Gamma,\bZ)}$.

We can finish the construction of the map we told above thanks to the following result.

\begin{lem}\label{Gal} Given a finite extension $\tilde{L}|L$ and any $\tilde{D}\in \bZ[\Omega_\cL(\tilde{L})]_0$ satisfying
$$
\forall\sigma\in Gal(\tilde{L}|L)\ \exists\ \gamma_\sigma\in\Gamma\text{ such that }\tilde{D}^\sigma=\gamma_\sigma\tilde{D}.
$$
We have
$$
\left(\mint_{\tilde{D}}{d}\right)^\sigma\equiv\mint_{\tilde{D}}{d}\quad  (\mbox{mod }\Lambda)
$$
\end{lem}
\begin{proof}
We just have to note how it is defined the integral, as a limit of products of the function $f_{D}$. This is integrated over $\cL^*$, set of $K$-rational points, so invariant by $\sigma$. Therefore, for any $\mu\in\mathscr{M}(\cL^*,\mb{Z})_{0}$ we have
$$
\left(\mint_{\tilde{D}}{d\mu}\right)^\sigma\left(\mint_{\tilde{D}}{d\mu}\right)^{-1}=\left(\mint_{\cL^*}{f_{\tilde{D}}d\mu}\right)^\sigma\left(\mint_{\tilde{D}}{d\mu}\right)^{-1}=
$$
$$
=\mint_{\cL^*}{f_{\tilde{D}^\sigma}d\mu}\left(\mint_{\tilde{D}}{d\mu}\right)^{-1}=\mint_{\cL^*}{f_{\gamma_\sigma \tilde{D}}d\mu}\left(\mint_{\cL^*}{f_{\tilde{D}}d\mu}\right)^{-1}=
$$
$$
=\mint_{\gamma_\sigma \tilde{D}-\tilde{D}}{d\mu}
$$
independent of $\mu$. Finally $\displaystyle{\mint_{\gamma_\sigma \tilde{D}-\tilde{D}}{d}\in\Lambda}$.
\end{proof}

\begin{cor}
Under the same hypothesis we get
$$
\mint_{\tilde{D}}{d}\in(T/\Lambda)(L)
$$
\end{cor}
\begin{proof}
It is immediate.
\end{proof}

Therefore, for $D\in\Div_0(C_\Gamma)(L)$ we have build a well defined element
$$
\mint_{D}{d}\in(T/\Lambda)(L),
$$
so we get the map
$$
\Div_0(C_\Gamma)(L)\longrightarrow(T/\Lambda)(L)
$$
Next we want to show its exhaustivity and compute its kernel. The next result is crucial to move forward:

\begin{lem}
Let $\tilde{D}$ be a degree zero divisor on $\Omega_\cL$ which can be represented as $\sum_{i=1}^r{(p_i-q_i)}$ and let us define the automorphic form
$$
\theta_{\tilde{D}}(z):=\theta(p_1-q_1;z)\cdots\theta(p_r-q_r;z)
$$
Then its factor of automorphy is given by
$$
c_{\theta_{\tilde{D}}}(\gamma)=\mint_{\tilde{D}}{d\mu_\gamma}\quad\forall\ \gamma\in\Gamma
$$
\end{lem}
\begin{proof}
On one hand we have
$$
c_{\theta_{\tilde{D}}}(\gamma)=\frac{\theta_{\tilde{D}}(z)}{\theta_{\tilde{D}}(\gamma z)}=\frac{\theta(p_1-q_1;z)\cdots\theta(p_r-q_r;z)}{\theta(p_1-q_1;\gamma z)\cdots\theta(p_r-q_r;\gamma z)}=
$$
$$
=\frac{\theta(z-\gamma z;p_1)\cdots\theta(z-\gamma z;p_r)}{\theta(z-\gamma z;q_1)\cdots\theta(z-\gamma z;q_r)}
$$
where the last equality is due to the straightforward symmetry of theta functions. On the other hand, applying the theorem~\ref{thMeas} and the extended Poisson formula (corollary~\ref{EPF}) we have
$$
\mint_{\tilde{D}}{d\mu_{\gamma}}=\mint_{\tilde{D}}{d\tilde{\mu}(\theta(z_0-\gamma z_0;\cdot))}=\prod_{i=1}^r{\frac{\theta(z_0-\gamma z_0;p_i)}{\theta(z_0-\gamma z_0;q_i)}}
$$
Since the right sides of two last chains of equalities are independent of $z$ and $z_0$ respectively, they are equal.
\end{proof}

\begin{lem}
If $h\in\cO(\Omega_\cL)^*$ is an (analytic) automorphic form, its factor of automorphy $c_{h}$ belongs to $\Lambda$.
\end{lem}
\begin{proof}
First, recall by the final results of the previous section that $\tilde{\mu}(h)=\mu_\delta$ for some $\delta\in\Gamma$. Next, let us compute its automorphic form on a $\gamma\in\Gamma$ by means of applying the Poisson formula:
$$
c_{h}(\gamma)=\frac{h(z)}{h(\gamma z)}=\mint_{z-\gamma z}{d\tilde{\mu}(h)}=\mint_{z-\gamma z}{d\mu_\delta}=\mint_{z-\delta z}{d\mu_\gamma}=\mint_{z-\delta z}{d}(\mu_\gamma)
$$
Finally, $\displaystyle{\mint_{z-\delta z}{d}}$ belongs to $\Lambda$ by definition.
\end{proof}

\begin{prop} Given an automorphic form $h\in\cA_\Gamma$ with factor of automorphy $c_h$, there is a finite divisor $\tilde{D}_h$ on $\Omega_{\cL}$ such that the infinite divisor of $h$ on $\Omega_\cL$ is $D_h=\Gamma\cdot\tilde{D}_h$ and
$$
c_h\equiv c_{\theta_{\tilde{D}_h}}=\mint_{\tilde{D}_h}{d}\  (\mbox{mod }\Lambda)
$$
\end{prop}
\begin{proof}
We take $\tilde{D}_h$ a finite divisor as in theorem~\ref{AutTh}, such that $D_h=\Gamma\cdot\tilde{D}_h$ and $h(z)=h'(z)\theta_{\tilde{D}_h}(z)$ with $h'(z)$ analytic. Then, by the previous lemmas we have
$$
c_h=c_{h'}c_{\theta_{\tilde{D}_h}}\equiv c_{\theta_{\tilde{D}_h}}=\mint_{\tilde{D}_h}{d}\  (\mbox{mod }\Lambda)
$$
\end{proof}

\begin{cor}
The map $\displaystyle{\Div_0(C_\Gamma)(L)\longrightarrow (T/\Lambda)(L)}$ factorize by the principal divisors of $C_\Gamma$ and the resulting map
$$
\Div_0(C_\Gamma)(L)/\Prin(C_\Gamma)(L)\longrightarrow (T/\Lambda)(L)
$$
is injective.
\end{cor}
\begin{proof}
First we will show that the map factorize by the principal divisors.\\
A divisor of $Div_0(C_\Gamma)(L)$ is principal when it is the divisor of a meromorphic function on $C_\Gamma$, that is
 a $\Gamma$-invariant meromorphic function on $\Omega_\cL$. Let $D_h$ and $h$ be such a divisor and such a function
 respectively. Since $h$ is $\Gamma$-invariant, its factor of automorphy is constant equal to 1. Therefore, by the proposition we get
 $$
 \mint_{\tilde{D_h}}{d}\equiv1\  (\mbox{mod }\Lambda)
 $$
 with $D_h=\Gamma\tilde{D}_h$, and so the factorization.

 Next we want to prove the injectivity of this factorized map. Take now a $D\in\Div_0(C_\Gamma)(L)$ such that
 $$
 \mint_{\tilde{D}}{d}\in\Lambda\text{ so there exists a }\delta\in\Gamma \text{ satisfying }\mint_{\tilde{D}}{d}=\mint_{\delta p-p}{d}
 $$
 where $D=\Gamma\tilde{D}$ with $\tilde{D}$ divisor on $\Omega_\cL$ and $p\in\Omega_\cL$. Now, as above we can build the automorphic form $\theta_{\tilde{D}}$, which has associated infinite divisor $D$.
Further, let us consider the analytic function $\theta(\delta p-p;z)$, and write $c_{\tilde{D}}$ and $c_\delta$ for
 the factors of automorphy of the two last automorphic forms. Observe that
 $$
 c_{\tilde{D}}(\gamma)=\mint_{\tilde{D}}{d\mu_{\gamma}}=\mint_{\delta p-p}{d\mu_\gamma}=c_\delta(\gamma).
 $$
Therefore, $D$ is the divisor associated to the function
$\theta_{\tilde{D}}(z)/\theta(\delta p-p;z)$, which is
$\Gamma$-invariant, so it is principal and thus the injectivity is
 done.

\end{proof}

\begin{prop}
There is an isomorphism
$$
\left(\Div_0(C_\Gamma)/\Prin(C_\Gamma)\right)(L)\longrightarrow (T/\Lambda)(L)
$$
\end{prop}
\begin{proof}
Let us check first that this map is well defined.\\
Consider a divisor $D$ in $\left(\Div_0(C_\Gamma)/\Prin(C_\Gamma)\right)(L)$. Then, there is a Galois extension $\tilde{L}|L$ and a divisor $\tilde{D}\in \Div_0(C_\Gamma)(\tilde{L})$ such that $\tilde{D}^\sigma-\tilde{D}\in Prin(C_\Gamma)(\tilde{L})$ for all $\sigma\in Gal(\tilde{L}|L)$. This implies that
$$
\mint_{\tilde{D}^\sigma-\tilde{D}}{d}=0_{T/\Lambda}\in(T/\Lambda)(\tilde{L})
$$
and so, as in the proof of the lemma~\ref{Gal} we get the next equalities in $(T/\Lambda)(\tilde{L})$:
$$
\left(\mint_{\tilde{D}}{d}\right)^\sigma=\mint_{\tilde{D}^\sigma}{d}=\mint_{\tilde{D}}{d}\quad\forall\sigma\in Gal(\tilde{L}|L)
$$
Therefore $\displaystyle{\mint_D{d}\in(T/\Lambda)(L)}$ and we get the morphism
$$
\left(\Div_0(C_\Gamma)/\Prin(C_\Gamma)\right)(L)\longrightarrow (T/\Lambda)(L)
$$
which is injective by the previous corollary.

Next we have to prove its exhaustivity. An element $\Xi\in
(T/\Lambda)(L)$ can be seen in $T(\tilde{L})/\Lambda$, satisfying
$\Xi^\sigma=\Xi$ for each $\sigma\in Gal(\tilde{L}|L)$, where
$\tilde{L}|L$ is a Galois extension. This element is the class of a
$\xi\in T(\tilde{L})\cong \mathrm{Hom}(\Gamma^{ab},\tilde{L}^*)$
such that
$$
\xi^\sigma\equiv\xi\ (\mbox{mod }\Lambda)\quad\text{ for each }\quad\sigma\in Gal(\tilde{L}|L),
$$
which in turn is the factor of automorphy $c_h$ of an automorphic form $h\in\cA_\Gamma$, by the proposition~\ref{autexh}. Now, by the proposition above we have
$$
\mint_{\tilde{D}_h}{d}\equiv c_h=\xi\  (\mbox{mod }\Lambda) \qquad\text{ and so }\qquad\mint_{D_h}{d}=\Xi
$$
with $D_h\in \Div_0(C_\Gamma)(\tilde{L})$. By the hypothesis
$$
\left(\mint_{D_h}{d}\right)^\sigma=\mint_{D_h}{d}\qquad\text{ so }\qquad\mint_{D_h^\sigma-D_h}{d}=0_{T/\Lambda}
$$
what, due to the injectivity of the map, gives that $D_h^\sigma-D_h\in\Prin(C_\Gamma)(\tilde{L})$. But this for each $\sigma\in Gal(\tilde{L}|L)$ implies that $D_h\in \left(\Div_0(C_\Gamma)/\Prin(C_\Gamma)\right)(L)$.
\end{proof}

Now we are ready to prove the main theorem, which generalizes to any complete field with respect to a non-trivial non-archimedean valuation \cite[Thm.~2.5]{Das05}:

\begin{thm}\label{mainT}
There is an isomorphism over $K$ of abelian varieties
$$
Jac(C_\Gamma)\longrightarrow T/\Lambda
$$
\end{thm}
\begin{proof}
First, as we told above, we recall the isomorphism
$$
Jac(C_\Gamma)\cong \Div_0(C_\Gamma)/\Prin(C_\Gamma)
$$
Second, we have built an analytic morphism of abelian varieties
$$
\Div_0(C_\Gamma)/\Prin(C_\Gamma)\longrightarrow T/\Lambda
$$
Since they are proper, by GAGA it is an algebraic morphism, and it also respects the group operations, so it is a morphism of abelian varieties. Further, it induces an isomorphism in the corresponding $L$-points for any extension of complete fields $L|K$, and this implies that it is an isomorphism.
\end{proof}

\begin{cor}
The abelian variety $T^{an}/\Lambda$ is the Jacobian of the curve $C_\Gamma$ and the Abel Jacobi map is given, after having fixed some point $z_0\in C_\Gamma$, by
$$
\xymatrix@R=.1pc{
C_\Gamma\ar[rr]^(.25){\displaystyle{i_{z_0}}}&&\displaystyle{\mathrm{Hom}(\mathscr{M}(\cL^*,\mb{Z})_{0}^\Gamma,\bG_{m,K})/\Lambda}\\
z\ar@{|->}[rr]&&\displaystyle{\mint_{z-z_0}{d}}
}
$$
\end{cor}

\end{document}